\documentclass[12pt]{article}

\usepackage[T1]{fontenc}
\usepackage[utf8]{inputenc}

\usepackage[french,english]{babel}

\usepackage{amsmath}
\usepackage{amsthm}
\usepackage{graphicx}
\usepackage{amsfonts}
\usepackage{amssymb}
\usepackage{url}
\usepackage{verbatim}

\providecommand{\keywords}[1]{\textit{Keywords --} #1}
\providecommand{\motscles}[1]{\textit{Mots-Clés --} #1}
\providecommand{\MSC}[1]{\textit{MSC-2020 --} #1}

\providecommand{\U}[1]{\protect\rule{.1in}{.1in}}

\newtheorem{Theorem}{Theorem}
\newtheorem{theorem}{Theorem}[section]
\newtheorem{corollary}[theorem]{Corollary}
\newtheorem{definition}[theorem]{Definition}
\newtheorem{lemma}[theorem]{Lemma}

\newtheorem{proposition}[theorem]{Proposition}
\newtheorem{remark}[theorem]{Remark}

\newcommand{\N}{\mathbb N}

\newcommand{\R}{\mathbb R}
\newcommand{\Z}{\mathbb Z}

\DeclareMathOperator{\Bad}{Bad}
\DeclareMathOperator{\SL}{SL}
\DeclareMathOperator{\GL}{GL}
\DeclareMathOperator{\M}{M}
\DeclareMathOperator{\SO}{SO}
\DeclareMathOperator{\dd}{d}

\begin{document}
\title{Lévy-Khintchin Theorem for best simultaneous Diophantine approximations}
\author{Yitwah Cheung, Nicolas Chevallier}
\maketitle

\begin{abstract}
We extend two results about the ordinary continued fraction expansion to best simultaneous Diophantine approximations of vectors or matrices. 
The first result is the L\'evy-Khintchin theorem about the almost sure growth rate of the denominators of the convergents. The second result is a theorem of Doeblin and of Bosma, Jager and Wiedijk about the almost sure limit distribution of the sequence of products  $q_n\dd(q_n\theta,\Z)$ where the $q_n$'s are the denominators of the convergents associated with the real number  $\theta$ by the ordinary continued fraction algorithm. Besides these two main results, we show that when $d\geq 2$, for almost all vectors $\theta\in\R^d$, $\liminf_{n\rightarrow\infty} q_{n+k}\dd(q_n\theta,\Z^d)^d=0$ for all positive integers $k$, where $(q_n)_{n\in\N}$ is the sequence of best approximation denominators of $\theta$. 
\end{abstract}

\selectlanguage{french}

\begin{abstract}
Nous étendons deux résultats sur le développement en fraction continue ordinaire aux meilleures approximations diophantiennes simultanées  de vecteurs ou de matrices.
Le premier résultat est le théorème de Lévy-Kintchin sur le taux de croissance presque sûr des dénominateurs des réduites. Le second est un théorème de Doeblin et de Bosma, Jager et Wiedijk sur la distribution presque sûre de la suite des produits $ q_n \dd (q_n \theta, \Z) $ où les $ q_n $  sont les dénominateurs des réduites associées au nombre réel $ \theta $ par l'algorithme des fractions continues ordinaires. En dehors de ces deux  résultats principaux, nous montrons que lorsque $ d \geq 2 $, pour presque tous les vecteurs $ \theta \in \R ^ d $, $ \liminf_ {n \rightarrow \infty} q_{n+k} \dd (q_n \theta, \Z^d)^{d} = 0 $ pour tous les entiers positifs $ k $, où $ (q_n) _ {n \in \N} $ est la suite des dénominateurs des meilleures approximations de $ \theta $.
\end{abstract}

\selectlanguage{english}

\noindent\keywords{Diophantine approximations, Ergodic theory, diagonal flow}\\
\motscles{approximations diophantienne, theorie ergodique, flot diagonal}\\
\MSC{11J13,11J83,11K55,37A17}

\section{Introduction}

In 1936, Aleksandr Khintchin showed that there exists a constant $\gamma$ such that the denominators $(q_{n})_{n\geq 0}$
of the convergents of the continued fraction expansions of almost all real
numbers $\theta$ satisfy
\[
\lim_{n\rightarrow\infty}q_{n}^{1/n}=\gamma
\]
(see \cite{Khin1936}). Soon afterward, in
\cite{Levy}, in the footnote page 289, Paul L\'{e}vy gave the explicit value
of the constant,
\[
\gamma=\exp\frac{\pi^{2}}{12\ln 2}.
\]

In \cite{Doe}, published in 1940, among many other results about the ordinary continued fraction expansion, Wolfgang Doeblin stated the following result: 
 for almost all real numbers $\theta,$
\[
\lim_{n\rightarrow\infty}\frac{1}{n}\operatorname*{card}\{0\leq k<n:q_{k}
\dd(q_{k}\theta,\mathbb{Z})\leq t\}=g(t)
\]
for all $t\in[0,1]$, where 
\[
g(t)=\int_0^t\frac{1}{2\ln 2}\frac{1-\left\vert
1-2s\right\vert}{s}\,ds.
\] 
Doeblin only sketched the proof of this result and it is difficult to reconstitute a complete proof from his paper. In \cite{Ios}, Iosifescu gave a very interesting account of Doeblin's paper.  Doeblin's outstanding article was not immediately noticed, so the above result was first called the Lenstra conjecture, now it is often called the Doeblin-Lenstra conjecture. A complete proof of the Doeblin-Lenstra conjecture was given in 1983 by Wieb Bosma, Hendrik Jager and Freek Wiedijk, see  \cite{BoJaWi}. Later, Jager proved variants of this result, in
particular with the quantity $q_{k+1}\dd(q_{k}\theta,\mathbb{Z})$ instead of
$q_{k}\dd(q_{k}\theta,\mathbb{Z})$ (see \cite{IoKr} for more results).

The aim of our work is to extend to the best simultaneous Diophantine approximations, both the Lévy-Khintchin result and the Jager version of the Doeblin, Bosma, Jager and Wiedijk result.
We choose the Jager version because one of the striking difference between one-dimensional and multidimensional Diophantine approximations is that the classic one-dimensional lower bound $q_{k+1}\dd(q_{k}\theta,\mathbb{Z})\geq 1/2$  no longer holds in the multidimensional setting; this explains for example that the only singular vectors are the rational numbers in dimension one.

Let $d$ and $c$ be two positive integers. Suppose $\mathbb{R}^{d}$ and
$\mathbb{R}^{c}$ are endowed with the standard Euclidean norms $\left\Vert
.\right\Vert _{\mathbb{R}^{d}}$ and $\left\Vert .\right\Vert _{\mathbb{R}^{c}
}$. We prove

\begin{Theorem}
\label{thm:levy-d} There exists a constant $L_{d,c}$ such that for almost all
matrices $\theta\in\M_{d,c}(\mathbb{R})$,
\begin{align*}
\lim_{n\rightarrow\infty}\frac{1}{n}\ln\left\Vert Q_{n}(\theta)\right\Vert
_{\mathbb R^c}  &  =L_{d,c}\,,\\
\lim_{n\rightarrow\infty}\frac{1}{n}\ln(\dd(\theta Q_{n}(\theta),\mathbb{Z}
^{d}))  &  =-\frac{c}{d}L_{d,c}
\end{align*}
where $Q_{n}(\theta)\in\mathbb{Z}^{c}$, $n\geq 0$, is the sequence of best
Diophantine approximation denominators of $\theta$ associated with the norms
$\left\Vert .\right\Vert _{\mathbb{R}^{d}}$ and $\left\Vert .\right\Vert
_{\mathbb{R}^{c}}$.
\end{Theorem}

(See Section 2.1 for the definition of best Diophantine approximation denominators).

 Theorem \ref{thm:levy-d} is stated for standard Euclidean norms but it holds for all pairs of Euclidean norms on $\R^d$ and $\R^c$ with the same constant $L_{d,c}$ (see Theorem \ref{thm:Birkhoff}). We are convinced that Theorem \ref{thm:levy-d} is valid for general norms but we do not know whether the constants $L_{d,c}$ depend on the norms  (see Section \ref{sec:question}).

For a matrix $\theta$ in $\M_{d,c}(\mathbb{R)}$, let us denote $\beta_{n}
(\theta)=\left\Vert Q_{n+1}(\theta)\right\Vert ^{c} _{\mathbb{R}^{c}}\dd(\theta
Q_{n}(\theta),\mathbb{Z}^{d})^{d}$ and for a real number $a$, let us denote  $\delta_{a}$  the Dirac measure at $a$.

\begin{Theorem}
\label{thm:jager}1. There exists a probability measure $\nu_{d,c}$ on
$\mathbb{R}$ such that, for almost all matrices $\theta\in\M_{d,c}(\mathbb{R}
)$, $\nu_{d,c}$ is the weak limit of the sequence of probability measures
\[
\frac{1}{n}\sum_{k=0}^{n-1}\delta_{\beta_{k}(\theta)}.
\]
\newline 2. The support of the
measure $\nu_{d,c}$ is included in a  bounded interval, and contains $0$ provided
that $c+d\geq 3$.
\end{Theorem}

The L\'{e}vy-Khintchin result has already been extended to multidimensional
settings. For instance, for almost all $\theta$ in $\mathbb{R}^{d}$, the
denominators $(J_{n}(\theta))_{n\geq 0}$ of the Jacobi-Perron expansion of
$\theta$ satisfy $\lim_{n\rightarrow\infty}\frac{1}{n}\ln J_{n}(\theta
)=c_{d}$ for some constant $c_{d}$ (see \cite{Broise}). The common proofs of such results use ergodic theory. The one-dimensional
L\'{e}vy-Khintchin result can be proven with the Birkhoff ergodic theorem, while
the growth rate of the Jacobi-Perron denominators can be derived from the Oseledec
multiplicative ergodic theorem. In both cases, the proof depends on the
existence of an underlying dynamical system: the Gauss map or the
Jacobi-Perron map (see \cite{Schweiger} for many examples of these kinds of
maps). However, no such map associated with best Diophantine approximations is
known when $d+c\geq 3$.  One  way to circumvent this problem is
to use the left action of the diagonal flow
\[
g_{t}=
\begin{pmatrix}
e^{ct}I_{d} & 0\\
0 & e^{-dt}I_{c}
\end{pmatrix}
 \in\SL (d+c,\mathbb{R)},\,t\in\R,
\]
on the space of unimodular lattices $\mathcal{L}_{d+c}=\SL (d+c,\mathbb{R}
)/\SL (d+c,\mathbb{Z)}$.  This idea can be traced back to Hermite \cite{Her} and has been used extensively in Diophantine approximation since the work of Dani \cite{Da1,Da2}; see for instance \cite{AnGuKl,Gho,HaMa}.
In the same vein, the parametric geometry of numbers introduced recently by Schmidt and Summerer \cite{SchSu1,SchSu2} and complemented by Roy \cite{Roy1}, studies the evolution of the successive minima of a lattice under the action of the diagonal flow. This led to many  results on Diophantine exponents; see for instance \cite{Mar,Roy2}.

In \cite{Chev-Levy}, the diagonal flow is used
to prove that the sequence of best Diophantine approximation denominators of
almost all $\theta$ in $\M_{d,c}(\mathbb{R})$ satisfies
\[
\lim\sup_{n\rightarrow\infty}\frac{1}{n}\ln\left\Vert Q_{n}(\theta)\right\Vert
_{\mathbb{R}^{c}}\leq K_{d,c}
\]
for some constant $K_{d,c}$. When $c=1$, it is also possible to derive this inequality
from a theorem of W. M. Schmidt (see
\cite{Chev-Khin}).

As in the aforementioned works, the diagonal flow $(g_{t})$ is the main tool. However, unlike e.g. \cite{Chev-Levy}, we introduce an important new ingredient: a co-dimension one submanifold transverse to the flow. It should be noticed that many works on continued fraction algorithms also use transversals, we mention two such works.

Firstly, the transformation induced on a well-chosen sub-interval by an interval
exchange transformation $T$ is an interval exchange transformation $\hat T$
defined with the same number of intervals. The Rauzy-Veech continued fraction algorithm is the map  $R:T \mapsto
\hat T$ where $T$ is an interval exchange transformation defined on an interval of length $1$ and $\hat T$ is renormalized so that it is defined on an interval of length $1$ (see \cite{Rau1,Rau2,Veech}).   In \cite{Veech}, W. Veech  proved
that $R$ admits a unique absolutely continuous invariant measure up
to a scalar multiple, using a bimeasureably invertible extension $\hat R$ of the map $R$.  In turn, this extension is constructed as the first return map of a ``diagonal flow'' on a transversal. It should be noticed that at the same time, H. Masur
proposed a very similar construction in \cite{Mas}. 

Secondly, A. Haas proved analogs of the result of Doeblin, Bosma, Jager and Wiedijk and of the Lévy-Khintchin result in the setting of excursions of geodesic into the neighborhood of a cusp on a finite area hyperbolic 2-orbifold, see \cite{Haas1,Haas2}. A. Haas used the ergodicity of the geodesic flow and a transversal in the unit tangent bundle that projects on a finite union of geodesic segments of the hyperbolic plane.
 
In our case,  up to a negligible set of flow trajectories, the transversal is the set of unimodular lattices  whose  first two  minima  are equal (see
Section 3.1). The first return map of the flow in the transversal  play the role of an invertible extension of the missing Gauss map. It is crucial to observe
that the visiting times of the transversal by the flow are given by a simple
formula involving best simultaneous Diophantine approximations; see
Lemma \ref{lem:return-best}.  Making use of the Birkhoff ergodic theorem, this observation leads to a
L\'{e}vy-Khintchin result in the space of lattices and to a formula close to
the Arnoux-Nogueira interpretation of the L\'{e}vy's constant (see \cite{Arn-Nog}):
\begin{equation}
L_{d,c}=\frac{d}{\mu_{S}(S)}\int_{S}\tau~d\mu_{S}=\frac{d\times\mu
(\mathcal{L}_{d+c})}{\mu_{S}(S)} \label{constant-levy}
\end{equation}
where $\mu$ is the invariant measure in the space of lattices, $\mu_{S}$ the
invariant measure induced by the flow on the transversal $S$ and $\tau$ the
return time to $S$, see Theorem~\ref{Birkhoff} and Corollary \ref{cor:levy-d}.

The second step of the proof of Theorem~\ref{thm:levy-d} consists in converting an
almost sure result in the space of lattices $\mathcal{L}_{d+c}$ into an almost sure result in
$\M_{d,c}(\mathbb{R})$. To achieve this goal, we prove a general result,
Theorem \ref{thm:reduction}, which might be of independent interest. At
first sight, this result might appear as an easy consequence of the following standard fact: the
set of lattices associated with the matrices $\theta$ in $\M_{d,c}
(\mathbb{R)}$, is the expanding direction of the flow $g_t$. However, an example
shows that Theorem \ref{thm:reduction} depends on some properties of the transversal; see
Section \ref{sect:reduction}.

When $d=1$ or $2$ and $c=1$, the submanifold $S$ and  the measure $\mu_{S}$ can be entirely calculated 
(see Section \ref{sec:S}). When $d=c=1$, thanks to  Siegel formula
giving the volume of the modular space $\SL(2,\R)/\SL(2,\Z)$,
computing  L\'evy's constant $L_{1,1}=\ln \gamma$ is easy; it is even
possible to determine the first return map to the transversal. It
turns out that this first return map is a $2$-fold extension of the
natural extension of the Gauss map (see Subsection \ref{sub:d=1}).
However, when $d=2$ and $c=1$, the calculation of
$L_{2,1}$ leads to a seven-tuple integral and we have only succeeded in reducing it
to a triple integral which  can be evaluated numerically (see Subsection \ref{sub:d=2}).

When $d=c=1$, the double inequality $\frac{1}{2}\leq q_{n+1}\dd(q_{n}
\theta,\mathbb{Z})\leq 1$ shows that the behaviors of the two sequences
$(\frac{1}{n}\ln q_{n})_n$ and $(\frac{-1}{n}\ln \dd(q_{n}\theta,\mathbb{Z}))_n$ are the
same and each of the limits in Theorem \ref{thm:levy-d} implies the other.
When $d$ is greater than or equal to two, no such double inequality exists. Indeed, it
has been proven in $\cite{Chev-Khin}$ that when $c=1$ and $d\geq 2$,
\[
\lim\inf_{n\rightarrow\infty}q_{n+1}\dd(q_{n}\theta,\mathbb{Z}^{d})^{d}=0
\]
for almost all $\theta$ in $\mathbb{R}^{d}$. Observe that  Theorem \ref{thm:jager} implies this latter result; it is an immediate consequence of the
fact that $0$ is in the support of the measure $\nu_{d,c}$. Hopefully for the proof of
Theorem \ref{thm:levy-d}, the weaker inequality
\[
\dd(\theta Q_{n}(\theta),\mathbb{Z}^{d})\geq\frac{1}{\left\Vert Q_{n}
(\theta)\right\Vert _{\mathbb{R}^{c}}^{c/d}\ln\left\Vert Q_{n}(\theta
)\right\Vert }
\]
which holds almost surely by the convergence part of the Khintchin-Groshev
theorem (see \cite{BeDiVe}), is enough to link both of the limits in Theorem \ref{thm:levy-d}.

In  Section 9, we extend the aforementioned result of \cite{Chev-Khin}
to best simultaneous approximations of matrices. Our proof leads to the
stronger result
\[
\lim\inf_{n\rightarrow\infty}\left\Vert Q_{n+k}(\theta)\right\Vert
_{\mathbb{R}^{c}}^{c}\dd(\theta Q_{n}(\theta),\mathbb{Z}^{d})^{d}=0
\]
for almost all $\theta$ in $\M_{d,c}(\mathbb{R)}$ and all $k\in\mathbb{N}$.

Obviously, if the matrix $\theta$ is badly approximable, then
\[
\lim\inf_{n\rightarrow\infty}\left\Vert Q_{n+k}(\theta)\right\Vert
_{\mathbb{R}^{c}}^{c}\dd(\theta Q_{n}(\theta),\mathbb{Z}^{d})^{d}>0
\]
because by definition, $\lim\inf_{n\rightarrow\infty}\left\Vert Q_{n}(\theta)\right\Vert
_{\mathbb{R}^{c}}^{c}\dd(\theta Q_{n}(\theta
),\mathbb{Z}^{d})^{d}>0$. When $d=2$ and $c=1$, we prove that the set of
$\theta$ such that, $\lim\inf_{n\rightarrow\infty}q_{n+1}\dd(q_{n}\theta
,\mathbb{Z})^{2}>0$, is not reduced to the set of badly approximable vectors; see Proposition~\ref{prop:notbad}.

{\bf Overview of the paper.} The notations are collected in Section 2. 
The third section is devoted to best Diophantine approximations, to minimal vectors and  to the relations between these two notions. 
The fourth section is devoted to the definition of the transversal $S$ and to another transversal $S'$. The key point of this section is the expression of visiting times of the transversals by the flow trajectories  with the minimal vectors and the best approximations (Lemmas \ref{lem:return-minimal} and \ref{lem:return-best}). The induced measure on the transversal is defined in Section 5. In Section 6, Theorem \ref{thm:Birkhoff}  is an important result of the paper, it is a version of Theorem \ref{thm:levy-d}  in the space of lattices. 
In Section 7, when $c=1$, a parametrization of the transversal $S$ is described  together with a formula for the induced measure. Furthermore, the two cases $d=1$ and $2$ are considered. Section 7 is not necessary for the proof of Theorems \ref{thm:levy-d} and \ref{thm:jager}.
In Section 8, we give the general result, Theorem \ref{thm:reduction}, which allows to convert the almost sure convergence in the space of  lattices proved in Theorem \ref{thm:Birkhoff}, to an almost sure convergence in the space of matrices $M_{d,c}(\R)$ needed in Theorems \ref{thm:levy-d} and \ref{thm:jager}. 
In Section 9, we study the behavior of the quantities $q_{n+k}^c(\theta)r_n^d(\theta)$ and finish the proof of Theorem \ref{thm:jager}. At last, we ask a few miscellaneous questions in Section 10.

{\bf Historical Note:} L\'{e}vy's proof does not rely on the ergodic theorem, which
was not known for non-invertible maps at that time. A proof of the Birkhoff ergodic
theorem for non-invertible maps was given by Fr\'{e}d\'{e}ric Riesz in 1945
(see \cite{Riesz}) and then a proof of L\'{e}vy's theorem using the ergodic theorem
was given in \cite{Ryll-Nardzewski}. The authors would like to thank Vitaly
Bergelson for bringing \cite{Ryll-Nardzewski} to their attention.

\section{Notation}
Let $d$ and $c$ be two positive integers.
\subsection{Vectors and distances}
For a positive integer $n$, let $\left\Vert \cdot\right\Vert _{\mathbb{R}^{n}}$ denote the usual
Euclidean norm on $\mathbb{R}^{n}$ and let $B_{\mathbb{R}^n}(u,r)$ denote the ball of center $u$ and radius $r$ associated with the norm $\|.\|_{\mathbb{R}^n}$.

Let us denote $\left\Vert .\right\Vert _{d,c}$ the norm defined on $\mathbb{R}^{d+c}$ by
 $\left\Vert (u,h)\right\Vert _{d,c}=\max\{\left\Vert
u\right\Vert _{\mathbb{R}^{d}},\left\Vert h\right\Vert _{\mathbb{R}^{c}}\}$ and let $B_{d,c}(X,r)$ denote the ball of center $X$ and radius $r$ associated with the norm $\left\Vert .\right\Vert_{d,c}$.

For $X=(u,h)$ in $\mathbb{R}^{d+c}$, let  $\left|  X\right|  _{-}=\left\|
h\right\|  _{\mathbb{R}^{c}}$  denote the \textit{height} of the vector $X$ and
let $\left|  X\right|  _{+}=\left\|  u\right\|  _{\mathbb{R}^{d}}$ denote the norm of the
projection of $X$ in the horizontal space. We also denote $X_{+}=u$ and
$X_{-}=h$ the   horizontal and vertical components of $X$.

For a vector $X$ in $\mathbb{R}^{d+c}$, let  $C(X)$ denote  the closed cylinder in $\mathbb{R}^{d+c}$
\[
C(X)=B_{\mathbb{R}^{d}}(0,\left\vert X\right\vert _{+}
)\times B_{\mathbb{R}^{c}}(0,\left\vert X\right\vert _{-}),
\]
and if $Y$ is another vector, let $C(X,Y)$ denote  the closed cylinder
\[
C(X,Y)=B_{\mathbb{R}^{d}}(0,\left\vert X\right\vert _{+}
)\times B_{\mathbb{R}^{c}}(0,\left\vert Y\right\vert _{-}).
\]
 Observe that in the basic case $d=2$ and $c=1$, $C(X)$ is a usual cylinder with height $2 \times |X|_-$ and radius $|X|_+$.
\includegraphics[width=12cm,height=5cm]{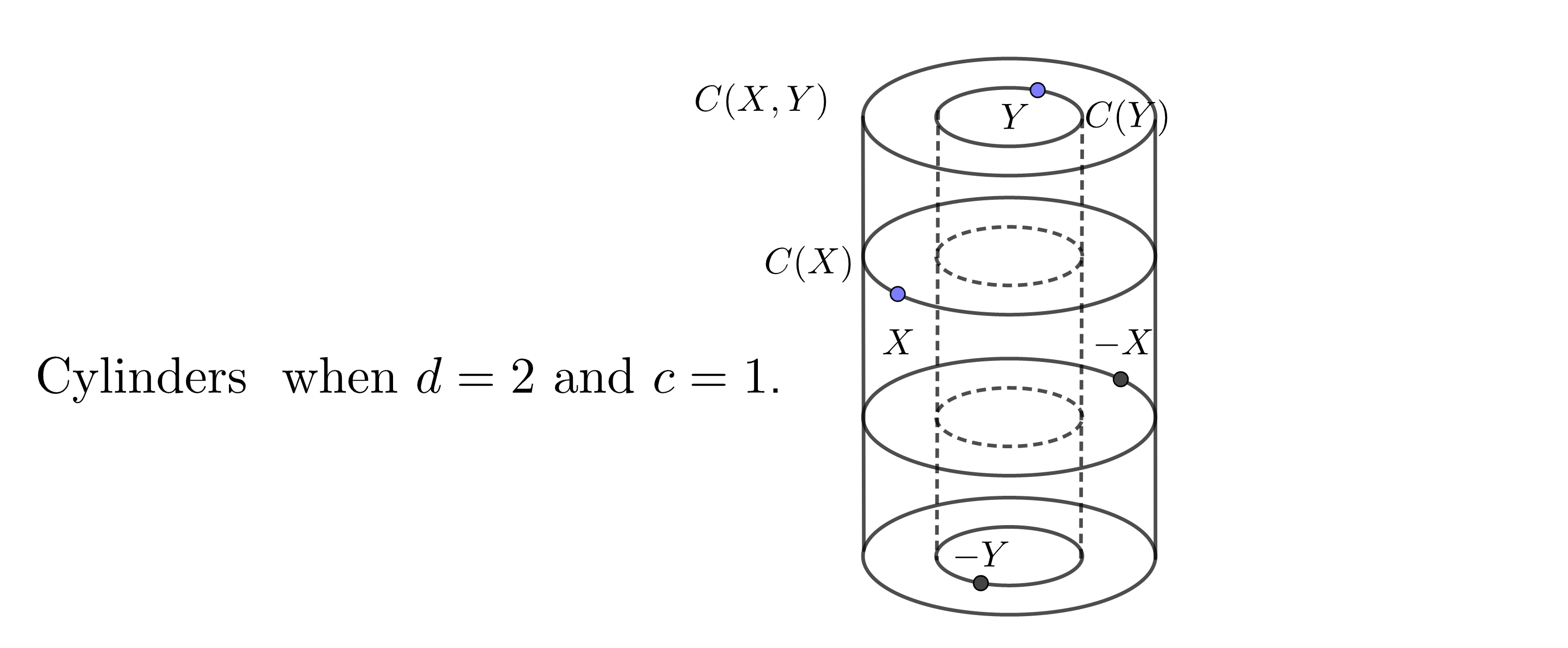} 

In all cases, we denote by $\dd(x,y)$ the distance associated with the underlying norm between the two points $x$ and $y$ and by $\dd(x,A)$ the distance between the point $x$ and the set $A$.

\subsection{Matrices}

 We fix a norm on $\M_{d+c}(\mathbb{R})$.  All the distances and balls in
the  space of matrices are associated with this norm.   This norm will be used as a tool in the proofs and when the choice of this norm is important we will specify it. When $E$ is a subset of
$\M_{d+c}(\mathbb{R})$, let $B_{E}(x,r)$ denote the set of matrices $y$ in $E$ such that $\dd(x,y)\leq r$.

Let $I_{n}$ denote the identity matrix in $\M_{n}(\mathbb{R)}$.
Let 
\[
g_{t}=
\begin{pmatrix}
e^{ct}I_{d} & 0\\
0 & e^{-dt}I_{c}
\end{pmatrix}
\in \SL (d+c,\mathbb R),
\]
$t\in\mathbb{R}$, denote the standard diagonal flow,
$E_{-}= \{0\}\times\mathbb R^c $ denote the contracting direction of the flow and
$E_{+}=\mathbb R^d \times \{0\}$ denote the expanding direction of the flow. We also refer to $E_{+}$ as the horizontal subspace and to $E_{-}$ as the vertical subspace.

Let  $\mathcal{H}_{>0}$ denote the subgroup of all matrices $M_{\theta}=\left(
\begin{array}
[c]{cc}
I_{d} & -\theta\\
0 & I_{c}
\end{array}
\right)$, $\theta\in
\M_{d,c}(\mathbb{R})$,  and let $\mathbb{T}_{d,c}$ denote its image in the space of unimodular lattices
$\mathcal{L}_{d+c}$.
In the same manner, let $\mathcal{H}_{<}$ denote the subgroup of $\SL(d+c,\mathbb{R})$ of matrices of
the form
\[
\left(
\begin{array}
[c]{cc}
I_{d} & 0\\
B & I_{c}
\end{array}
\right)
\]
and let $\mathcal{H}_{\leq}$ denote the subgroup of $\SL(d+c,\mathbb{R})$ of
matrices of the form
\[
\left(
\begin{array}
[c]{cc}
A & 0\\
B & C
\end{array}
\right)
\]
where $A\in\GL(d,\mathbb{R)}$, $B\in M_{c,d}(\mathbb{R)}$ and $C\in
\GL(c,\mathbb{R})$.

Observe that for any $M\in\mathcal H_{<}$, $g_tMg_{-t}\rightarrow I_{d+c}$ when $t$ tends to $+\infty$ and that the distance from $g_tMg_{-t}$ to $I_{d+c}$ is bounded when $t\in\mathbb R_+$ for all $M\in\mathcal H_{\leq}$.
\subsection{Lattices}\label{sec:NotationLattice}
Let  $\mathcal{L}=\mathcal{L}_{d+c}$ denote the space
of $(d+c)$-dimensional unimodular lattices in $\mathbb
R^{d+c}$. We identify $\mathcal{L}_{d+c}$ with 
$\SL (d+c,\mathbb{R})/\SL (d+c,\mathbb{Z})$.

For $\theta\in\M_{d,c}(\mathbb{R)}$, let $\Lambda_{\theta}$ be the lattice in $\mathbb R^{d+c}$ generated by the columns of the matrix
$M_{\theta}$, i.e.,  $\Lambda_{\theta}=M_{\theta}\mathbb{Z}^{d+c}$.

Suppose $\mathbb R^n$ is equipped with a norm $\Vert.\Vert$. For a lattice $\Lambda$ and an integer $i\in\{1,...,n\}$, let $\lambda_i(\Lambda,\|.\|)$ denote the $i$-th minimum of the lattice $\Lambda$ with respect to the norm $\Vert.\Vert$, i.e.,
\[
\lambda_i(\Lambda,\Vert.\Vert)=\inf\{\lambda>0:B(0,\lambda)\cap\Lambda \text{ contains at least $i$ independent vectors}\}.
\]
Observe that $\lambda_1(\Lambda)$ is the length of the shortest nonzero vector in $\Lambda$. When there is no ambiguity about the norm, we write $\lambda_i(\Lambda)$ instead of $\lambda_i(\Lambda,\Vert.\Vert)$.

We fix a measure $\mu$  on $\mathcal{L}_{d+c}$  invariant by the $\SL
(d+c,\mathbb{R)}$ left action. Recall that $\mu$ is unique up to a multiplicative constant.

\section{Best approximations}

\subsection{Best Diophantine approximations}

Best approximation vectors have been introduced since a long time inside
proofs in an non-explicit form. 
To our knowledge, C. A. Rogers in 1951 \cite{Rog} was the first to define best
Diophantine approximations. In fact, he defined the best Diophantine approximations associated with the sup norm.

Lagarias was the first to have undertaken a systematic study of the best Diophantine approximations, see  \cite{Lag1,Lag2,Lag3,Lag4,Lag5}. For instance, Lagarias proved that it is not possible to conciliate the unimodularity and the best approximation property (see \cite{Lag4} and \cite{Mosh}). This explains why a multidimensional continued fraction algorithm never satisfies the best approximation property when this algorithm is defined by a piecewise unimodular homographic generalized Gauss map.   

The connection between best simultaneous Diophantine approximations and the diagonal flow in the space of 
lattices goes back at least to Hermite, see \cite{Her}. In \cite{Lag1}, based on  Hermite's idea, Lagarias defined the Minkowski multidimensional continued fraction expansion  which gives all the ``Hermite best approximations'' to a given $\theta\in\R^d$ (see also \cite{GrLa} for a deep study of the case $d=c=1$).

For more details and the history of best Diophantine approximations see the surveys  \cite{Mosh,Chev-Mosc}.

\begin{definition}
Let $\theta\in\M_{d,c}(\mathbb{R)}$. \newline 
1. A nonzero vector
$Q\in\mathbb{Z}^{c}$ is a \textit{best simultaneous Diophantine approximation
denominator} of $\theta$ (with respect to the norms $\Vert.\Vert_{\mathbb{R}^d}$ and $\Vert.\Vert_{\mathbb{R}^c}$), if for all nonzero $U$ in $\mathbb{Z}^{c}$,
\begin{align*}
\left\Vert U\right\Vert _{\mathbb{R}^{c}}  &  <\left\Vert Q\right\Vert
_{\mathbb{R}^{c}}\Rightarrow \dd(\theta Q,\mathbb{Z}^{d})<\dd(\theta U,\mathbb{Z}^{d}) \text{ and}\\
\left\Vert U\right\Vert _{\mathbb{R}^{c}}  &  =\left\Vert Q\right\Vert
_{\mathbb{R}^{c}}\Rightarrow \dd(\theta Q,\mathbb{Z}^{d})\leq \dd(\theta U,\mathbb{Z}^{d}).
\end{align*}
2. An element $(P,Q)$ in $\mathbb{Z}^{d}\times\mathbb{Z}^{c}$ is a
\textit{best Diophantine approximation vector} of $\theta$ if $Q$ is a best
simultaneous Diophantine approximation denominator of $\theta$\ and if
\[
\left\Vert \theta Q-P\right\Vert _{\mathbb{R}^{d}}=\dd(\theta Q,\mathbb{Z}
^{d}).
\]

\end{definition}

If $Q=0$ is the only solution in $\mathbb{Z}^d$ of the equation $\theta Q=0 \mod \Z^{d}$, the set of best Diophantine approximation denominators of
$\theta$ is infinite. Numbering the set of best approximation denominators in
ascending order of the norm $q=\left\Vert Q\right\Vert _{c}$, we obtain two
sequences
\[
q_{0}=q_{0}(\theta)=\left\Vert
Q_{0}(\theta)\right\Vert _{c}=\lambda_{1}(\mathbb{Z}^{c})<q_{1}=q_{1}(\theta)=\left\Vert
Q_{1}(\theta)\right\Vert _{c}<...<q_{n}=q_{n}(\theta)=\left\Vert Q_{n}
(\theta)\right\Vert _{c}<....
\]
and 
\[
r_{0}=r_{0}(\theta)=\dd(\theta Q_0,\mathbb{Z}^{d})>r_{1}=r_{1}(\theta)=\dd(\theta Q_1,\mathbb{Z}^{d})>...>r_{n}=r_{n}(\theta)=\dd(\theta Q_n,\mathbb{Z}^{d})>....
\]

When $d=c=1$, by the best approximation property of the ordinary continued fraction expansion (see \cite{HaWr} Theorem 182), the integers $q_{0}
,q_{1},...,q_{n},...$ are the denominators of the convergents of the ordinary continued fraction
expansion of $\theta$. The only slight difference is that in the ordinary
continued fraction expansion, it can happen that $q_{0}=q_{1}=1$.   Unless $d=c=1$, the sequence $(V_n=(P_n,Q_n))_n$ of best approximation vectors does not satisfy a linear recurrence relation of the form $V_n=a_{n,1}V_{n-1}+\dots+a_{n,d+c}V_{n-d-c}$, but   satisfies weak forms of well-known properties of the one-dimensional continued fraction expansion, see Lemmas \ref{lem:Minkowski} and \ref{lem:lowerrate} below.

\subsection{Minimal vectors in lattices}

The notion of minimal vector goes
back to Voronoi. He used minimal vectors to find units in cubic fields (see \cite{Vor} and also \cite{BuchI,BuchII}). Minimal vectors are a key ingredient of the proofs of Theorems \ref{thm:levy-d} and \ref{thm:jager}, they allow to
convert statements about best simultaneous Diophantine approximations of vectors or 
matrices into statements about lattices.

\begin{definition}
Let $M\in\SL (d+c,\mathbb{R})$ and let $\Lambda=M\mathbb{Z}^{d+c}\in
\mathcal{L}_{d+c}$. A nonzero vector $X\in\Lambda$ is a minimal vector in
$\Lambda$ (with respect to the norms $\left\Vert \cdot\right\Vert
_{\mathbb{R}^{d}}$ and $\left\Vert \cdot\right\Vert _{\mathbb{R}^{c}}$) if the
only nonzero vectors $Y\in\Lambda$ in the cylinder $C(X)$ are such that
$C(X)=C(Y)$, i.e., such that
\[
\left\vert X\right\vert _{+}=\left\vert Y\right\vert _{+},\ \left\vert
X\right\vert _{-}=\left\vert Y\right\vert _{-}.
\]
Two minimal vectors $X$ and $Y$ are equivalent if they define the same cylinder.
\end{definition}

\begin{remark} 
The connection between minimal vectors and shortest vectors of a lattice in $\mathbb{R}^{d+c}$ with respect to the norm $\Vert.\Vert_{d,c}$ is simple: in a
given lattice there is at least one shortest vector of the lattice that is minimal, and if $X = (u,h)$
is minimal in the lattice, then $g_tX$ is a shortest vector of the lattice $g_t\Lambda$ with $t(X) = \tfrac{1}{d+c}
\ln\frac{\Vert h\Vert_{\mathbb{R}^c}}{\Vert u\Vert_{\mathbb{R}^d}}$
(if $u=0$ or $h=0$, it is to understand that $X$ is a shortest vector of $g_t\Lambda$ when $t$ is in a
neighborhood of $+\infty$ or $-\infty$).
Moreover if a vector $X = (u,h)$  is a ``robust'' shortest vector, {\it i.e.}, if
$g_tX$ is a shortest vector in $g_t\Lambda$ for all $t$ in a neighborhood of $t(X)$, then $X$ is a minimal vector.
\end{remark}
\begin{remark} The connection between minimal vectors and shortest vectors of a lattice in $\mathbb{R}^{d+c}$ with respect to the standard Euclidean norm is not that simple: in a given lattice, if $g_tX$ is a shortest vector in the lattice $g_t\Lambda$ for some $t$,  then $X$ is a minimal vector in $\Lambda$. However some minimal vectors are not of this shape even when $d=c=1$, this was already observed by Humbert in 1916 (see \cite{Hum} and \cite{Lag1}).
\end{remark}

The connection between minimal vectors and best Diophantine approximations is given by: 

\begin{lemma}(\cite{Cheu,Chev-Mosc}) 
\label{lem:min-best}
For all $\theta$ in $\M_{d,c}(\mathbb{R)}$, a vector $X=M_{\theta}V^T$ where $V=(P,Q)\in\mathbb{Z}^{d+c}$, is a minimal vector   with nonzero
height in the lattice $\Lambda_{\theta}$ iff $V$ is a best approximation vector of $\theta$ such that
\[
\Vert \theta Q-P\Vert_{\mathbb{R}^d}<d(0,\mathbb{Z}^{d}\setminus\{0\}).
\]
\end{lemma}

\begin{proof}
Let $V=(P,Q)$ be a non zero vector in $\mathbb{Z}^{d+c}$ and $X=M_{\theta
}V^{T}$.
\newline Suppose that $V$ is a best approximation vector of $\theta$
such that $\Vert \theta Q-P\Vert_{\mathbb{R}^d}<d(0,\mathbb{Z}^{d}\setminus\{0\})$. Let $W=(u,h)$ be
a non zero vector in $\mathbb{Z}^{d+c}$. If $Y=M_{\theta}W^{T}$ is in $C(X)$
then $\Vert h\Vert_{\mathbb{R}^c} \leq \Vert Q\Vert_{\mathbb{R}^c}$ and $\Vert \theta h-u\Vert_{\mathbb{R}^d}\leq \Vert \theta Q-P\Vert_{\mathbb{R}^d}$. The latter inequality implies
$h\neq0$. Since $V$ is a best approximation vector the two previous
inequalities must be two equalities which shows $X$ is minimal with non zero
height. Conversely suppose that $X$\ is a minimal vector in $\Lambda_{\theta}$
with non zero height. If a vector $W=(h,u)$ in $\mathbb{Z}^{d+c}$ is such that $\Vert h\Vert_{\mathbb{R}^c} < $ or $\leq\Vert Q\Vert_{\mathbb{R}^c}$, then by minimality, $\Vert \theta h-u\Vert_{\mathbb{R}^d}>$ or $\geq \Vert \theta Q-P\Vert_{\mathbb{R}^d}$ 
which shows that $V$ is a best approximation vector of $\theta$. Moreover, since $X$ is minimal with nonzero height, $\Vert \theta Q-P\Vert_{\mathbb{R}^d}<d(0,\mathbb{Z}^{d}\setminus\{0\})$.
\end{proof}

\subsubsection{The sequence of minimal vectors} Given a lattice $\Lambda$ in $\mathcal{L}_{d+c}$, we select one minimal vector
in each equivalent class of minimal vectors. We number these vectors in
ascending order of heights. Such a numbering exists because $0$ is the only
possible limit point of the set of heights of minimal vectors (see the proof of Lemma \ref{lem:lowerrate}). We obtain a
sequence
\[
...X_{-n}(\Lambda),...,\ X_{-1}\left(  \Lambda\right)  ,\;X_{0}(\Lambda
),\;X_{1}(\Lambda),...
\]
This sequence might be finite, infinite one-sided or two-sided. The sequence
$(\left\vert X_{n}(\Lambda)\right\vert _{+})_{n}$ is decreasing while
the sequence $(\left\vert X_{n}(\Lambda)\right\vert _{-})_{n}$ is 
increasing. Such a numbering is not unique but two such numberings differ only by a constant. Although this is not essential, we will set such a numbering in the Section \ref{sec:visit} about return
times.\bigskip

We shall  use the following notations
\[
q_{n}(\Lambda)=\left|  X_{n}(\Lambda)\right|  _{-}\quad\text{ and }\quad
r_{n}(\Lambda)=\left|  X_{n}(\Lambda)\right|  _{+}.
\]

Lemma \ref{lem:min-best} above is very important. It shows that for $\theta
\in\M_{d,c}(\mathbb{R})$, the sequences $(q_{n}(\Lambda_{\theta}))_{n}$
and $(q_{n}(\theta))_{n}$ are deduced one another by a shift. Therefore, if one of the two limits
\[
\lim_{n\rightarrow\infty}\frac{1}{n}\ln q_{n}(\theta),\ \text{and }\lim_{n\rightarrow\infty}\frac{1}{n}\ln q_{n}(\Lambda
_{\theta})
\]
exists, then the other exists and have the same value. The same results holds with the sequences $(r_n(\theta))_n$ and $(r_n(\Lambda_{\theta}))_n$.

 The classical inequality 
 \[
 q_{n+1}r_n\leq 1
 \]
 which holds for the one-dimensional continued fraction expansion of a real number, can be extended  to minimal vectors of lattices or to best approximation vectors. This fact is well known but it is worth stating it.

\begin{lemma}
\label{lem:Minkowski}
There is a constant $C_{d,c}$ depending only on $c$ and $d$ such that for all lattice $\Lambda \in \mathcal L_{d+c}$, all matrices $\theta \in \M_{d,c}$  and all integers $n$, we have 
\begin{align*}
q_{n+1}^c(\Lambda)r_n^d(\Lambda)\leq C_{d,c}\text{ and }
q_{n+1}^c(\theta)r_n^d(\theta)\leq C_{d,c}.
\end{align*}
\end{lemma}

\begin{proof}
Just use the Minkowski convex body theorem with the cylinder defined by two consecutive minimal vectors or best approximation vectors.
\end{proof}

The classical inequality
\[
q_{n+2}\geq 2q_{n}
\]
which holds for the denominators of the one-dimensional continued fraction expansion
of a real number, can be extended to minimal vectors of lattices. This
inequality has already been extended to best simultaneous Diophantine
approximations, see \cite{Lag3},\ \cite{Lag4} and \cite{Chev-Mosc}. The
extension to minimal vectors of lattices is straightforward.

\begin{lemma}
\label{lem:lowerrate}There is a positive integer constant $A=A(d,c)$ such that for
any $\Lambda$ in $\mathcal{L}_{d+c}$ and any $n\in\mathbb{Z}$, if
$X_{n}(\Lambda),\ X_{n+1}(\Lambda),....,X_{n+A}(\Lambda)$ exist, then
\begin{align*}
q_{n+A}(\Lambda)  &  \geq 2q_{n}(\Lambda),\\
r_{n+A}(\Lambda)  &  \leq\frac{1}{2}r_{n}(\Lambda).
\end{align*}

\end{lemma}

\begin{proof}
Let $A$ be an integer constant such that if $A$ points $(x_1,y_1),...,(x_A,y_A)$ are in the product  of balls $B_{\mathbb{R}^{d}}(0,r_1)\times B_{\mathbb{R}^{c}}(0,r_2)$ with $r_1,r_2\geq 0$ then
there exist two indices $i\neq j$ such that $\left\Vert x_{i}-x_{j}\right\Vert
_{\mathbb{R}^{d}}\leq\frac{1}{2}r_1$ and $\left\Vert y_{i}-y_{j}\right\Vert
_{\mathbb{R}^{c}}\leq\frac{1}{2}r_2$. With this choice of the constant $A$, if $k\geq A$ is a positive
integer such that $q_{n+k}(\Lambda)\leq 2q_{n}(\Lambda)$, then there are two
integers $0\leq i<j\leq k$ such that the vector $X_{n+j}(\Lambda
)-X_{n+i}(\Lambda)$ satisfies both conditions
\[
\left\{
\begin{array}
[c]{c}
\left\vert X_{n+j}(\Lambda)-X_{n+i}(\Lambda)\right\vert _{+}\leq\frac{1}
{2}r_{n}(\Lambda)\\
\left\vert X_{n+j}(\Lambda)-X_{n+i}(\Lambda)\right\vert _{-}\leq \frac{1}{2}2q_{n}
(\Lambda)
\end{array}
\right.
\]
which contradicts the definition of $X_{n}(\Lambda)$. The same way of
reasoning leads to the other inequality.
\end{proof}

We shall use several times the following very simple lemma which is a
consequence of the following observation. For any minimal vector $X$ of a
lattice $\Lambda$ in $\mathbb{R}^{d+c}$ and any $t\in\R$, $g_{t}X$ is a minimal vector of the
lattice $g_{t}\Lambda$. It follows that

\begin{lemma}
\label{lem:flow-minimal}Let $\Lambda$ be in $\mathcal{L}_{d+c}$ and let $t$ be
in $\mathbb{R}$. The  sequence of minimal vectors of the lattice $g_t\Lambda$ is  $(g_{t}
(X_{n}(\Lambda)))_{n}$.
\end{lemma}

\section{The transversals $S$ and $S^{\prime}$}

We assume that $\mathbb{R}^{d+c}$ is endowed with the norm
\[
\left\|  (x,y)\right\|  _{d,c}=\max\{\left\|  x\right\|
_{\mathbb{R}^{d}},\left\|  y\right\|  _{\mathbb{R}^{c}}\}.
\]
The main idea of the proofs of Theorems \ref{thm:levy-d} and \ref{thm:jager}
is to induce the flow $g_{t}$ on the co-dimension one subset
\[
S=\{\Lambda\in\mathcal{L}_{d+c}:\lambda_{2}(\Lambda)=\lambda_{1}(\Lambda)\}
\]
where $\lambda_{1}(\Lambda)$ and $\lambda_{2}(\Lambda)$ are the first two minima of the lattice $\Lambda$ associated with the above norm. 
The reason for introducing  $S$ is that for a generic lattice $\Lambda$, one can use the flow to transform any cylinder $C(X,Y)$ defined by two consecutive minimal vectors $X$ and $Y$, into a cylinder with equal horizontal and vertical sizes which brings the lattice $\Lambda$ into $S$ (see Subsection \ref{sec:visit}).
And then, to use the Birkhoff ergodic theorem with the first return map associated with the flow $g_t$. For
technical reason it is better to slightly restrict the definition of $S$. For instance the
above set is not a submanifold of $\mathcal{L}_{d+c}$. It could have some ``branching
points'' while a slightly smaller set is clearly a submanifold, see Lemma
\ref{lem:submanifold}. It will be convenient to use two transversals $S$ and
$S^{\prime}$ for the proof of Theorem \ref{thm:levy-d}. These two transversals are
very similar; we state all the results we need for both transversals but we only perform 
the proofs for the first transversal $S$.

\subsection{Definition of $S$}\label{sec:defS}

The transversal $S=S_{d,c}$ is the set of lattices $\Lambda$ in $\mathcal{L}_{d+c}$ such
that there exist two independent vectors $v_0^S(\Lambda)$ and $v_1^S(\Lambda)$ in
$\Lambda$ such that:

\begin{itemize}
\item $\left\vert v_1^S(\Lambda)\right\vert _{+}$ and $\left\vert
v_0^S(\Lambda)\right\vert _{-}$ are $<\left\vert v_1^S(\Lambda)\right\vert
_{-}=\left\vert v_0^S(\Lambda)\right\vert _{+}$,

\item the only nonzero points of $\Lambda$ in the ball $B_{d,c}(0,\lambda_{1}(\Lambda))$ are $\pm v_0^S(\Lambda)$ and $\pm v_1^S(\Lambda)$.
\end{itemize}

Observe that for $\Lambda$ in $S$, $v_0^S(\Lambda)$ and $v_1^S(\Lambda)$ are
unique up to sign and are consecutive minimal vectors of $\Lambda$.

Since $v_0^S(\Lambda)$ and $ v_1^S(\Lambda)$ are linearly independent
and are both in the ball $B_{d,c}(0,\lambda_{1}(\Lambda
))$, $S$ is included in the set
\[
\{\Lambda\in\mathcal{L}_{d+c}:\lambda_{1}(\Lambda)=\lambda_{2}(\Lambda)\}.
\]

\subsection{Definition of $S^{\prime}$}

The transversal $S^{\prime}$ is the set of lattices in $\mathcal{L}_{d+c}$ such
that there exists a vector $w_{0}^{S'}(\Lambda)$ in $\Lambda$ such that:

\begin{itemize}
\item the only nonzero points
of $\Lambda$ in the ball $B_{d,c}(0,\lambda_{1}(\Lambda))$ are $\pm w_{0}^{S'}(\Lambda)$,

\item the ball $B_{d,c}(0,\lambda_{1}(\Lambda))$ is equal to the cylinder $C(w_{0}^{S'}
(\Lambda))$, i.e., $|w_0^{S'}|_+=|w_0^{S'}|_-$.
\end{itemize}

Observe that $w_{0}^{S'}(\Lambda)$ is unique up to sign and is a minimal vector of
$\Lambda$.

We gather some technical results in the next three subsections. Then,  the subsection \ref{sec:visit} about the visiting times and the subsection \ref{sec:function-S}    about two functions defined on $S$, explain the relations between  the transversal $S$ and the sequences of minimal vectors of a lattice.

\subsection{Lattices bases and minima}

We shall need the following results about lattices. Recall that a sub-lattice $\Gamma$ of a lattice $\Lambda$ is primitive if $\Gamma_{\mathbb R}\cap\Lambda=\Gamma$ where $\Gamma_{\mathbb R}$ is the vector subspace generated by $\Gamma$.

\begin{lemma}
Suppose $\mathbb{R}^{n}$ is equipped with any norm $\left\Vert .\right\Vert $.
Let $\Lambda$ be a lattice in $\mathbb{R}^{n}$ and let $v_{1}$, $v_{2}$ be two
independent vectors of $\Lambda$ such that $\left\Vert v_{1}\right\Vert =\lambda
_{1}(\Lambda)$ and $\left\Vert v_{2}\right\Vert =\lambda_{2}(\Lambda)$. Then
$\ \mathbb{Z}v_{1}+\mathbb{Z}v_{2}$ is a primitive sub-lattice of $\Lambda$
unless $\frac{1}{2}(v_{1}+v_{2})\in\Lambda$ and $\left\Vert v_{1}\right\Vert
=\left\Vert v_{2}\right\Vert =\left\Vert \frac{1}{2}(v_{1}+v_{2})\right\Vert $.
\end{lemma}

\begin{proof}
Consider the parallelogram $\mathcal{P}$ defined by the vectors $v_{1}$ and
$v_{2}$. Let $v$ be an element of $\Lambda$ that belongs to the interior of
$\mathcal{P}$. If $v$ is not in the segment joining $v_{1}$ and $v_{2}$ then
the distance from $v$ to either $0$ or $v_{1}+v_{2}$ is of the form
$\left\Vert t_{1}v_{1}+t_{2}v_{2}\right\Vert $ for some positive real numbers
$t_{1}$ and $t_{2}$ with $t_{1}+t_{2}<1$. Hence this distance is
$\leq t_{1}\left\Vert v_{1}\right\Vert +t_{2}\left\Vert v_{2}\right\Vert
<\lambda_{2}(\Lambda)$ which contradicts the definition of $\lambda
_{2}(\Lambda)$. If $v$ is in the segment joining $v_{1}$ and $v_{2}$ but
is not the point $\frac{1}{2}(v_{1}+v_{2})$ then the distance from $v$ to
either $v_{1}$ or $v_{2}$ is of the form $\left\Vert t(v_{1}-v_{2})\right\Vert
$ with $t<\frac{1}{2}$ which implies that this distance is $<\lambda
_{2}(\Lambda)$, again a contradiction. If $v=\frac{1}{2}(v_{1}+v_{2})$, we
have $\left\Vert v\right\Vert \leq\frac{1}{2}(\left\Vert v_{1}\right\Vert
+\left\Vert v_{2}\right\Vert )$ which is $<\lambda_{2}(\Lambda)$ unless
$\left\Vert v_{1}\right\Vert =\left\Vert v_{2}\right\Vert=\left\Vert v \right\Vert .$
\end{proof}

It follows that when the norm is strictly convex, the sub-lattice $\mathbb{Z}
v_{1}+\mathbb{Z}v_{2}$ is always primitive. In our setting despite that the
norm is not strictly convex it is possible to use the above lemma. With our
choice of the norm on $\mathbb{R}^{d+c}$, the triangle inequality is strict
for two vectors one inside the ``top'' of the cylinder $B_{d,c}(0,r)$ and one inside the lateral side of $B_{d,c}(0,r)$. Therefore,

\begin{corollary}
\label{cor:lattice}Let $\Lambda$ be in $S$. Then the vectors $v_0^S(\Lambda)$
and $v_1^S(\Lambda)$ associated with $\Lambda$ are the first two vectors of a
basis of $\Lambda$.
\end{corollary}

\subsection{Geometric properties of $S$ and $S^{\prime}$}

\begin{lemma}
\label{lem:submanifold}$S$ and $S^{\prime}$ are submanifolds of
$\mathcal{L}_{d+c}$ of dimension $(d+c)^{2}-2$, transverse to the diagonal
flow $g_{t}$.
\end{lemma}

\begin{proof}
Let $\Lambda_{0}$ be in $S$ and call $v_{0}^S(\Lambda_{0})$ and $v_{1}^S(\Lambda_{0})$
the two vectors provided by the definition of $S$. By Corollary
\ref{cor:lattice},  $v_{0}^S(\Lambda_{0})$ and $v_{1}^S(\Lambda_{0})$
are the first two vectors of a basis $(b_{1},...,b_{d+c})$ of $\Lambda_{0}$.
We can find a small enough positive real number $\varepsilon$ such that for
any $(v_{1},...,v_{d+c})$ in the open set
\[
W=B_{\mathbb{R}^{d+c}}(b_{1},\varepsilon)\times...\times B_{\mathbb{R}^{d+c}
}(b_{d+c},\varepsilon),
\]
\begin{itemize}
\item the matrix $M=M(v_{1},...,v_{d+c})$, the columns of which are the $v_{i}
$, is in $\GL (d+c,\mathbb R)$ and the sets $WP$, $P\in\SL (d+c,\mathbb  Z)$ are disjoint,
\item the vectors $\pm v_{1}$ and $\pm v_{2}$ are the only nonzero vectors of
the lattice $\Lambda=M\mathbb Z^{d+c}$ in the cylinder $C(v_{1},v_{2})$,
\item $\left\Vert v\right\Vert_{d,c} >\left\Vert v_{1}\right\Vert_{d,c} $ and $\left\Vert
v_{2}\right\Vert_{d,c} $ for all $v$ in $\Lambda\setminus\{0,\pm v_{1},\pm v_{2}\}$,
\item $|v_1|_{+}>|v_1|_{-}$ and $|v_2|_{-}>|v_2|_{+}$.
\end{itemize}
Consider the map
\begin{align*}
f  &  :W\rightarrow\mathbb{R}^{2}\\
&  :M=(v_{1},...,v_{d+c})\rightarrow(f_{1}(M)=\det M,f_{2}(M)=\left\vert
v_{1}\right\vert _{+}^{2}-\left\vert v_{2}\right\vert _{-}^{2}).
\end{align*}
Then a lattice $\Lambda=M\mathbb{Z}^{d+1}$ with $M\in W$, is in $S$ iff
$f(M)=(1,0)$. To prove that $S$ is a submanifold, it is enough to show that the differential
$Df(M)$ is onto at every point $M$ in $W$. The differential of $f_{2}$ is
given by
\[
Df_{2}(M).(w_{1},...,w_{d+c})=2v_{1}^{+}.w_{1}^{+}-2v_{2}^{-}.w_{2}^{-}.
\]
The linear map $Df_{2}(M)$ depends only on $w_{1}$ and $w_{2}$ and since
$v_{1}^{+}\neq 0$ for all $M$ in $W$, $Df_{2}(M)$ is never the zero map. The
differential of $f_{1}$ is given by
\[
Df_{1}(M).(w_{1},...,w_{d+c})=\sum_{i,j=1}^{d+c}(-1)^{i+j}\Delta_{i,j}w_{i,j}
\]
where $w_{j}=(w_{i,j})_{i=1,...,d+c}$, $j=1,...,d+c$ and $\Delta_{i,j}$ is the
$(i,j)$-minor of the matrix $M$. Since $\det M\neq 0$, one at least of the
minors $\Delta_{i,3}$, $i\leq d+c$, is not zero. Therefore the linear
$Df_{1}(M)$ is not zero and depends on $w_{3}$. It follows that the two linear
maps $Df_{1}(M)$ and $Df_{2}(M)$ are linearly independent for all $M$ in $W$
which implies that $S$ is a submanifold of $\mathcal{L}_{d+c}$.
To show that the flow is transverse to $S$, we have to check that for a matrix
$M=M(v_1,\dots,v_{d+c})$ in $W$ such that $f(M)=0$ we have $Df(M).(w_{1},...,w_{d+c})\neq 0$ when
$w_{i}=(cv_{i}^{+},-dv_{i}^{-})$. Now, for such $w_{i}$, $Df_{2}
(M).(w_{1},...,w_{d+c})=2c\left\vert v_{1}\right\vert _{+}^{2}$ $+2d\left\vert
v_{2}\right\vert _{-}^{2}>0$, hence $Df(M).(w_{1},...,w_{d+c})$ is not zero.
\end{proof}

\subsection{Negligible sets}
An important ingredient of the proof of Theorem \ref{thm:levy-d} is that, for a given lattice $\Lambda$, the visiting times $t$, i.e. the times $t$ such that $g_t\Lambda\in S$, can be read from the sequence $(X_n(\Lambda))_n$ of minimal vectors.
However, this reading is straightforward only for generic lattices, a small subset of lattices has to be avoided.
\subsubsection{A negligible set $\mathcal N$ in the space of lattices }

Let $\mathcal{N}=\mathcal{N}_{d,c}$ be the set of lattices $\Lambda$ in
$\mathcal{L}_{d+c}$ such that either 
\begin{itemize}
\item there exist two vectors $v_{1}$, and $v_2$, such that
$v_{1}\neq\pm v_{2}$ and $\left\vert v_{1}\right\vert _{+}=\left\vert
v_{2}\right\vert _{+}>0$ or $\left\vert v_{1}\right\vert _{-}=\left\vert
v_{2}\right\vert _{-}>0$,
\item or there exists a nonzero vector in $\Lambda$ lying in
the vertical subspace $\{0\}\times\mathbb{R}^{c}$, 
or in the horizontal subspace
$\mathbb{R}^{d}\times\{0\}$.
\end{itemize}

\begin{remark}
Any  lattice $\Lambda_{\theta}=M_{\theta}\mathbb Z^{d+c}$  (see Section \ref{sec:NotationLattice}) is in $\mathcal{N}$ because it contains $\mathbb Z^d\times\{0\}$.
\end{remark}

Recall that we have fixed a measure $\mu$  on $\mathcal{L}_{d+c}$  invariant by the $\SL
(d+c,\mathbb{R)}$ action. 

\begin{lemma}
$\mathcal{N}$ is $\mu$-negligible and $g_{t}$ invariant.
\end{lemma}

\begin{proof}
Clearly $\mathcal{N}$ is $g_{t}$ invariant and the set of lattices with a
nonzero vector in the vertical subspace or in the horizontal subspace is
negligible. So we are reduced to prove that if $X\neq\pm Y$ are two nonzero
vectors in $\mathbb{Z}^{d+c}$ the set of matrices $M$ in $\SL (d+c,\mathbb{R)}
$ satisfying one of the equations
\[
\left\vert MX\right\vert _{+}^{2}-\left\vert MY\right\vert _{+}^{2}=0
\]
or
\[
\left\vert MX\right\vert _{-}^{2}-\left\vert MY\right\vert _{-}^{2}=0
\]
is of zero measure. Firstly, by symmetry, it is enough to deal with one of the
equations, say the first one. Secondly, by homogeneity it is equivalent to prove
that the set of matrices in $M_{d+c}(\mathbb{R)}$ that satisfy this equation
is of zero measure. Since this is an algebraic equation, it is enough to prove
that there exists at least one matrix $M$ such that $\left\vert MX\right\vert
_{+}^{2}-\left\vert MY\right\vert _{+}^{2}\neq 0$. If $X$ and $Y$ are
proportional just choose a matrix $M$ such that $|MX|_{+}\neq 0$. Otherwise,
first choose a vector $Z$ in the subspace spanned by $X$ and $Y$ that is
orthogonal to $X$. Observe that $Z.Y\neq 0$. Next choose a $d$-dimensional
subspace $V$ of $\mathbb{R}^{d+c}$ containing $Z$ and orthogonal to $X$. A
matrix $M$, the first $d$ rows of which are a basis of $V$, is such that
$\left\vert MX\right\vert _{+}=0$ and $\left\vert MY\right\vert _{+}\neq 0$.
\end{proof}

\begin{remark}
A lattice $\Lambda$ that is not in $\mathcal{N}$ has a
bi-infinite sequence of minimal vectors and is in $S$ iff $\lambda_{1}
(\Lambda)=\lambda_{2}(\Lambda)$. In fact if $u$ and $v$ are two linearly independent shortest vectors in $\Lambda$ then $|u|_+\neq|v|_+$ and $|u|_-\neq |v|_-$. Therefore $\Lambda\in S$.
\end{remark}

\subsubsection{A negligible set $\mathcal M$ in the space of matrices $\M_{d,c}(\mathbb R)$}

Let $C$ be a positive real constant and let $\mathcal{M}=\mathcal{M}_{d,c}=\mathcal M_{d,c}(C)$ be the set of matrices $\theta \in \M_{d,c}(\R)$
 such that either 
 \begin{itemize}
 \item there exist two nonzero vectors $X\neq \pm Y$ in $\Z^{d+c}$ with nonzero heights such that $\left\vert M_{\theta}X\right\vert _{+}=\left\vert M_{\theta}Y\right\vert _{+}$
 \item or there exist infinitely many pairs $X\neq \pm Y$ in $\Z^{d+c}$  such that $\left\vert X\right\vert _{-}=\left\vert Y\right\vert _{-}\neq 0$ and 
 $\left\vert M_{\theta}X\right\vert _{+},\,\left\vert M_{\theta}Y\right\vert _{+}\leq C\left\vert X \right\vert _{-}^{-\frac{c}{d}}$. 
 \end{itemize}
 Observe that the second item only occurs when $c>1$.
 The set $\mathcal M$ depends on the constant $C$. Actually, we will only use the value  $C=C_{d,c}$ where $C_{d,c}$ is given by Lemma \ref{lem:Minkowski}.

  \begin{lemma}
  $\mathcal M$ is negligible with respect to the Lebesgue measure on $\M_{d,c}(\mathbb R)$.
  \end{lemma}

 \begin{proof} We prove that $\mathcal M$ is included in a countable union of negligible sets.
 
Given $X\neq \pm Y$ two nonzero vectors in $\Z^{d+c}$ with nonzero heights, consider the set $\mathcal M(X,Y)$ of matrices $\theta \in \M_{d,c}(\R)$  such that
 $\left\vert M_{\theta}X\right\vert _{+}=\left\vert M_{\theta}Y\right\vert _{+}$. In order to show that $\mathcal M(X,Y)$ has zero measure it is enough to show that the polynomial 
\begin{align*}
f(\theta)&=\left\vert M_{\theta}X\right\vert _{+}^2-\left\vert M_{\theta}Y\right\vert _{+}^2\\
&=\left\Vert X_{+}\right\Vert _{\R^d}^2-\left\Vert  Y_{+}\right\Vert_{\R^d}^2
-2(X_{+}.\theta X_{-}-Y_{+}.\theta Y_{-})
 +\left\Vert\theta X_{-}\right\Vert _{\R^d}^2-\left\Vert \theta Y_{-}\right\Vert_{\R^d}^2
\end{align*}
 is not the zero polynomial.

If $X_{-}\neq \pm Y_{-}$, we can choose $\theta_0$ such that $\left\Vert\theta_0 X_{-}\right\Vert _{\R^d}^2-\left\Vert \theta_0 Y_{-}\right\Vert_{\R^2} ^2\neq 0$. With this choice, the one variable polynomial $P(t)= f(t\theta_0)$ has a nonzero degree two monomial which implies that the polynomial $f$ is not the zero polynomial. 

If $X_{-}=Y_{-}$ (the case $X_{-}=-Y_{-}$ is similar), $f(\theta)=\left\vert
X\right\vert_{+}^2-\left\vert Y\right\vert_{+}^2-2(X_{+}-Y_{+}).\theta X_{-}.$ Since $X_{-}\neq
0$, the map $\varphi:\theta\in\M_{d,c}(\R)\rightarrow \theta X_{-}\in \R^d$ is onto. It follows that we can
choose $\theta$ such that $\theta X_{-}=X_{+}-Y_{+}$. With this value of $\theta$, we obtain $f(\theta)-f(-\theta)=-4\left\vert X_{+}-Y_{+}\right\vert^2 \neq 0$ which implies that the
polynomial $f$ is not the zero polynomial.  It follows that $\mathcal M(X,Y)$ is negligible.\\

Consider now, for a positive integer $n$, the set $\mathcal M_n$ of matrices $\theta \in
\M_{d,c}(\R)$  such that there is a pair of linearly independent vectors $(X, Y)$  in $\Z^{d+c}\times \Z^{d+c}$  such 
that $n\leq\left\vert X\right\vert _{-},\,\left\vert Y\right\vert _{-}<n+1$ and 
\[ 
\left\vert M_{\theta}X\right\vert _{+},\,\left\vert M_{\theta}Y\right\vert _{+}\leq C\left\vert
X \right\vert _{-}^{-\frac{c}{d}}.
\]
We want to prove that the set of matrices $\theta$
that are in infinitely many $\mathcal M_n$ is negligible. We can move into the space $\M_{d,c}(\R/\Z)$ and consider instead  the set $\mathcal T_n$ of $\theta \in \M_{d,c}(\R/\Z)$ such that there exist $q,q' \in \Z^c$ linearly independent with $n\leq\left\Vert q \right\Vert_{\R^c},\,\left\Vert q' \right\Vert_{\R^c}<n+1$ and 
\[
d(\theta q,\Z^d),\,d(\theta q',\Z^d)\leq Cn^{-\frac{c}{d}}.
\]
For $q$ fixed, the measure of the set of $\theta \in \M_{d,c}(\R/\Z)$ such that
$d(\theta q,\Z^d)\leq Cn^{-\frac{c}{d}}$ is $a_{d,c} n^{-c}$ where the constant $a_{d,c}$ depends only on $C$ and the dimensions $d$ and $c$. When the inequality holds simultaneously for two linearly independent integer vectors $q$ and $q'$, the measure is bounded above by the square of $a_{d,c} n^{-c}$. 
It follows that the measure of $\mathcal T_n$ is bounded above by 
\[
u_n=\text{card}\{(q,q')\in\Z^c\times \Z^c:n\leq\left\Vert q \right\Vert_{\R^c},\,\left\Vert q' \right\Vert_{\R^c}<n+1\}\times a_{d,c}^{2} n^{-2c}.
\]
By Borel-Cantelli, it is enough to prove that $\Sigma_n u_n<\infty$. Now, $\text{card}\{q\in\Z^c:n\leq\left\Vert q \right\Vert_{\R^c}<n+1\}\ll n^{c-1}$, hence
\[
u_n\ll n^{-2} 
\]
and we are done.

\end{proof}

\subsection{Visiting times and return maps}\label{sec:visit}

Let $\Lambda$ be in $\mathcal{L}_{d+c}$. By definition of $S$, when
$g_{t}\Lambda$ is in $S$, $v_{0}^S(g_{t}\Lambda)$ and $v_{1}^S(g_{t}\Lambda)$
are two consecutive minimal vectors of $g_{t}\Lambda$. Therefore, 
\[
\left\{
\begin{array}
[c]{l}
v_{0}^S(g_{t}\Lambda)=g_{t}X_{k}(\Lambda)\\
v_{1}^S(g_{t}\Lambda)=g_{t}X_{k+1}(\Lambda)
\end{array}
\right.  
\]
for an integer $k$.
Hence $e^{ct}\left\vert X_{k}(\Lambda)\right\vert _{+}=e^{-dt}\left\vert
X_{k+1}(\Lambda)\right\vert _{-}$ which implies
\[
t=\frac{1}{d+c}\ln\frac{q_{k+1}(\Lambda)}{r_{k}(\Lambda)}.
\]
It follows that the set of real numbers $t$ such that $g_{t}\Lambda\in S$ is
included in the discrete set
\[
V_{\Lambda}(S)=\{t_{k}=\frac{1}{d+c}\ln\frac{q_{k+1}(\Lambda)}{r_{k}(\Lambda
)}:\ k\in\mathbb{Z\}}.
\]
It can happen that some values $t_{k}$ are skipped, but in that case, $\Lambda$
must be in $\mathcal{N}$. So, when $\Lambda$ is not in $\mathcal{N}$,
$g_{t}\Lambda\in S$ iff $t\in V_{\Lambda}(S)$. For the transversal $S^{\prime}$, the
same results hold with
\[
V_{\Lambda}(S')=\{t_{k}^{\prime}=\frac{1}{d+c}\ln\frac{q_{k}(\Lambda
)}{r_{k}(\Lambda)}:\ k\in\mathbb{Z\}}.
\]
It follows that for almost all $\Lambda$, both the backward trajectory
$(g_{t}\Lambda)_{t\leq 0}$ and the forward trajectory $(g_{t}\Lambda)_{t\geq 0}$
visit the two tranversals  $S$ and $S^{\prime}$  infinitely often. 
Therefore the first return/entrance times in
$S$ and $S^{\prime}$,
\begin{align*}
\tau(\Lambda)  &  =\inf\{t>0:g_{t}(\Lambda)\in S\}\in\mathbb{R}_{>0}
\cup\{\infty\},\\
\tau^{\prime}(\Lambda)  &  =\inf\{t>0:g_{t}(\Lambda)\in S^{\prime}
\}\in\mathbb{R}_{>0}\cup\{\infty\},
\end{align*}
are finite almost everywhere and 
the first return/entrance maps 
\begin{align*}
R(\Lambda)  &  =g_{\tau(\Lambda)}\Lambda,\\
R^{\prime}(\Lambda)  &  =g_{\tau^{\prime}(\Lambda)}\Lambda
\end{align*}
are defined for all $\Lambda$ that are not in $\mathcal{N}$.

For an integer $n\geq 1$, let us denote $\tau_{n}$ the $n$-th return (or entrance)
time in $S$, i.e.
\[
\tau_{n}(\Lambda)=\sum_{k=0}^{n-1}\tau(R^{k}(\Lambda))
\]
($R^{0}(\Lambda)=\Lambda$ for all $\Lambda$ in $\mathcal{L}_{d+1}$).
It will be convenient to choose the numbering of the sequence of minimal vectors $(X_n(\Lambda))$ in order to have simple formulas for the return time and the return map.
\medskip

\textbf{Numbering convention}: \textit{For a lattice $\Lambda \in \mathcal L_{d+c}$, $n=0$
is the smallest integer $n\in\mathbb{Z}$ such that
\[
\left\vert X_{n+1}(\Lambda)\right\vert _{-}\geq\left\vert X_{n}(\Lambda
)\right\vert _{+}
\] 
when the set of such integers is non empty.}
\medskip

With this numbering convention, for all $\Lambda$ in $S$, we have 
\[
X_0(\Lambda)=v_0^S(\Lambda)
\]
and for all $\Lambda$ in $S'$, we have 
\[
X_0(\Lambda)=w_0^{S'}(\Lambda).
\]
Moreover when $\Lambda\notin S$ is not in $\mathcal N$,
\begin{align*}
&\left\vert X_{1}(\Lambda)\right\vert _{-}-\left\vert X_{0}(\Lambda
)\right\vert _{+}> 0,\\
&\tau(\Lambda)=\frac{1}{d+c}\ln\left(\frac{\left\vert X_{1}(\Lambda)\right\vert_{-}}
{\left\vert X_{0}(\Lambda
)\right\vert_{+}}\right),
\end{align*}
and
\[
g_{\tau(\Lambda)}X_{0}(\Lambda)=\pm v_{0}^S(R(\Lambda))=\pm X_{0}(R(\Lambda)).
\]

Let us summarize the above.
\begin{lemma}
\label{lem:return-minimal}
Let $\Lambda$ be a lattice in $\mathcal L_{d+c}$.\newline
1. The set of visiting times in $S$ is included in $V_{\Lambda}(S)=\{t_{k}=\frac{1}{d+c}\ln\frac{q_{k+1}(\Lambda)}{r_{k}(\Lambda
)}:\ k\in\mathbb{Z\}}$ and the set of visiting times in $S'$ is included in $V_{\Lambda}(S')=\{t_{k}^{\prime}=\frac{1}{d+c}\ln\frac{q_{k}(\Lambda
)}{r_{k}(\Lambda)}:\ k\in\mathbb{Z\}}$. \newline
2. Suppose $\Lambda$ is not in $\mathcal{N}$. The set of visiting times in $S$ is equals to $V_{\Lambda}(S)$ and the set of visiting times in $S'$ is equals to $V_{\Lambda}(S')$.

\end{lemma}

For $\theta$  in $\M_{d,c}(\R)$, we also need to connect the visiting times of the transversal $S$ with the best approximation vectors of $\theta$. 

\begin{lemma}
\label{lem:return-best}
Let $\theta$  be in $\M_{d,c}(\R)\setminus \mathcal M$. Then for all large enough integers $n$, 

\[
t_{n}(\theta)
=\frac{1}{d+c}\ln\frac{q_{n+1}(\theta)}{r_{n}(\theta)}
\] 
and 
\[
t'_{n}(\theta)
=\frac{1}{d+c}\ln\frac{q_{n}(\theta)}{r_{n}(\theta)}
\]
are visiting times for the transversals $S$ and $S'$ respectively.

\end{lemma} 

\begin{proof}
Let $\theta$  be in $\M_{d,c}(\R)\setminus \mathcal M$. Consider the sequence of all best approximation vectors $(Y_n(\theta))_{n\in\N}$ of $\theta$. By Lemma \ref{lem:min-best}, there are integers $n_1$ and $k$ such that $X_{n+k}(\Lambda_{\theta})=M_{\theta}Y_{n}(\theta)$ for all $n\geq n_1$. Since $\theta$ is not in $\mathcal M$, by Lemma \ref{lem:Minkowski} there is another integer $n_2$ such that for all $n\geq n_2$, the only nonzero vector of $\Lambda_{\theta}$ in the box $\mathcal C(X_{n+k}(\theta),X_{n+k+1}(\theta))$ are $\pm X_{n+k}(\theta)$ and $\pm X_{n+k+1}(\theta)$. This means that for all $n$ large enough, the times 
\[
t_n(\theta)=\frac{1}{d+c}\ln\frac{q_{n+1}(\theta)}{r_{n}(\theta)}
\text{ and }
t'_{n}(\theta))
=\frac{1}{d+c}\ln\frac{q_{n}(\theta)}{r_{n}(\theta)}
\] 
are  visting times for the transversals $S$ and $S'$.
\end{proof}

\subsection{Functions defined on $S$}\label{sec:function-S}

Let $\Lambda$ be in $S$. By definition of $S$, the functions
\[
\rho,\rho^{\ast}:\Lambda\in S\rightarrow\ln\frac{\left\vert v_1^S(\Lambda)\right\vert _{-}}{\left\vert v_0^S(\Lambda)\right\vert _{-}}
,\;\ln\frac{\left\vert v_0^S(\Lambda)\right\vert _{+}}{\left\vert
v_1^S(\Lambda)\right\vert _{+}}\in\mathbb{R}_{>0}\cup\{+\infty\}
\]
are well defined on $S$. The next lemma is easy, its proof is close to the beginning of
the proof of Lemma \ref{lem:submanifold} and is omitted.

\begin{lemma}
The functions $\Lambda\in\mathcal S\rightarrow\, |v_0^S(\Lambda)|_-,\,|v_0^S(\Lambda)|_+,\,|v_1^S(\Lambda)|_-,\,|v_1^S(\Lambda)|_+$ are continuous and thus the functions $\rho$ and $\rho^{\ast}$ are continuous.
\end{lemma}

The following lemma is important. On the one hand, it will imply that the functions $\rho$ and $\rho^\ast$ are integrable. On the other hand, it will explain the connection between the Lévy's constant
$L_{d,c}$ and the average return times on $S$.

\begin{lemma}
\label{lem:return-rho}Let $\Lambda$ be a lattice in $S\setminus\mathcal{N}$.
Then
\[
\tau(\Lambda)=\frac{1}{d+c}(\rho(R(\Lambda))+\rho^{\ast}(\Lambda)).
\]

\end{lemma}

\begin{proof}
Let $\Lambda$ be in $S\setminus\mathcal{N}$. By definition of $S$,
$q_{1}(\Lambda)=r_{0}(\Lambda)$. Hence, by Lemma \ref{lem:return-minimal},
\begin{align*}
(d+c)\tau(\Lambda)  &  =\ln\frac{q_{2}(\Lambda)}{r_{1}(\Lambda)}\times
\frac{r_{0}(\Lambda)}{q_{1}(\Lambda)}\\
&  =\ln\frac{q_{2}(\Lambda)}{q_{1}(\Lambda)}+\ln\frac{r_{0}(\Lambda)}
{r_{1}(\Lambda)}\\
&  =\rho(R(\Lambda))+\rho^{\ast}(\Lambda).
\end{align*}
\end{proof}

\section{Induced measures on $S$ and $S^{\prime}$}
\subsection{Definition of the induced measures $\mu_S$ and $\mu_{S'}$}

Recall that we have fixed a measure $\mu$  on $\mathcal{L}_{d+c}$  invariant by the action of $\SL
(d+c,\mathbb{R)}$. Let us explain how the flow $g_t$ and the measure $\mu$ induce a measure $\mu_S$ on $S$ which is invariant by the first return map $R$. The construction of such a measure is standard but we have not found any reference suitable for our case. 

Let $X=(S\setminus\mathcal N)\times\R$, let $\pi:X\rightarrow\mathcal L_{d+c}\setminus \mathcal N$ be the map defined by $\pi(\Lambda,t)=g_t\Lambda$, let $g_t^X:X\rightarrow X$ be the flow defined by $g_t^X(\Lambda,s)=(\Lambda,s+t)$, let $\overline R:X\rightarrow X$ be the map defined by $\overline R(\Lambda,t)=(R(\Lambda),t-\tau(\Lambda))$ and let $D=\{(\Lambda,t)\in X:0\leq t<\tau(\Lambda)\}$.
\begin{itemize}
\item
Since the first return map $R$ is a bijection from $S\setminus\mathcal N$ onto itself, $\overline R$ is a bijection with inverse $\overline{R}^{-1}(\Lambda,t)=(R^{-1}(\Lambda),t+\tau(R^{-1}(\Lambda))$. It is easy to check that the sets $\overline R^n(D)$, $n\in\Z$, are disjoint.
Since the infinite sequence $\tau_n(\Lambda)$ strictly increases from $-\infty$ to $+\infty$, $X=\cup_{n\in\Z}\overline R^n(D)$. 
\item
We can lift the measure $\mu$ from $\mathcal L_{d+c}$ to $X$: we define the  measure $\overline{\mu}$ on $X$ by 
\[
\overline{\mu}(B)=\sum_{n\in\Z}\mu(\pi(D\cap \overline R^{\,-n} B))
\]
for all Borel sets $B$ in $X$. By definition $\overline{\mu}$ is $\overline R$-invariant.
\item
Let us see that $\overline{\mu}$ is  $g_t^X$-invariant for all $t\in\R$. If $B$ is a Borel set in $X$, we have the partition $B=\cup_{m,n\in\Z}B_{m,n}$ where $B_{m,n}=B\cap \overline R^mD\cap g_{-t}^X \overline{R}^nD$. 
On the one hand, since $B_{m,n}\subset \overline{R}^mD$, we have $\overline{\mu}(B_{m,n})=\mu(\pi(D\cap\overline{R}^{-m}B_{m,n}))=\mu(\pi\overline{R}^{-m}B_{m,n})$, hence $\overline{\mu}(B_{m,n})=\mu(\pi B_{m,n})$
because $\pi\circ\overline R=\pi$. On the other hand, since $g_t^XB_{m,n}\subset \overline R^nD$,
\[
\overline{\mu}(g_t^XB_{m,n})=\overline{\mu}(\overline{R}^{-n}g_t^XB_{m,n})=\mu(\pi\overline{R}^{-n}g_t^XB_{m,n})=\mu(\pi g_t^X\overline{R}^{-n}B_{m,n})=\mu(g_t\pi\overline{R}^{-n}B_{m,n}),
\]
and since $\mu$ is $g_t$-invariant, we obtain, $\overline{\mu}(g_t^XB_{m,n})=\mu(B_{m,n})$.
\item
Since $\overline{\mu}$ is $g_t^X$-invariant, there exists a measure $\mu_S$ on $S\setminus\mathcal N$ such that
\[
\overline{\mu}=\mu_S\otimes \operatorname{Lebesgue}_{\R}.
\]
\item
By definition for any Borel set $B\subset S$ and any $\delta>0$ such that $B\subset\{\Lambda:\tau(\Lambda)>\delta\}$,
\[
\mu_S(B)\times\delta=\overline{\mu}(B\times[0,\delta])=\mu(\cup_{ t\in[0,\delta]}g_tB).
\]
\item
Let us show that $\mu_S$ is $R$-invariant. Let $B$ be a Borel set in $S\setminus \mathcal N$. Since $\overline{\mu}$ is $\overline{R}$-invariant,
\begin{align*}
    \mu_S(B)&=\overline{\mu}(B\times[0,1])=\overline{\mu}(\overline{R}(B\times[0,1]))\\
    &=\int_{R(B)}d\mu_S(\Lambda)\int_{-\tau(R^{-1}\Lambda)}^{1-\tau(R^{-1}\Lambda)}dt=\mu_S(R(B)).
\end{align*}
\item
We extend $\mu_S$ to $S$ by $\mu_S(S\cap\mathcal N)=0$.
\end{itemize}

With our definition of $\mu_S$, and since $\pi(D)=\mathcal L_{d,c}\setminus \mathcal N$, the Kac return-time theorem reduces to
\[
\mu(\mathcal L_{d+c})=\mu(\pi(D))=\overline{\mu}(D)=\int_S\tau(\Lambda)d\mu_S(\Lambda).
\]

Finally, the measure $\mu_S$ is $R$-ergodic because any $R$-invariant set of positive measure in $S$ gives rise to a set of  flow trajectories of positive measure and it is known that the action of the flow $(g_t)$ on $(\mathcal L_{d+c},\mu)$ is ergodic (see \cite{BeMa}).

The flow induces a measure $\mu_{S'}$ on $S'$ as well. 
\subsection{Finiteness of the induced measures} 
Let us prove that the induced measures $\mu_S$ and $\mu_{S'}$ are finite. This is a simple consequence of the next lemma which will be very important in the proof of Theorem 
\ref{thm:reduction} about the almost sure convergence in $M_{d,c}(\mathbb R)$.

\begin{lemma}
\label{lem:nb-returntimes} 1. There exists an integer constant $A$ such that
$\tau_{A}(\Lambda)\geq 1$ for all $\Lambda$ in $S$
 and  $\tau_{A}^{\prime}(\Lambda)\geq 1$ for all $\Lambda$ in $S'$.
\end{lemma}

\begin{proof}
Let $A$ be the constant given by Lemma \ref{lem:lowerrate} about the growth rate of the sequences $(q_n(\Lambda))_n$ and $(r_n(\Lambda))_n$. For all integers $k$, we have 
\begin{align*}
\frac{1}{d+c}\left(\ln\frac{q_{k+A+1}(\Lambda)}{r_{k+A}(\Lambda)}-
\ln\frac{q_{k+1}(\Lambda)}{r_{k}(\Lambda)}\right)
&\geq 
\frac{1}{d+c}\ln\frac{q_{k+A+1}(\Lambda)}{q_{k+1}(\Lambda)}\\
&\geq\frac{1}{d+c}\ln 2.
\end{align*}
Since by Lemma \ref{lem:return-minimal}, the set of visiting times of $\Lambda$ is included in
\[
V_{\Lambda}(S)=\{t_{k}=\frac{1}{d+c}\ln\frac{q_{k+1}(\Lambda)}{r_{k}(\Lambda
)}:\ k\in\mathbb{Z\}},
\]
$\tau_{A}(\Lambda)\geq \frac{1}{d+c}\ln 2
$. Multiplying $A$ by the smallest integer larger than $\frac{d+c}{\ln 2}$ we are done.
\end{proof}

\begin{proposition}
$\mu_{S}(S)$ and $\mu_{S^{\prime}}(S^{\prime})$ are finite and nonzero.
\end{proposition}

\begin{proof}
By Lemma \ref{lem:submanifold}, $S$ is submanifold of co-dimension one, transverse to the flow, and nonempty, hence  $\mu_S(S)>0$.

Since $\mu_{S}$ is $R$-invariant,  for all $k$
\[
\int_{S}\tau(R^{k}\Lambda)d\mu_{S}(\Lambda)=\int_{S}\tau(\Lambda)d\mu
_{S}(\Lambda),
\]
and by the Kac return time theorem,
\[
\int_{S}\tau(\Lambda)d\mu
_{S}(\Lambda)=\mu(\mathcal{L}_{d+1}),
\]
therefore,
\[
 \int_{S}\tau_{A}(\Lambda)d\mu_{S}(\Lambda)=\int_{S}\sum_{k=0}^{A-1}\tau
(R^{k}\Lambda)d\mu_{S}(\Lambda)=A\mu(\mathcal{L}_{d+c}).
\]
By the above lemma, $\tau_A\geq 1$ on $S$, hence $\mu_{S}(S)\leq A\, \mu(\mathcal{L}_{d+c})$ which is finite by Siegel's theorem.
\end{proof}

\section{Almost sure convergence in the space of lattices}

\subsection{Consequence of  the Birkhoff ergodic theorem}
Recall that for a lattice $\Lambda\in\mathcal L_{d+c}$, $q_{n}(\Lambda)=\left|  X_{n}(\Lambda)\right|  _{-}$ and $r_{n}(\Lambda)=\left|  X_{n}(\Lambda)\right|  _{+}$ where $(X_n(\Lambda))_n$ is the sequence of minimal vectors of $\Lambda$ with the numbering convention of Section \ref{sec:visit}. Recall also that $\tau_n(\Lambda)$ is the $n$-th return time in $S$.
\begin{theorem}\label{thm:Birkhoff}
\label{Birkhoff}There exist two positive constants $L_{d,c}$ and
$L_{d,c}^{\ast}$ such that for almost all lattices $\Lambda$ in $\mathcal{L}_{d+c}$,
\begin{align*}
\lim_{n\rightarrow\infty}\frac{1}{n}\ln q_{n}(\Lambda)  &  =\frac{1}{\mu
_{S}(S)}\int_{S}\rho d\mu_{S}=L_{d,c}>0,\\
\lim_{n\rightarrow\infty}\frac{-1}{n}\ln r_{n}(\Lambda)  &  =\frac{1}{\mu
_{S}(S)}\int_{S}\rho^{\ast}d\mu_{S}=L_{d,c}^{\ast}>0,\\
\lim_{n\rightarrow\infty}\frac{1}{n}\tau_{n}(\Lambda)  &  =\frac{1}
{d+c}(L_{d,c}+L_{d,c}^{\ast})=\frac{\mu(\mathcal{L}_{d+c})}{\mu_S(S)}.
\end{align*}
Moreover, given two new  Euclidean norms on the horizontal space $\mathbb{R}^{d}\times\{0\}$ and the vertical space $\mathbb{R}^{c}\times\{0\}$,  if we define the minimal vectors and the transversal $S$ with these two new Euclidean norms, then the conclusion of the theorem remains valid with the same constant $L_{d,c}$ and $L^*_{d,c}$.
\end{theorem}

The above theorem is an important result in the direction of Theorem \ref{thm:levy-d}.  The big difference between the two theorems is the reference measure. However, the second statement in Theorem \ref{thm:Birkhoff} explains why Theorem \ref{thm:levy-d} remains valid with the same constant $L_{d,c}$ and $L^*_{d,c}$ for any pair of Euclidean norms.  

\begin{proof}
Let $\Lambda$ be in $S\setminus\mathcal{N}$. By Lemma \ref{lem:return-rho},
$\tau(\Lambda)=\frac{1}{d+c}(\rho(R(\Lambda))+\rho^{\ast}(\Lambda))$. Because
the space of lattices has finite measure, the return time $\tau$ is in
$\mathcal{L}^{1}(S)$ and since  the non negative functions $\rho\circ R$
and $\rho^{\ast}$  are bounded above by $(d+c)\tau$, they are also in $\mathcal{L}^{1}(S)$. Making use of the
Birkhoff ergodic theorem with the functions $\rho$ and $\rho^{\ast}$, we obtain the
almost everywhere convergence of the sums
\[
\frac{1}{N}\sum_{n=0}^{N-1}\rho\circ R^{n},\ \frac{1}{N}\sum_{n=0}^{N-1}
\rho^{\ast}\circ R^{n}
\]
on $S$ to $R$-invariant functions. Now the ergodicity of the flow $g_{t}$ (see \cite{BeMa})
implies the ergodicity of the return map $R$. Therefore $\frac{1}{N}\sum_{k=0}^{N-1}\rho\circ
R^{n}$ and $\frac{1}{N}\sum_{n=0}^{N-1}\rho^{\ast}\circ R^{n}$ converge almost
everywhere on $S$ to the constants 
\[L_{d,c}=\frac{1}{\mu_{S}(S)}\int_{S}\rho d\mu_{S} \text{ and }
L_{d,c}^{\ast}=\frac{1}{\mu_{S}(S)}\int_{S}\rho^{\ast}d\mu_{S}.
\]
We would like to see that the Birkhoff sums converge to the same limits almost everywhere in the whole space of lattices. Let $\Lambda$ be in
$\mathcal{L}_{d+c}\setminus(S\cup\mathcal{N})$ and let $n$ be \ a positive integer. By
Lemma \ref{lem:return-minimal} and the numbering convention, for all $n\geq 1$,
\begin{align*}
\rho\circ R^{n}(\Lambda)  &  =\ln\frac{\left\vert X_{1}(R^{n}(\Lambda
))\right\vert _{-}}{\left\vert X_{0}(R^{n}(\Lambda))\right\vert _{-}}\\
&  =\ln\frac{\left\vert g_{\tau_{n}(\Lambda)}(X_{n}(\Lambda))\right\vert
_{-}}{\left\vert g_{\tau_{n}(\Lambda)}(X_{n-1}(\Lambda))\right\vert _{-}}\\
&  =\ln\frac{\left\vert X_{n}(\Lambda)\right\vert _{-}}{\left\vert
X_{n-1}(\Lambda)\right\vert _{-}}
\end{align*}
and $\rho^{\ast}\circ R^{n}(\Lambda)    =\ln\frac{\left\vert X_{0}(R^{n}
(\Lambda))\right\vert _{+}}{\left\vert X_{1}(R^{n}(\Lambda))\right\vert _{+}
}
  =\ln\frac{\left\vert X_{n-1}(\Lambda)\right\vert _{+}}{\left\vert
X_{n}(\Lambda)\right\vert _{+}}$
as well. It follows that if the Birkhoff sums $\frac{1}{N}\sum_{n=1}^{N}
\rho\circ R^{n}(\Lambda)$ and $\frac{1}{N}\sum_{n=1}^{N}\rho^{\ast}\circ
R^{n}(\Lambda)$ converge to $L_{d,c}$ and
$L_{d,c}^{\ast}$, then
\begin{align*}
\lim_{N\rightarrow\infty}\frac{1}{N}\ln q_{N}(\Lambda)    =L_{d,c}
\text{ and }
\lim_{N\rightarrow\infty}\frac{-1}{N}\ln r_{N}(\Lambda)   =L_{d,c}^{\ast}.
\end{align*}
Now the image by the map $R:\mathcal{L}_{d+c}\rightarrow S$ of a subset of
nonzero measure in $\mathcal{L}_{d+c}$ is a set of nonzero measure in $S$, therefore the sums $\frac{1}{N}\sum_{n=1}^{N}\rho\circ R^{n}$ and $\frac{1}
{N}\sum_{n=1}^{N}\rho^{\ast}\circ R^{n}$ converge almost everywhere in $\mathcal{L}_{d+c}$ to
$L_{d,c}$ and $L_{d,c}^{\ast}$.

By Lemma \ref{lem:lowerrate}, we know that the sequences $(q_{n}(\Lambda))_{n}$
and $(r_{n}(\Lambda)^{-1})_{n}$ have at least exponential growth rate;
therefore, the constants $L_{d,c}$ and $L_{d,c}^{\ast}$ are $>0$.

By Lemma \ref{lem:return-rho}, for all $\Lambda$ in $S\setminus\mathcal{N}$
and $k\in\mathbb{N}$,
\[
\tau_{k+1}(\Lambda)-\tau_{k}(\Lambda)=\frac{1}{d+c}(\rho(R^{k+1}
(\Lambda))+\rho^{\ast}(R^{k}(\Lambda))),
\]
hence
\[
\lim_{n\rightarrow\infty}\frac{1}{n}\tau_{n}(\Lambda)=\frac{1}{d+c}(L_{d,c}+L_{d,c}^{\ast})
\]
almost everywhere.

Finally, let us proof that the constants $L_{d,c}$ and $L_{d,c}^{\ast}$
do not depend on the Euclidean norm on $\R^d$ and $\R^c$.  For a matrix $M_d$ in 
$\SL (d,\R)$ and a matrix $M_c$ in 
$\SL (c,\R)$,
let us denote $M$ the matrix
\[
M=\left(
\begin{array}
[c]{cc}
M_d & 0\\
0 & M_{c}
\end{array}
\right) \in\SL(d+c,\mathbb R) .
\]
Since the action of $M$ on
$\mathcal{L}_{d+c}$ is measure preserving,
\begin{align*}
\lim_{n\rightarrow\infty}\frac{1}{n}\ln q_{n}(A\Lambda)  &  =L_{d,c}\\
\lim_{n\rightarrow\infty}\frac{-1}{n}\ln r_{n}(A\Lambda)  &  =L_{d,c}^{\ast}
\end{align*}
almost everywhere in $\mathcal{L}_{d+c}$. Now, a vector $AX$ in the lattice
$A\Lambda$ is  minimal  iff $X$ is a minimal vector of the  lattice $\Lambda$ with respect to
the new Euclidean norms $\left\Vert .\right\Vert _{A_d,\mathbb{R}^{d}}$  and $\left\Vert .\right\Vert _{A_c,\mathbb{R}^{c}}$ where
\[
\left\Vert u\right\Vert _{M_d,\mathbb{R}^{d}}=\left\Vert M_du\right\Vert
_{\mathbb{R}^{d}}, \text{ and }
\left\Vert v\right\Vert _{M_c,\mathbb{R}^{c}}=\left\Vert M_cv\right\Vert
_{\mathbb{R}^{c}}.
\]
Since up to  multiplicative constants, all the Euclidean norms are of the above form, the constants $L_{d,c}$ and $L_{d,c}^{\ast}$ do not depend on the Euclidean norms on $\mathbb R^d$ and $\R^c$.
\end{proof}

\subsection{A consequence of the Borel-Cantelli lemma}
For any lattice $\Lambda\in\mathcal L_{d,c}\setminus\mathcal N$ and any integer $n$, we have $q_{n+1}(\Lambda)r_n(\Lambda)=\lambda_1(g_{t_n}\Lambda)\lambda_2(g_{t_n}\Lambda)$
where $t_n=\tau_n(\Lambda)=\tfrac1{d+c}\ln\frac{q_{n+1}(\Lambda)}{r_n(\Lambda)}$ is the $n$-th return time in $S$.   
When $d=c=1$,   $\lambda_1(g_{t_n}\Lambda)\lambda_2(g_{t_n}\Lambda)\asymp 1$ by the Minkowski second convex body theorem. It implies that $L_{1,1}=L^*_{1,1}$. When $d$ or $c\geq 2$, the product of the first two minima of a lattice in $\mathcal L_{d,c}$ can be arbitrarily small, so it is not clear whether there is a relation between $L_{d,c}$ and $L^*_{d,c}$. However,

\begin{proposition}
\label{prop:borel-cantelli}With the notation of Theorem \ref{Birkhoff}, we have
\[
cL_{d,c}=dL_{d,c}^{\ast}.
\]

\end{proposition}
\begin{lemma}\label{lem:Siegel}
The function $\Lambda\in\mathcal L_{d+c}\rightarrow 1/\lambda_1(\Lambda)$ is in $L^1(\mu)$.
\end{lemma}
\begin{proof}
By Siegel formula (see \cite{SiegelAnnals}) used with the function 
$
f=1_{B_{\R^{d+c}}}(0,r)$, we have for some positive constants $C_1$ and $C_2$ depending on the dimension $d+c$ and on the normalization of the Haar measure, 
\[
C_1r^{d+c}=\int_{\R^{d+c}}f(x)dx=C_2\int_{\mathcal L_{d+c}}\sum_{x\in\Lambda\setminus\{0\}}f(x)d\mu(\Lambda)
\geq C_2\mu(\{\Lambda:\lambda_1(\Lambda)\leq r \}).
\] 
Therefore $\int_{\mathcal L_{d+c}}\frac{1}{\lambda_1}d\mu=\int_0^{\infty}\mu(\{\Lambda:\frac{1}{\lambda_1(\Lambda)}\geq t \})dt\leq\mu(\{\mathcal L_{d+c}\}) +\frac{C_1}{C_2} \int_1^{\infty}\frac{1}{t^{d+c}}dt<\infty$.
\end{proof}
\begin{proof}[Proof of the proposition]
The inequality $cL_{d,c}\leq dL_{d,c}^{\ast}$ is easy to prove. By Lemma \ref{lem:Minkowski}, for all lattices $\Lambda$ in $\mathcal{L}_{d+c}$ and all $n$, we have
$q_{n+1}^{c}(\Lambda)r_{n}^{d}(\Lambda)\leq C_{d,c}$. Hence for a lattice $\Lambda$
such that
\begin{align*}
\lim_{n\rightarrow\infty}\frac{1}{n}\ln q_{n}(\Lambda)    =L_{d,c}
\text{ and }
\lim_{n\rightarrow\infty}\frac{-1}{n}\ln r_{n}(\Lambda)    =L_{d,c}^{\ast},
\end{align*}
we have,
\[
cL_{d,c}-dL_{d,c}^{\ast}=c\lim_{n\rightarrow\infty}\frac{\ln q_{n}(\Lambda)}
{n}+d\lim_{n\rightarrow\infty}\frac{\ln r_{n}(\Lambda)}{n}=\lim_{n\rightarrow
\infty}\frac{\ln q_{n}^{c}(\Lambda)r_{n}^{d}(\Lambda)}{n}\leq 0.
\]
The converse inequality uses the Borel-Cantelli lemma. Let $\varphi:]0,\infty
[\rightarrow]0,\infty[$ be a decreasing function such that
$\sum_{n\geq 1}\varphi(n)<\infty$, for instance $\varphi(t)=\frac{1}{t^{\alpha
}}$ with $\alpha>1$. Since for such a function $\varphi$, $\lim\inf
_{n\rightarrow\infty}\frac{1}{n}\ln\varphi(n)=0$, the inequality
$cL_{d}-dL_{d}^{\ast}\geq 0$ holds, provided that for almost all lattices
$\Lambda$, we have $q_{n}^{c}(\Lambda)r_{n}^{d}(\Lambda)\geq\varphi(n)^{d+c}$
for $n$ large enough.
Let $K$ be a constant that will be chosen later. For each integer $n\geq 1$,
consider the set $A_{n}$ of lattices $\Lambda$ in $\mathcal{L}_{d+c}$ such
that
\[
\lambda_{1}(\Lambda)\leq K\varphi(n)
\]
and the set $B_{n}=g_{t_{n}}A_{n}$ where $t_{n}=\frac{1}{d}(\ln\varphi(n)-n)$.
By the above lemma $\left\Vert \lambda_{1}^{-1}\right\Vert _{1}<\infty$, hence using  Markov inequality, we obtain
\[
\mu(B_{n})=\mu(A_{n})\leq\frac{\left\Vert \lambda_{1}^{-1}\right\Vert _{1}
}{\frac{1}{K\varphi(n)}}\ll\varphi(n).
\]
Therefore, by the Borel-Cantelli lemma, the set $\mathcal{B}$ of lattices
$\Lambda$ in $\mathcal{L}_{d+c}$ such that $\Lambda\in B_{n}$ for infinitely
many integers $n$, is negligible.
Suppose now that $\Lambda$ is a lattice such that $q_{n}^{c}(\Lambda)r_{n}
^{d}(\Lambda)\leq\varphi(n)^{d+c}$ for infinitely many $n$. For each integer
$n\geq 1$, let $k_{n}=k_{n}(\Lambda)=$ $\lfloor\ln q_{n}(\Lambda)\rfloor$. By
Theorem \ref{Birkhoff}, for almost all lattices, we have $k_{n}\leq
(L_{d,c}+1)n$ for $n$ large enough. Therefore, for almost all lattices $\Lambda$, for $n$
large enough, if $q_{n}^{c}(\Lambda)r_{n}^{d}(\Lambda)\leq\varphi(n)^{d+c}$
then the vector $g_{t_{k_{n}}}^{-1}(X_{n}(\Lambda))$ satisfies both
\begin{align*}
\left\vert g_{t_{k_{n}}}^{-1}(X_{n}(\Lambda))\right\vert _{+}  &
=r_{n}(\Lambda)e^{\frac{c}{d}(k_{n}-\ln\varphi(k_{n}))}\\
&  \leq r_{n}(\Lambda)q_{n}^{\frac{c}{d}}(\Lambda)\varphi(k_{n})^{-\frac{c}
{d}}\\
&  \leq\varphi(n)^{\frac{d+c}{d}}\varphi(k_{n})^{-\frac{c}{d}}\\
&  =\varphi(k_{n})\times\left(  \frac{\varphi(n)}{\varphi(k_{n})}\right)
^{\frac{d+c}{d}}\\
&  \leq\varphi(k_{n})\times\left(  \frac{\varphi(n)}{\varphi((L_{d,c}
+1)n)}\right)  ^{\frac{d+c}{d}}\leq K\varphi(k_{n})
\end{align*}
for some constant $K$ depending only on $\varphi$ (we use that $\frac
{\varphi(t)}{\varphi((L_{d,c}+1)t)}$ is bounded above which is obviously true
when $\varphi(t)=\frac{1}{t^{\alpha}}$) and
\[
\left\vert g_{t_{k_{n}}}^{-1}(X_{n}(\Lambda))\right\vert _{-}=q_{n}
(\Lambda)e^{-k_{n}+\ln\varphi(k_{n})}\leq e\varphi(k_{n})\leq K\varphi
(k_{n}).
\]
Thus there are infinitely many $n$ such that $\Lambda\in B_{k_{n}}$. Since the
sequence $(k_{n})_{n}$ goes to infinity, $\Lambda\in\mathcal{B}$. It follows
that for almost all lattices $\Lambda$, $q_{n}^{c}(\Lambda)r_{n}^{d}
(\Lambda)\geq\varphi(n)$ for $n$ large enough and we are done.
\end{proof}

As an immediate consequence of the previous proposition and of Theorem \ref{thm:Birkhoff}, we have
\begin{corollary}
\label{cor:levy-d}
\[
L_{d,c}=\frac{d}{\mu_{S}(S)}\int_{S}\tau~d\mu_{S}=\frac{d\times\mu
(\mathcal{L}_{d+c})}{\mu_{S}(S)}.
\]

\end{corollary}

\section{Parametrization of $S$ when $c=1$ \label{sec:S}}

This section is not necessary neither for the proofs of Theorems
\ref{thm:levy-d} and \ref{thm:jager}, nor for Sections \ref{sect:reduction} and \ref{sec:qnrn}.
The aim is to show that the computation of the constant $L_{d,c}$ is theoretically feasible in the case $c=1$. However, if the case $d=1$ is easy (see below), the case $d=2$ is already difficult. It is possible to give an integral formula for $L_{2,1}$. But, we are not able to compute the integral, only a numerical estimation of the integral has been carried out. 
An exact description of $S$ when $d\geq 3$ seems to be rather difficult.
In this section we assume $c=1$.
\subsection{$rkN$ decomposition }

In this subsection we give a parametrization of a set of lattices that contains $S$.

Let $\Lambda$ be a lattice in $S$ and let $u_1=v_{0}^S(\Lambda)$ and
$u_2=v_{1}^S(\Lambda)$ be the two vectors associated with $\Lambda$ by the
definition of $S$ (see the definition of $S$ in Section \ref{sec:defS}). When $d\geq 2$, we suppose these two
vectors have non negative heights and when $d=1$, we only suppose that $u_2$ has a non negative height. Since $u_1$ and $u_2$ are independent shortest vectors, by Corollary \ref{cor:lattice}, they
 are the first two vectors of a basis of $\Lambda$. Thus, there is a
matrix $M\in \SL(d+1,\mathbb{R)}$ defining $\Lambda$ the first two columns of
which are the vectors $u_{1}$ and $u_{2}$. 

When $d=1$, using the scaling factor $r=\left\vert u_{1}\right\vert_{+}=\left\vert u_2\right\vert_{-}>0$, we can write
$
M =rN,
$
where $N$ is in the set $U_1$ of $2\times 2$ matrices such that
\begin{align*}
n_{1,1}=n_{2,2}=1 \text{ and }
\left\vert n_{2,1}\right\vert,\,\left\vert n_{1,2}\right\vert&<1.
\end{align*}

When $d\geq 2$, let us denote $(e_1,e_2,...,e_{d+1})$ the standard basis of $\mathbb R^{d+1}$. Using the same scaling factor 
$r=\left\vert u_{1}\right\vert_{+}=\left\vert u_2\right\vert_{-}>0$
and an orthogonal matrix $k$ that fixes $e_{d+1}$  and sends $e_{1}$ to $\frac{1}{r}u_{1,+}$, we can find a matrix
$N=(n_{i,j})_{1\leq i,j\leq d+1}$ such that $ M    =rkN$,  $\det N    >0$ and
\begin{gather}
n_{1,1}=n_{d+1,2}=1>n_{d+1,1}=\left\vert u_1\right\vert_{-}\geq 0,\label{decomposition1}\\
\left\Vert (n_{1,2},...,n_{d,2})\right\Vert _{\mathbb{R}^{d}}<1,
\label{decomposition2}\\
n_{2,1}=...=n_{d,1}=0. \label{decomposition3}
\end{gather}

When $d\geq 2$, $k$ is chosen in the group
\[
K_{d}=\{k\in \SO(d+1):ke_{d+1}=e_{d+1}\}.
\]
and using the decomposition of a $d\times d$ matrix in a product of an orthogonal 
matrix with positive determinant and of an upper triangular matrix, we can even suppose that
\begin{equation}
n_{i,j}=0,\text{\ for all }1\leq j<i\leq d. \label{decomposition4}
\end{equation}

For $d\geq 2$, let us denote $U_d$ the set of $(d+1)\times(d+1)$-matrices such that
(\ref{decomposition1}), (\ref{decomposition2}),  and
(\ref{decomposition4}) hold ((\ref{decomposition4} implies (\ref{decomposition3})). 

Since $\det M=1$, the scaling factor $r$ must be equal to $(\det N)^{-\frac{1}{d+1}}$.
Puting $K_1=\{I_2\}$, for all $d\geq 1$, the map
\[
(k,N)\in K_d\times U_d\rightarrow (\det N)^{-\frac{1}{d+1}}kN
\]
provides a natural parametrization of a subset $\Sigma$ in
$\SL(d+1,\mathbb{R)}$ whose  projection
in $\mathcal{L}_{d+1}$  contains $S$. The main problem is now to find
which of these couples $(k,N)$ are such that $rkN\mathbb{Z}^{d+1}\in S$ and
to select a fundamental domain in this set of couples. This problem reduces to
\begin{itemize}
	\item finding the set of matrices $N\in U_d$ such that the first two columns $u_{1}$ and $u_{2}$ of $N$ are in the unit ball
	$B_{d,1}(0,1)$ and are the only nonzero vectors of the lattice
	$N\mathbb{Z}^{d+1}$ in this ball,
	\item then select a fundamental domain in this set of matrices $N$.
\end{itemize}
This is easy when $d=1$ and doable when $d=2$. When
$d=1$, it is even possible to find the first return map $R$.

Another issue is to find the measure $\mu_{S}$ on $S$ induced by the flow
$g_{t}$ and the invariant measure $\mu$ of $\mathcal{L}_{d+1}$. This comparatively easier
issue can be performed for all $d$ without knowing explicitly $S$.

\subsection{The induced measure $\mu_S$}

Consider the manifold $V_{d}=\mathbb{R}_{>0}
\times\mathbb{R\times}K_{d}\times U_{d}$ and the  submanifold
\[
W=\{(\Delta,t,k,N)\in V_{d}: \Delta=1,\,t=0\}=\{1\}\times\{0\}\times
K_{d}\times U_{d}
\]
together with the map $F  :V_{d}\rightarrow \GL(d+1,\mathbb{R})$ defined by
\begin{align*}
F\left(\Delta,t,k,N\right)= \left(\frac{\Delta}{\det N}\right)
^{\frac{1}{d+1}}g_tkN
\end{align*}
and  the map $\overline F:V_d\rightarrow\GL(d+1,\mathbb R)/\SL(d+1,\mathbb Z)$ defined by
\[
\overline F(\Delta,t,k,N)=F(\Delta,t,k,N)\mathbb Z^{d+1}.
\]
By the discussion of the previous subsection, $\overline F$ provides a parametrization of $S$:
$
S\subset \overline F(W).
$
We would like to compute the measure $\mu_S$ in the coordinates $(1,0,k,N)$. The submanifold $W$ is equipped with the reference measure 
$
\mu_{K_d}\otimes \lambda_{U_d}
$
where $\lambda_{U_d}$ is the Lebesgue measure on $U_d$ and  $\mu_{K_d}$ is the invariant measure on $K_d$ associated with the invariant volume form $\gamma$
on $K_{d}$ that is dual to the exterior product of 
the	invariant vector fields generated by the standard skew symmetric matrices
$(A_{i,j}=E_{ij}-E_{ji})_{1\leq j<i\leq d}$.

We have to precise the normalization of the Haar measure that we use on $\SL(d+1,\R)$.  Consider the standard volume form $\omega$ on $\M_{d+1}(\mathbb{R)}$ seen as $\mathbb{R}^{(d+1)^{2}}$ and the
vector fields $X$ and $Y$ on $\M_{d+1}(\mathbb{R)}$ defined by
\[
Y(M)=M \text{ and } X(M)=\frac{\partial(g_{t}M)}{\partial t}\Big|_{t=0}.
\]
The invariant measure in $\SL(d+1,\mathbb{R)}$ is given by the volume form
$\frac{1}{(d+1)^2}i_{Y}\omega$. With this
normalization of the volume form we have the beautiful Siegel formula
\begin{equation*}
\mu(\mathcal L_{d+1})=\operatorname*{vol}(\SL(d+1,\mathbb{R})/\SL(d+1,\mathbb{Z))=}\frac{1}{d+1}\prod_{k=2}
^{d+1}\zeta(k) 
\end{equation*}
(see \cite{Siegel}, p. 152). The induced measure $\mu_{S}$ on $S$ is
given by the restriction to $S$ of the projection in the space of lattices of the differential form
\begin{equation*}
\frac{1}{(d+1)^2}i_{X}i_{Y}\omega. 
\end{equation*}
Using this differential form it is possible to compute the induced measure $\mu_S$ with the parametrization $\overline{F}$, we give without proof an explicit formula in next proposition.
\begin{proposition} \label{prop:invariant-measure}
 Assume $d \geq 2$. Suppose that $\mathcal D$ is an open subset of $W$ such that $\overline F(\mathcal D)\subset S$ and the restriction  of $\overline F$ to $\mathcal D$ is one to one. Then the image by $\overline F$ of the measure 
\[
1_{\mathcal{D}}\,\left(\frac{1}{\det N}\right)^{d+1}\left(
	{\displaystyle\frac{1}{d+1}\prod\limits_{j=2}^{d-1}}
	n_{j,j}^{d-j}\right)\,\mu_{K_d}\otimes \lambda_{U_d}
\]
is the restriction of $\mu_S$ to $\overline F(\mathcal D)$.
\end{proposition}

\subsection{Determination of the transversal $S$, $c=1,\,d=1$ \label{sub:d=1}}
We already have a map 
\begin{align*}
& U_1\rightarrow \mathcal L_2\\
& N\rightarrow (\det N)^{-\frac{1}{2}}N\mathbb Z^2
\end{align*}
that sends $U_1$ onto a set that contains $S$.
By definition of $S$, the image of a matrix
\[
N=\left(
\begin{array}
[c]{cc}
1& n_{1,2} \\
n_{2,1}& 1
\end{array}
\right)\in U_1
\]
is in $S$ iff the only nonzero vectors of the lattice $\Lambda=N\mathbb Z^2$ in the ball $B_{1,1}(0,1)$ are the two columns of $N$ up to sign.  We obtain that  $\Lambda\in S$ iff
\begin{itemize}
	\item $0<\left\vert n_{2,1}\right\vert,\,\left\vert n_{1,2}\right\vert <1$,
	\item the signs of $n_{1,2}$ and $n_{2,1}$ are opposite.
\end{itemize}
So the map $F$ defined on $]0,1[^2\times\{-1,1\}$ defined by 
\[
(x,y,\varepsilon)\rightarrow \frac{1}{(1+xy)^{1/2}}\left(
\begin{array}
[c]{cc}
1& -\varepsilon x \\
\varepsilon y& 1
\end{array}
\right)\mathbb Z^2
\]
provide a parametrization of $S$ and it is easy to see that $F$ is a bijection. An easy computation of a $4\times 4$ determinant gives that 
\[
i_Xi_Y\tfrac14\omega(\tfrac{\partial F}{\partial x},\tfrac{\partial F}{\partial y})=\tfrac14\omega(X,Y,\tfrac{\partial F}{\partial x},\tfrac{\partial F}{\partial y})=\frac{1}{2(1+xy)^2},
\]
hence
\[
f(x,y,\varepsilon)=\frac{1}{2(1+xy)^2}
\]  
is the density the measure $\mu_S$ with respect to the Lebesgue measure.
Therefore $\mu_S(S)=\ln 2$.
With the Siegel formula (\cite{Siegel}, p. 152) and Corollary  \ref{cor:levy-d}, we obtain the L\'evy's constant
\[
L_{1,1}=\frac{\mu(\mathcal{L}_{d+c})}{\mu_{S}(S)}=\frac{\zeta(2)}{2\ln 2}
=\frac{\pi^2}{12\ln 2}.
\] 

\subsubsection{Determination of the first return map, $c=1,\ d=1$}
Let 
\[
\Lambda=F(x,y,\varepsilon)=\frac{1}{(1+xy)^{1/2}}\left(
\begin{array}
[c]{cc}
1& -\varepsilon x \\
\varepsilon y& 1
\end{array}
\right)\mathbb Z^2
\]
be in $S$.   By Lemma \ref{lem:return-minimal}, to find the first return map $R(\Lambda)$ in $S$, it is enough to find the minimal vector $X_2(\Lambda)$. Then $R(\Lambda)$ is given by 
$g_{\tau(\lambda)}(\Lambda)$ with 
\[
\tau(\Lambda)=\frac{1}{2}\ln\left(\frac{\left\vert X_{2}(\Lambda)\right\vert_{-}}
{\left\vert X_{1}(\Lambda
	)\right\vert_{+}}\right).
\]
By corollary \ref{cor:lattice}, the first minimal vectors
$X_0(\Lambda)$ and $X_1(\Lambda)$ form a basis of $\Lambda$. The
minimal vector $X_2(\Lambda)$ is the vector of the form
$X=aX_0(\Lambda)+bX_1(\Lambda)$  in the strip $\left\vert X
\right\vert_{+}<\left\vert X_1(\Lambda)\right\vert_{+}=x$
with $a,b\in \mathbb Z$, and with the smallest height. It is not difficult to see that 
\[
X_2(\Lambda)=\varepsilon X_0(\Lambda)+\lfloor\tfrac{1}{x}\rfloor X_1(\Lambda).
\]
So we obtain $R(\Lambda)=F(x',y',\varepsilon')$ where
\begin{align*}
\varepsilon'=-\varepsilon ,\,\,
x'=\{\tfrac{1}{x}\} \text{ and }
y'=\frac{1}{y+\lfloor\tfrac{1}{x}\rfloor}.
\end{align*}
It follows that the return map $R$ is a two-fold extension of the natural extension of the Gauss map.

\subsection{Value of L\'evy's constant when $d=2$ and $c=1$, \label{sub:d=2}}
An exact description of $S$ is possible when $d=2$ and $c=1$. Together with the expression of the measure $\mu_S$ in Proposition \ref{prop:invariant-measure}, this leads to a closed formula for L\'evy's constant as a seven-tuple integral of an algebraic function over a union of domains the boundaries of which are algebraic surfaces of degree at most two. We are not able to compute this seven-tuple integral. However using Octave, Seraphine Xieu (see \cite{Xieu},\cite{CheCheComp}) has computed a numerical approximation of Levy's constant
\[
L_{2,1}=1.135256974\dots
\]
This can be compared with the one dimensional Levy's constant
\[
L_{1,1}=1.186569111\dots
\]

\section{Almost sure convergence in M$_{d,c}(\mathbb{R)}\label{sect:reduction}
$}
We want to prove that   $\lim_{n\rightarrow\infty}\frac{1}{n}\ln q_{n}(\Lambda_{\theta})=L_{d,c}$ for Lebesgue-almost all  $\theta\in M_{d,c}(\R)$. By Theorem \ref{thm:Birkhoff}, we already know that $\lim_{n\rightarrow\infty}\frac{1}{n}\ln q_{n}(\Lambda)=L_{d,c}$ for $\mu$-almost all lattices $\Lambda\in\mathcal L_{d+c}$, but we have to change the  reference measure. Furthermore,   the set $\mathcal H_>$ of lattices of the shape $\Lambda_{\theta}$, $\theta\in M_{d,c}(\R)$, is $\mu$-negligible. Actually, we will first prove a general result,  Theorem \ref{thm:reduction} (see below): the Birkhoff sums $\frac{1}{N}\sum_{n=0}^{N-1}\varphi\circ R^{n}(\Lambda_{\theta})$ converge for almost all $\theta$ when the function $\varphi$ satisfies some regularity assumptions. The main ingredient of the proof of this general result is that $\mathcal H_>$ contains the expanding direction of the flow $g_t$.   However, Theorem \ref{thm:reduction} also requires that Lemma \ref{lem:boundary} below holds for the transversal. This is proven in Subsections \ref{sec:example} and \ref{sec:example2} with an example. These two subsections are not necessary for the proof of Theorem \ref{thm:reduction}.  

\subsection{A general result}

Recall that $\mathcal{H}_{\leq}$ is the subgroup of $\SL(d+c,\mathbb{R)}$
defined by
\[
\mathcal{H}_{\leq}=\{h\in\SL(d+c,\mathbb{R)}:h=\left(
\begin{array}
[c]{cc}
A & 0\\
B & C
\end{array}
\right)  \}
\]
with $ A\in\GL(d,\mathbb{R)}$, $B\in\M_{c,d}(\mathbb{R)}$ and $C\in
\GL(c,\mathbb{R})$. We say that a function $f:\mathcal{L}_{d+c}\rightarrow
\mathbb{R}$ is \textit{uniformly continuous in the }$\mathcal{H}_{\leq }
$\textit{-direction} if for all $\varepsilon$ there exists $\beta>0$ such that
for all $\Lambda\in\mathcal{L}_{d+c}$ and all $h\in B_{\mathcal{H}_{\leq }}
(I_{d+c},\beta)$, $\left\vert f(h\Lambda)-f(\Lambda)\right\vert \leq\varepsilon$.

\begin{theorem}
\label{thm:reduction} 1. Let $\varphi:S\rightarrow\mathbb{R}$ be a function
continuous almost everywhere on $S$. Suppose there exists a non negative function $f:\mathcal{L}_{d+c}\rightarrow\mathbb{R}_{\geq 0}$ that is continuous,
uniformly continuous in the $\mathcal{H}_{\leq}$-direction, integrable and
  such that $\left\vert
\varphi\right\vert \leq f$ on $S$. Then,
\[
\int_{S}fd\mu_{S}<+\infty
\]
and for almost all $\theta$ in $\M_{d,c}(\mathbb{R})$,
\[
\lim_{n\rightarrow\infty}\frac{1}{n}\sum_{k=0}^{n-1}\varphi\circ R^{k}
(\Lambda_{\theta})=\frac{1}{\mu_{S}(S)}\int_{S}\varphi d\mu_{S}.
\]
2. The same result holds for $S^{\prime}$ instead of $S$
\end{theorem}

We
can formulate Theorem \ref{thm:reduction} for  a general transversal $S$. The assumptions about $S$ are:

\begin{itemize}
\item $S$ is a co--dimension one submanifold transverse to the flow,

\item the number of visiting times in a time interval of length $1$ is bounded
above by a universal constant $A$ (Lemma \ref{lem:nb-returntimes}),

\item Lemma \ref{lem:boundary} below holds for $S$.
\end{itemize}

The other assumptions and the conclusion are the same as in Theorem \ref{thm:reduction}.

\medskip

For a compact subset $K$ of the submanifold $S$ and $\delta>0$, let us denote
\[
U(K,\delta)=\{g_{t}h\Lambda:t\in[0,1],\ h\in B_{\mathcal{H}_{\leq}}
(I_{d+c},\delta),\ \Lambda\in S\setminus K\}.
\]

\begin{lemma}
\label{lem:boundary} For all $\varepsilon>0$, there exist a compact subset
$K$ in $S$ and $\delta>0$ such that $\mu(U(K,\delta))\leq\varepsilon$.
\end{lemma}

 This is an important lemma because it explains that the
part of $S$ near its \textquotedblleft boundary\textquotedblright\ is not
relevant. This lemma  holds  for $S'$ as well but it will be proven below only for $S$.

We shall prove first a particular case of Theorem \ref{thm:reduction} which also requires Lemma \ref{lem:boundary}:

\begin{proposition}
\label{prop:reduction} Let $\varphi:S\rightarrow\mathbb{R}$ be a bounded
continuous function. Then for almost all $\theta$ in $\M_{d,c}(\mathbb{R})$,
\[
\lim_{n\rightarrow\infty}\frac{1}{n}\sum_{k=0}^{n-1}\varphi\circ R^{k}
(\Lambda_{\theta})=\frac{1}{\mu_{S}(S)}\int_{S}\varphi d\mu_{S}.
\]

\end{proposition}

\subsubsection{Auxiliary lemmas}
The proof of Theorem \ref{thm:reduction} needs three more lemmas.
The first one only uses  that $S$ and $S'$ are transverse to the flow together with the inverse function theorem. In the remaining of the section \ref{sect:reduction}, the balls $B(I_{d+c},r)$ are in $\SL(d+c,\R)$.
\begin{lemma}
	\label{lem:K-S} For all compact subset $K$ in $S$ (or in $S'$), there exist $\alpha$ and
	$\eta>0$ such that 
	\newline- the map $(t,\Lambda)\rightarrow g_{t}\Lambda$ is
	one to one on $[-\alpha,\alpha]\times K$,
	\newline- for all $h\in
	B(I_{d+c},\eta)$ and all $\Lambda$ in $K$, there exists an unique
	$t=t(h,\Lambda)\in[-\alpha,\alpha]$ such that $g_{-t}h\Lambda\in S$.
	\newline- the maps $\sigma:(h,\Lambda)\rightarrow t=t(h,\Lambda)$ and $\pi
	:(h,\Lambda)\rightarrow g_{-t}h\Lambda$ are continuous on $B(I_{d+c}
	,\eta)\times K$ and the values of $\sigma$ are in $[-\alpha/4,\alpha/4]$.
\end{lemma}

\begin{proof}
Let $\psi:(t,\Lambda)\in\R\times S\rightarrow g_t\Lambda\in\mathcal L_{d+c}$.	
Since $S$ is a submanifold and since the flow is transverse to $S$, thanks to the  inverse function theorem, we see that for each $\Lambda\in S$, there exists 	an open neighborhood $W_{\Lambda}$ of $\Lambda$ in $S$ and $a_{\Lambda}>0$ such that $\psi$ is a diffeomorphism from $]-a_{\Lambda},a_{\Lambda}[\times W_{\Lambda}$ onto the open set $\psi(]-a_{\Lambda},a_{\Lambda}[\times W_{\Lambda})$ in $\mathcal L_{d+c}$ and such that
$
\psi(]-a_{\Lambda},a_{\Lambda}[\times W_{\Lambda})\cap S=W_{\Lambda}
$. Since $K$ can be covered  by finitely many $W_{\Lambda}$, 	we can find an open set $W$ in $S$ containing $K$ and $\alpha>0$ such that 
\[
\psi(]-2\alpha,2\alpha[\times W)\cap S=W.
\]
We can suppose that  $W$ is relatively compact w.l.o.g.. Since $\psi$ is one to one on the compact set $\{0\}\times \overline W$, reducing $\alpha$, we can suppose that $\psi$ is one to one on $]-2\alpha,2\alpha[\times\overline W$. It follows that $\psi$ is a diffeomorphism from $U=]-2\alpha,2\alpha[\times W$ onto $\psi(U)$.
We can choose $\eta>0$ such that $B(I_{d+c},\eta)K\subset \psi([-\alpha/4,\alpha/4]\times W)$. With this choice of $\eta$, for each $(h,\Lambda)\in B(I_{d+c},\eta)\times K$, there exists  a pair $(t,\Gamma)\in [-\alpha/4,\alpha/4]\times W$ such that $g_t\Gamma=\psi(t,\Gamma)=h\Lambda$.
If $(t',\Gamma')\in [-\alpha,\alpha]\times S$ is such that $g_{t'}\Gamma'=h\Lambda=g_t\Gamma$, then $\Gamma'=g_{t-t'}\Gamma\in \psi(]-2\alpha,2\alpha[\times W)\cap S=W$. Hence $\Gamma'\in W$ and since $\psi$ is one to one, $t'-t=0$ and $\Gamma'=\Gamma$.
Finally, $(\sigma(h,\Lambda),\pi(h,\Lambda))=(\psi_{|_U})^{-1}(h\Lambda)$ is a continuous function of $(h,\Lambda)\in B(I_{d+c},\eta)\times K$.
\end{proof}

 The second lemma is
a purely measure-theoretic result.

\begin{lemma}
	\label{lem:measure-image} Let $X$ and $Y$ be locally compact second countable
	metric spaces. Let $\mu_{X}$ and $\mu_{Y}$ be two measures on $X$ and $Y$
	finite on compact subsets. Suppose $\psi:X\rightarrow Y$ is a continuous map
	such that every $y$ in $Y$ has at most $N$ preimages and such that for all $x$ in $X$
	there exists a compact neighborhood $\omega_{x}$ of $x$ with the following
	property:
	\newline- $\psi$ is one to one on $\omega_{x}$,
	\newline- the image by
	$\psi$ of the measure $1_{\omega_{x}}\mu_{X}$ is the measure $1_{\psi
		(\omega_{x})}\mu_{Y}$. 
	\newline Then for all nonnegative measurable functions
	$f:Y\rightarrow\mathbb{R}$,
	\[
	\int_{X}f\circ\psi\, d\mu_{X}\leq N\int_{Y}f\,d\mu_{Y}.
	\]
	
\end{lemma}

\begin{proof}
Since $X$ is a countable union of
compact sets, making use of the monotone convergence theorem, we can reduce to $X$
compact. Next, since for any positive Radon measure $\mu$ on $Y$ and  for any Borel set $B\subset Y$, $\mu(B)=\sup\{\mu(K):K\text{ compact }\subset B\}$, we can assume that $f=1_{K}$ is
the indicator function of a compact subset $K$ of $Y$. With these assumptions, $L=\psi^{-1}(K)$ is a
compact subset of $X$ and therefore is included in a finite union of
$\omega_{x}$, 
$
L\subset\cup_{i=1}^{n}\omega_{x_{i}}.
$
For each $i$, let $\omega_{x_{i}}^{\prime}=\omega_{x_{i}}\setminus\cup
_{j<i}\omega_{x_{j}}$, $L_{i}=\omega_{x_{i}}^{\prime}\cap L$ and $K_{i}
=\psi(L_{i})$. Since $L$ is the disjoint union of the $L_{i}$, $\mu
_{X}(L)=\sum_{i=1}^{n}\mu_{X}(L_{i})$. Since $L_{i}$ is included in
$\omega_{x_{i}}$, $\mu_{X}(L_{i})=\mu_{Y}(K_{i})$. Since the $L_{i}$ are
disjoint, the number of $\psi$-preimage of any $y$ is at least $F(y)=\sum
_{i=1}^{n}1_{K_{i}}(y)$ therefore $F\leq N$. It follows that
\begin{align*}
N\mu_{Y}(K)  &  \geq\int_{K}F(y)d\mu_{Y}(y)=\sum_{i=1}^{n}\mu_{Y}(K_{i})
=\sum_{i=1}^{n}\mu_{X}(L_{i})=\mu_{X}(L).
\end{align*}
\end{proof}

The last lemma is an easy consequence of the previous lemma and of  the definition of the induced measure $\mu_S$.

\begin{lemma}
	\label{lem:induce-measure} Let $U$ be a Borel subset in $\mathcal{L}_{d+c}$
	such that for all $\Lambda$ in $U$, $g_{t}\Lambda\in U$ for all $t$ in a time
	interval $I_{\Lambda}$ of length $1$ containing $0$. Then
	\[
	\mu_{S}(U\cap S)\leq 4A\mu(U)
	\]
	where $A$ is the maximum number of entrance times in $S$ of a flow trajectory
	during a time interval of length $1$ (see Lemma \ref{lem:nb-returntimes}).
\end{lemma}

The constant $4A$ is certainly not the best one.

\begin{proof}
Take $X=[-1,1]\times S$,
$Y=\mathcal{L}_{d+c}$, $\psi:X\rightarrow Y$ the map
defined by $\psi(t,\Lambda)=g_{t}\Lambda$, the measures $\mu_{X}=dt\otimes
\mu_{S}$ and $\mu_{Y}=\mu$ and the function $f=1_{U}$. By definition of the
induced measure, we know that the image of the restriction of $dt\otimes
\mu_{S}$ to a small enough compact neighborhood $\omega$ of any $x\in X$ is the restriction to
$\psi(\omega)$ of the invariant measure $\mu$ on $\mathcal{L}_{d+c}$.
Therefore by Lemma \ref{lem:measure-image},
\[
N\mu(U)\geq\int_{S}\int_{-1}^{1}1_{U}(g_{t}\Lambda)dt\ d\mu_{S}
(\Lambda)
\]
where $N$ is the maximum number of preimages. By assumption $\int_{-1}
^{1}1_{U}(g_{t}\Lambda)dt\geq 1$ for all $\Lambda$ in $U$ and it is easy to see
that $N\leq 4A$. Therefore $4A\mu(U)\geq\mu_{S}(U\cap S)\times 1$.
\end{proof}

\subsubsection{An example}\label{sec:example}

We want to construct a co-dimension one submanifold $V$ in $\mathcal{L}_{d+c}$
transverse to the flow $g_{t}$ together with a bounded continuous function
$\varphi:V\rightarrow\mathbb{R}$ such that for a set  of
$\theta\in M_{d,c}(\R)$ of positive measure, \ the sequence $\frac{1}{n}\sum_{k=0}^{n-1}\varphi\circ R_{V}
^{k}(\Lambda_{\theta})$ does not converge to $\frac{1}{\mu_{V}(V)}\int
_{V}\varphi d\mu_{V}$ where $R_{V}$ is the first return map in $V$ and
$\mu_{V}$ is the invariant measure induced by the flow. The idea is the
following. Take $V$ an open set in $S$. Then $\mu_{V}$ is the restriction of
$\mu_{S}$ to $V$. Suppose that the open set $V$ can be chosen in order that
for all $\theta\in M_{d,c}(\R)$, and all $k\geq 1$,$\ $
\[
R_{V}^{k}(\Lambda_{\theta})=R_{S}^{k}(\Lambda_{\theta})\,(=R^{k}(\Lambda
_{\theta})).
\]
Then if $\varphi:S\rightarrow\mathbb{R}$ is a non negative continuous function
not identically zero with support included in $V$, the sequences
\[
\frac{1}{n}\sum_{k=0}^{n-1}\varphi\circ R_{V}^{k}(\Lambda_{\theta}),\ \frac
{1}{n}\sum_{k=0}^{n-1}\varphi\circ R_{S}^{k}(\Lambda_{\theta})
\]
converge to the same limit which cannot be equal to both $\frac{1}{\mu_{S}
(S)}\int_{S}\varphi d\mu_{S}$ and $\frac{1}{\mu_{V}(V)}\int_{V}\varphi
d\mu_{V}=\frac{1}{\mu_{S}(V)}\int_{S}\varphi d\mu_{S}$ provided that $\mu
_{S}(V)<\mu_{S}(S)$. So we are reduced to constructing $V$.

Observe that Theorem \ref{thm:reduction} implies that for such a $V$, $\varphi=1_{V}$
is not almost everywhere continuous on $S$ which means that the boundary of
$V$ in $S$ has positive measure. Moreover, it shows that the assumption about the continuity of the function $\varphi$ in Theorem \ref{thm:reduction}, cannot be dropped.

\subsubsection{Construction of $V$}\label{sec:example2}

Consider the set $\mathbb{T}$ of lattices $\Lambda_{\theta}$ such that the
coefficients of $\theta$ are all in $[0,1]$. It is a compact subset in
$\mathcal{L}_{d+c}$ containing all the lattices $\Lambda_{\theta}$. Denote
$W_{\varepsilon}$ the open ball $B(I_{d+c},\varepsilon)$ in $\SL(d+c,\mathbb{R)}$.
We consider the open sets
\[
U_{n}(\varepsilon)=\bigcup_{t\in[n,n+1]}g_{t}(W_{\varepsilon
}\mathbb{T})
\]
and for a sequence $(\varepsilon_{n})_{n\in\mathbb{N}}$ of positive real
numbers, we consider the open set
\[
U=U((\varepsilon_{n})_{n\in\mathbb{N}})=
\bigcup_{n\in\mathbb{N}}
U_{n}(\varepsilon_{n}).
\]
Take $V=S\cap U$. For all $t\geq 0$ and all $\theta$, $g_{t}\Lambda_{\theta
}=g_{t}I_{d+c}\Lambda_{\theta}$ is in $U$, hence for all $k\in\mathbb{N}$, we have
\[
R_{S}^{k}(\Lambda_{\theta})=R_{V}^{k}(\Lambda_{\theta}).
\]
So we are reduced to showing that when the sequence $(\varepsilon_{n})_{n}$ is
small enough,
\[
\mu_{S}(V)<\mu_{S}(S).
\]

By definition of $U_{n}(\varepsilon)$, if $\Lambda=g_{t}g\Lambda_{\theta}$ with
$t\in[n,n+1]$, $g\in W_{\varepsilon}$ and $\Lambda_{\theta}\in\mathbb{T},$ then
\[
g_{s}g_{t}g\Lambda_{\theta}\in U_{n}(\varepsilon)
\]
for all $s$ in the interval $[n-t,n+1-t]$. So $U$ satisfies the assumption of the
Lemma \ref{lem:induce-measure} and therefore
\[
\mu_{S}(U\cap S)\leq 4A\sum_{n\in\mathbb{N}}\mu(U_{n}(\varepsilon_{n})).
\]
Using that $g_{t}g\Lambda=g_{n}(g_{t-n}gg_{-(t-n)})g_{t-n}\Lambda$, we see
that
\begin{align*}
U_{n}(\varepsilon)  &  =\{g_{t}\Lambda:t\in[n,n+1],\ \Lambda\in W_{\varepsilon
}\mathbb{T}\}\\
&  \subset g_{n}W_{\varepsilon^{\prime}}\{g_{s}\Lambda:s\in[0,1],\ \Lambda
\in\mathbb{T\}}
\end{align*}
where $\varepsilon^{\prime}$ is such that $g_{s}W_{\varepsilon}g_{-s}\subset
W_{\varepsilon^{\prime}}$ for all $s$ in $[0,1]$. Furthermore, the compact set
$\{g_{s}\Lambda:s\in[0,1],\ \Lambda\in\mathbb{T\}}$ has zero measure because
it has dimension $cd+1$ which is $<(c+d)^{2}-1$. Therefore
\[
\lim_{\varepsilon^{\prime}\rightarrow 0}\mu(W_{\varepsilon^{\prime}}
\{g_{s}\Lambda:s\in[0,1],\ \Lambda\in\mathbb{T}\})=0,
\]
which implies
\[
\lim_{\varepsilon\rightarrow 0}\mu(U_{n}(\varepsilon))=0.
\]
So there exists a sequence $(\varepsilon_{n})_{n\in\mathbb{N}}$ such that
\[
\sum_{n\in\mathbb{N}}\mu(U_{n}(\varepsilon_{n}))<\frac{1}{4A}\mu(S)
\]
and for such a sequence, we have $\mu_S(V)<\mu_S(S)$.

\subsection{Proof of Proposition \ref{prop:reduction}}

Let $\varphi:S\rightarrow\mathbb{R}$ be a bounded continuous  function.

Let $\varepsilon$ be a positive real number, let $K$ and $\delta$ be
associated with $\varepsilon$ by Lemma \ref{lem:boundary}, and $\alpha$ and
$\eta$ associated with $K$ by Lemma \ref{lem:K-S}.

{\it Preliminary observation. }
Let $(a_{n})_{n\in\N}$ be a decreasing sequence of reals numbers in $]0,\eta[$
tending to zero and let $L_{a_{n}}=B(I_{d+c},a_{n})\times K$. Since
the intersection of all the compact sets $L_{a_{n}}$, $n\in\N$, is $L_{0}=\{I_{d+c}\}\times
K$ and since the map $\psi(g,\Lambda)=\varphi(\pi(g,\Lambda))-\varphi
(\Lambda)$ is continuous, we have
\[
\cap_{n\geq 0}\psi(L_{a_{n}})=\psi(\cap_{n\geq 0}L_{a_{n}})=\psi(L_{0})=\{0\}.
\]
Therefore, for $n$ large enough, $\psi(L_{a_{n}})\subset]-\varepsilon
,\varepsilon[$ which implies there exists $\beta>0$ such that for all
$\Lambda\in K$ and all $g\in B(I_{d+c},\beta)$,
\begin{equation}
\left\vert \varphi(\pi(g,\Lambda))-\varphi(\Lambda)\right\vert \leq
\varepsilon. \label{uniform}
\end{equation}

Finally, let $\gamma>0$ be such that for all $s\geq 0$ and all $h\in
B_{\mathcal{H}_{\leq}}(I_{d+c},\gamma)B_{\mathcal{H}_{\leq}}^{-1}
(I_{d+c},\gamma)$, 
\[
\dd(g_{s}hg_{-s},I_{d+c})\leq\min(\delta,\beta,\eta).
\]

For $T\geq 0$, $\Lambda\in\mathcal L_{d+c}$, and $E\subset S$, denote
\[
I(T,\Lambda,E)=\{t\in[0,T]:g_{t}\Lambda\in E\}.
\]

For almost all $\theta\in M_{d,c}(\R)$, we can fix $h_{\theta}\in B_{\mathcal{H}_{\leq}
}(I_{d+c},\gamma)$ such that the conclusion of the Birkhoff ergodic theorem holds for the flow $g_t$ or the first return map in $S$ and the 
lattice $h_{\theta}\Lambda_{\theta}$. Observe that $h_{\theta}=h_{\theta,\varepsilon}$ 
depends on $\varepsilon$.
It is understood that we shall use the Birkhoff ergodic theorem in  countably many 
situations. We fix a sequence $(\varepsilon_{n})$ going to zero and for each
$\varepsilon=\varepsilon_{n}$, we use three times Birkhoff ergodic theorem and the ergodicity of the flow: for
almost all $\theta$,
\begin{equation}
\lim_{T\rightarrow\infty}\frac{1}{T}\int_{0}^{T}1_{U(K,\delta)}(g_{t}
h_{\theta}\Lambda_{\theta})dt
=\frac{1}{\mu(\mathcal{L}_{d+c})}\mu(U(K,\delta))\leq\varepsilon,
\label{birkhoff1}
\end{equation}
\begin{equation}
\lim_{T\rightarrow\infty}\frac{1}{T}\operatorname{card}I(T,h_{\theta}
\Lambda_{\theta},S)=\frac{1}{\mu_{S}(S)}\int_{S}\tau_S d\mu_{S}=\frac
{\mu(\mathcal{L}_{d+c})}{\mu_{S}(S)}, \label{birkhoff2}
\end{equation}
where $\tau_S$ is the first return time in $S$,
and
\begin{equation}
\lim_{T\rightarrow\infty}\frac{1}{\operatorname*{card}I(T,h_{\theta}
\Lambda_{\theta},S)}\sum_{t\in I(T,h_{\theta}\Lambda_{\theta},S)}\varphi
(g_{t}h_{\theta}\Lambda_{\theta})=\frac{1}{\mu_{S}(S)}\mu_{S}(\varphi).
\label{birkhoff3}
\end{equation}

Let $T$ be positive and let $s_{1}<...<s_{m}$ be the elements of $I(T,\Lambda_{\theta},S\setminus K)$,
we have
\[
g_{t}h_{\theta}\Lambda_{\theta}=g_{t-s_{i}}(g_{s_{i}}h_{\theta}g_{-s_{i}
})g_{s_{i}}\Lambda_{\theta}\in U(K,\delta)
\]
for all $t\in[ s_{i},s_{i}+1]$ and we can extract a subsequence $s_{n_{1}
},...,s_{n_{p}}$ defined by $n_{1}=1$ and $n_{i+1}=\min\{j:s_{j}\geq s_{n_{i}
}+1\}$. Now by Lemma \ref{lem:nb-returntimes}, there is an absolute constant $A$ such that there are at most
$A$ elements of $I(T,\Lambda_{\theta},S\setminus K)$ ($\subset I(T,\Lambda
_{\theta},S)$) in an interval of length $1$, hence $Ap\geq m$. Therefore, by
(\ref{birkhoff1})
\[
\frac{m}{A}\leq p\leq\int_{0}^{T+1}1_{U(K,\delta)}(g_{t}h_{\theta}
\Lambda_{\theta})dt\leq 2(T+1)\varepsilon
\]
and hence
\begin{equation}
\operatorname*{card}I(T,\Lambda_{\theta},S\setminus K)\leq 3AT\varepsilon
\label{nbvisitingSK1}
\end{equation}
for $T$ large enough: $T\geq T(\Lambda_{\theta},\varepsilon)$. We will also
need to bound above the number of elements in $I(T,h_{\theta}\Lambda_{\theta
},S\setminus K)$. Making use of (\ref{birkhoff1}), the same way of
reasoning leads to the same result
\begin{equation}
I(T,h_{\theta}\Lambda_{\theta},S\setminus K)\leq 3AT\varepsilon
\label{nbvisitingSK2}
\end{equation}
for $T\geq T(\Lambda_{\theta},\varepsilon)$.

{\it Heart of the proof. }We want to compare
\[
\Sigma_{1}=\frac{1}{\operatorname*{card}I(T,\Lambda_{\theta},S)}\sum_{t\in
I(T,\Lambda_{\theta},S)}\varphi(g_{t}\Lambda_{\theta})
\]
with
\[
\Sigma_{2}=\frac{1}{\operatorname*{card}I(T,h_{\theta}\Lambda_{\theta},S)}
\sum_{t\in I(T,h_{\theta}\Lambda_{\theta},S)}\varphi(g_{t}h_{\theta}
\Lambda_{\theta})
\]
because by (\ref{birkhoff3}), this latter sum tends to $\frac{1}{\mu_{S}(S)}
\int_{S}\varphi d\mu_{S}$ when $T$ goes to infinity. We split $\sum_{t\in
I(T,\Lambda_{\theta},S)}$ into two sums $\sum_{t\in I(T,\Lambda_{\theta},K)}$
and $\sum_{t\in I(T,\Lambda_{\theta},S\setminus K)}$. Observe that for $t\in
I(T,\Lambda_{\theta},K)$, $g_{t}h_{\theta}\Lambda_{\theta}=(g_{t}h_{\theta
}g_{-t})g_{t}\Lambda_{\theta}$ is of the form $g\Lambda$ with $g\in
B(I_{d+c},\eta)$ and $\Lambda\in K$, this allows to use Lemma \ref{lem:K-S}.
We use the notations of Lemma \ref{lem:K-S} and for $t$ in $I(T,\Lambda
_{\theta},K)$, we denote $t^{\prime}=\sigma(g_{t}h_{\theta}g_{-t},g_{t}
\Lambda_{\theta})$. By (\ref{uniform}), we have
\[
\left\vert \sum_{t\in I(T,\Lambda_{\theta},K)}\varphi(g_{t}\Lambda_{\theta
})-\sum_{t\in I(T,\Lambda_{\theta},K)}\varphi(\pi(g_{t}h_{\theta}g_{-t}
,g_{t}\Lambda_{\theta}))\right\vert \leq\varepsilon\operatorname{card}
I(T,\Lambda_{\theta},K).
\]
Now,
\[
\pi(g_{t}h_{\theta}g_{-t},g_{t}\Lambda_{\theta})=g_{-t^{\prime}}g_{t}
h_{\theta}g_{-t}g_{t}\Lambda_{\theta}=g_{t-t^{\prime}}h_{\theta}
\Lambda_{\theta}
\]
hence
\[
\left\vert \sum_{t\in I(T,\Lambda_{\theta},K)}\varphi(g_{t}\Lambda_{\theta
})-\sum_{t\in I(T,\Lambda_{\theta},K)}\varphi(g_{t-t^{\prime}}h_{\theta
}\Lambda_{\theta})\right\vert \leq\varepsilon\operatorname{card}
I(T,\Lambda_{\theta},K).
\]
Observe that the map $t\in I(T,\Lambda_{\theta},K)\rightarrow t-t^{\prime}$ is one to one because $t^{\prime}\in
[-\frac{\alpha}{4},\frac{\alpha}{4}]$ and the gap between two visiting times
of $K$ is $\geq\alpha$. Observe also that $t-t'\in
I(T,h_{\theta}\Lambda_{\theta},S)$ except possibly for the first and the last element of $t\in I(T,h_{\theta}\Lambda_{\theta},S)$. On the one hand, it follows that
\begin{multline*}
\left\vert \sum_{t\in I(T,\Lambda_{\theta},S)}\varphi(g_{t}\Lambda_{\theta
})-\sum_{t\in I(T,h_{\theta}\Lambda_{\theta},S)}\varphi(g_{t}h_{\theta}
\Lambda_{\theta})\right\vert \\
\leq\varepsilon\operatorname{card}I(T,\Lambda_{\theta},K)\\
+\left\Vert \varphi\right\Vert _{\infty}(\operatorname{card}I(T,\Lambda
_{\theta},K\setminus S)+\operatorname*{card}I(T,h_{\theta}\Lambda_{\theta
},S)-\operatorname*{card}I(T,\Lambda_{\theta},K)+2).
\end{multline*}
On the other hand, it follows that
\[
\operatorname*{card}I(T,h_{\theta}\Lambda_{\theta},S)\geq\operatorname*{card}
I(T,\Lambda_{\theta},K)-2
\]
and the same way of reasoning leads to
\[
\operatorname*{card}I(T,\Lambda_{\theta},S)\geq\operatorname*{card}
I(T,h_{\theta}\Lambda_{\theta},K) -2.
\]
Making use of (\ref{nbvisitingSK1}) and (\ref{nbvisitingSK2}), we obtain
\begin{align*}
-2  &  \leq\operatorname*{card}I(T,h_{\theta}\Lambda_{\theta}
,S)-\operatorname*{card}I(T,\Lambda_{\theta},K)=\\
&  \operatorname*{card}I(T,h_{\theta}\Lambda_{\theta},S)-\operatorname*{card}
I(T,h_{\theta}\Lambda_{\theta},K)\\
&  +\operatorname*{card}I(T,h_{\theta}\Lambda_{\theta},K)-\operatorname*{card}
I(T,\Lambda_{\theta},S)\\
&  +\operatorname*{card}I(T,\Lambda_{\theta},S)-\operatorname*{card}
I(T,\Lambda_{\theta},K)\\
&  \leq 3 AT\varepsilon+2+3AT\varepsilon=6AT\varepsilon+2,
\end{align*}
hence (using (\ref{nbvisitingSK1}) once again)
\[
\left\vert \sum_{t\in I(T,\Lambda_{\theta},S)}\varphi(g_{t}\Lambda_{\theta
})-\sum_{t\in I(T,h_{\theta}\Lambda_{\theta},S)}\varphi(g_{t}h_{\theta}
\Lambda_{\theta})\right\vert \leq\varepsilon\operatorname{card}I(T,\Lambda
_{\theta},K)+(9A\varepsilon T+2)\left\Vert \varphi\right\Vert _{\infty}
\]
for $T\geq T(\Lambda_{\theta},\varepsilon)$. We obtain
\begin{align*}
\left\vert \operatorname*{card}I(T,h_{\theta}\Lambda_{\theta}
,S)-\operatorname*{card}I(T,\Lambda_{\theta},S)\right\vert  &  \leq\left\vert
\operatorname*{card}I(T,h_{\theta}\Lambda_{\theta},S)-\operatorname*{card}
I(T,\Lambda_{\theta},K)\right\vert \\
&  +\left\vert \operatorname*{card}I(T,\Lambda_{\theta},S\setminus
K)\right\vert \\
&  \leq 9AT\varepsilon+2
\end{align*}
as well. Relation (\ref{birkhoff2}) implies that $\operatorname{card}
I(T,h_{\theta}\Lambda_{\theta},S)\geq aT$ for $T\geq T(\Lambda_{\theta
},\varepsilon)$ where $a=\frac{1}{2}\frac{\mu(\mathcal{L)}}{\mu_{S}(S)}$. All
together, for $T\geq T(\Lambda_{\theta},\varepsilon)$, we obtain
\begin{align*}
\left\vert \Sigma_{2}-\Sigma_{1}\right\vert  &  \leq\left\vert \Sigma
_{2}-\frac{\operatorname{card}I(T,\Lambda_{\theta},S)}{\operatorname*{card}
I(T,h_{\theta}\Lambda_{\theta},S)}\Sigma_{1}\right\vert +\left\vert
\frac{\operatorname*{card}I(T,\Lambda_{\theta},S)-\operatorname*{card}
I(T,h_{\theta}\Lambda_{\theta},S)}{\operatorname*{card}I(T,h_{\theta}
\Lambda_{\theta},S)}\right\vert \left\vert \Sigma_{1}\right\vert \\
&  \leq\frac{1}{\operatorname{card}I(T,h_{\theta}\Lambda_{\theta}
,S)}\left\vert \sum_{t\in I(T,h_{\theta}\Lambda_{\theta},S)}\varphi(g_{t}
h_{\theta}\Lambda_{\theta})-\sum_{t\in I(T,\Lambda_{\theta},S)}\varphi(g_{t}\Lambda
_{\theta})\right\vert \\
&  +\left\vert \frac{\operatorname*{card}I(T,\Lambda_{\theta}
,S)-\operatorname*{card}I(T,h_{\theta}\Lambda_{\theta},S)}
{\operatorname*{card}I(T,h_{\theta}\Lambda_{\theta},S)}\right\vert \left\Vert
\varphi\right\Vert _{\infty}\\
&  \leq\frac{\varepsilon\operatorname{card}I(T,\Lambda_{\theta}
,K)+(9AT\varepsilon+2)\left\Vert \varphi\right\Vert _{\infty}}{\operatorname{card}
I(T,h_{\theta}\Lambda_{\theta},S)}+\frac{(9AT\varepsilon+2)\left\Vert \varphi\right\Vert
_{\infty}}{\operatorname{card}I(T,h_{\theta}\Lambda_{\theta},S)}\\
&  \leq\varepsilon(1+\frac{18AT\left\Vert \varphi\right\Vert _{\infty}}
{aT})+\frac{4\|\varphi\|_{\infty}}{aT}
\end{align*}
which is $\ll \varepsilon$ when $T$ is large enough.

\subsection{Proof of Theorem \ref{thm:reduction}}

\textbf{Step 1. }Let us show that the restriction of $f$ to $S$ is integrable
with respect to $\mu_{S}$.

We use Lemma \ref{lem:measure-image} with $X=]0,1[\times S$, $Y=\mathcal{L}
_{d+c}$, the map $\psi:]0,1[\times S\rightarrow\mathcal{L}$ defined by
$\psi(t,\Lambda)=g_{t}\Lambda$, the measures $\mu_{X}=dt\otimes\mu_{S}$ and
$\mu_{Y}=\mu$, and the function $f$. By definition of the induced measure, we
know that the image of the restriction of $dt\otimes\mu_{S}$ to any small enough
open subset $\omega$ is the restriction to $\psi(\omega)$ of the invariant
measure $\mu$ on $\mathcal{L}_{d+c}$. Now, by Lemma \ref{lem:nb-returntimes}, each element of $\mathcal{L}_{d+c}$ has at most
$A$ $\psi$-preimages, therefore by Lemma \ref{lem:measure-image},
\[
\int_{0}^{1}\int_{S}f(g_{t}\Lambda)d\mu_{S}dt\leq A\int_{\mathcal{L}_{d+c}}fd\mu.
\]
Since $f$ is uniformly continuous in the $\mathcal{H}_{\leq}$ direction,
there exists $\Delta>0$ such that for all $\Lambda$, and all $t\in
[0,\Delta],\ f(g_{t}\Lambda)\geq f(\Lambda)-1.$ Therefore
\[
\int_{0}^{\Delta}\int_{S}(f(\Lambda)-1)d\mu_{S}dt\leq A\int_{\mathcal{L}_{d+c}}
fd\mu,
\]
which implies $\int_{S}f(\Lambda)d\mu_{S}\leq\mu_{S}(S)+\frac{A}{\Delta}
\int_{\mathcal{L}_{d+c}}fd\mu<+\infty$.

\textbf{Step 2.} It is enough to prove the theorem for non negative continuous functions
$\varphi$.

Indeed, since $\varphi$ is continuous almost everywhere and since $\left\vert
\varphi\right\vert \leq f$ with $f$ continuous and in $L^{1}(\mu_{S})$, for all positive
integer $p$, there exist two continuous functions $\varphi_{p}^{-}$ and
$\varphi_{p}^{+}$ such that
\[
-f\leq\varphi_{p}^{-}\leq\varphi\leq\varphi_{p}^{+}\leq f
\]
and
\[
\int_{S}\varphi d\mu_{S}-\frac{1}{p}\leq\int_{S}\varphi_{p}^{-}d\mu_{S}
\leq\int_{S}\varphi_{p}^{+}d\mu_{S}\leq\int_{S}\varphi d\mu_{S}+\frac{1}{p}.
\]
Therefore, if the convergence holds for almost every $\theta$, we have for all the functions $\varphi_p^-$ and $\varphi_p^+$, 
\begin{align*}
\int_{S}\varphi_{p}^{-}d\mu_{S}  &  =\lim_{n\rightarrow\infty}\frac{1}{n}
\sum_{k=0}^{n-1}\varphi_{p}^{-}\circ R^{k}(\Lambda_{\theta})\leq\lim
\inf_{n\rightarrow\infty}\frac{1}{n}\sum_{k=0}^{n-1}\varphi\circ R^{k}
(\Lambda_{\theta})\\
&  \leq\lim\sup_{p\rightarrow\infty}\frac{1}{n}\sum_{k=0}^{n-1}\varphi\circ
R^{k}(\Lambda_{\theta})\leq\lim_{n\rightarrow\infty}\frac{1}{n}\sum
_{k=0}^{n-1}\varphi_{p}^{+}\circ R^{k}(\Lambda_{\theta})=\int_{S}\varphi
_{p}^{+}d\mu_{S}
\end{align*}
which implies that for almost all $\theta$,
\[
\lim_{n\rightarrow\infty}\frac{1}{n}\sum_{k=0}^{n-1}\varphi\circ R^{k}
(\Lambda_{\theta})=\int_{S}\varphi d\mu_{S}.
\]
Thus, we are reduced to proving the theorem for $\varphi$ continuous.
Writing $\varphi=\varphi^+-\varphi^-$, we can also suppose $0\leq\varphi\leq f$.

\textbf{Step 3. }We prove that $
\lim_{n\rightarrow\infty}\inf\frac{1}{n}\sum_{k=0}^{n-1}\varphi\circ
R^{k}(\Lambda_{\theta})\geq\frac{1}{\mu_{S}(S)}\int_{S}\varphi d\mu_{S}
$ for almost all $\theta$.

 Using Proposition \ref{prop:reduction} with the
minimum of $\varphi$ and of a constant $M$, we obtain for almost all $\theta$,
\[
\lim_{n\rightarrow\infty}\inf\frac{1}{n}\sum_{k=0}^{n-1}\varphi\circ
R^{k}(\Lambda_{\theta})\geq\lim_{n\rightarrow\infty}\frac{1}{n}\sum
_{k=0}^{n-1}\min(\varphi,M)\circ R^{k}(\Lambda_{\theta})=\frac{1}{\mu_{S}
(S)}\int_{S}\min(\varphi,M) d\mu_{S},
\]
hence, letting $M$ going to infinity, we obtain
\[
\lim_{n\rightarrow\infty}\inf\frac{1}{n}\sum_{k=0}^{n-1}\varphi\circ
R^{k}(\Lambda_{\theta})\geq\frac{1}{\mu_{S}(S)}\int_{S}\varphi d\mu_{S}.
\]

\textbf{Step 4.} Main step, we  bound above the sums $\sum_{k=0}^{n-1}\varphi\circ R^{k}
(\Lambda_{\theta})$.

Since $f$ is in $L^{1}(\mu)$, there exists $\varepsilon^{\prime}>0$ such that
for any measurable subset $B$ in $\mathcal{L}_{d+c}$, we have
\[
\mu(B)\leq\varepsilon^{\prime}\Rightarrow\int_{B}fd\mu\leq\varepsilon.
\]
This allows to strengthen Lemma \ref{lem:boundary}:

\begin{lemma}
For all $\varepsilon>0$, there exists a compact subset $K$ in $S$ and
$\delta>0$ such that $\frac{1}{\mu(\mathcal{L}_{d+c})}\int_{U(K,\delta)}fd\mu$
and $\frac{1}{\mu(\mathcal{L}_{d+c})}\mu(U(K,\delta))$ are $\leq\varepsilon$.
\end{lemma}

We keep all the choices and the notations of the proof of Proposition
\ref{prop:reduction}, and we use the Birkhoff ergodic theorem with one more function:
\begin{equation}
\lim_{T\rightarrow\infty}\frac{1}{T}\int_{0}^{T}f(g_{t}h_{\theta}
\Lambda_{\theta})1_{U(K,\delta)}(g_{t}h_{\theta}\Lambda_{\theta})dt=\frac
{1}{\mu(\mathcal{L}_{d+c})}\int_{U(K,\delta)}fd\mu\leq\varepsilon\label{L1}
\end{equation}
so that (\ref{birkhoff1}), (\ref{birkhoff2}), (\ref{birkhoff3}) and (\ref{L1})
hold for almost all $\theta$ (for (\ref{birkhoff3}) observe that $0\leq\varphi\leq f\in L^1(\mu_S)$).

Since the function $f$ is uniformly continuous in the $\mathcal{H}_{\leq }
$-direction, there exists $\kappa>0$ such that $f(h\Lambda)\geq f(\Lambda)-\frac
{1}{2}$ for all $\Lambda$ and all $h\in B_{\mathcal{H}_{\leq }}(I_{d+c}
,\kappa)$. By choosing $\gamma$ small enough we can suppose that
$(g_{s}hg_{-s})\in B_{\mathcal{H}_{\leq}}(I_{d+c},\kappa)$ for all $s\geq 0$
and all $h\in B_{\mathcal{H}_{\leq}}(I_{d+c},\gamma)$. Furthermore, there
exists a positive constant $\Delta=\Delta(\kappa)$ such that $g_{t}\in
B_{\mathcal{H}_{\geq 0}}(I_{d+c},\kappa)$ for all $t\in[0,\Delta]$. Therefore,
for all lattices $\Lambda$, all non negative real number $s$, all $h\in
B_{\mathcal{H}_{\leq}}(I_{d+c},\gamma)$ and all $t\in [0,\Delta]$, we have
\[
f(g_{t}(g_{s}hg_{-s})g_{s}\Lambda)\geq f((g_{s}hg_{-s})g_{s}\Lambda)-\frac
{1}{2}\geq f(g_{s}\Lambda)-1
\]
and hence
\begin{equation}
f(g_{t+s}h\Lambda)=f(g_{t}(g_{s}hg_{-s})g_{s}\Lambda)\geq f(g_{s}\Lambda)-1.
\label{uniform-f}
\end{equation}

As in the proof of Proposition \ref{prop:reduction}, let $s_{1}<...<s_{m}$ be
the elements of $I(T,\Lambda_{\theta},S\setminus K)$. On the one hand $g_{t}h_{\theta
}\Lambda_{\theta}\in U(K,\delta)$ for all $t\in[ s_{i},s_{i}+1]$, and on the other hand, for almost
all $\theta$,  (\ref{nbvisitingSK1}) and (\ref{nbvisitingSK2}) hold for $T\geq
T(\Lambda_{\theta},\varepsilon)$. We can suppose $\Delta<1$ and since there
are at most $A$ elements of $I(T,\Lambda_{\theta},S\setminus K)$ ($\subset
I(T,\Lambda_{\theta},S)$) in an interval of length $1$, by
(\ref{uniform-f}) we obtain
\begin{align*}
\sum_{s\in I(T,\Lambda_{\theta},S\setminus K)}\varphi(g_{s}\Lambda_{\theta})
&  \leq\sum_{s\in I(T,\Lambda_{\theta},S\setminus K)}f(g_{s}\Lambda_{\theta
})\\
&  \leq\sum_{i=1}^{m}\frac{1}{\Delta}\int_{s_{i}}^{s_{i}+\Delta}
(1+1_{U(K,,\delta)}(g_{t}h_{\theta}\Lambda_{\theta})f(g_{t}h_{\theta}
\Lambda_{\theta}))dt\\
&  \leq m+\frac{A}{\Delta}\int_{0}^{T+1}1_{U(K,\delta)}(g_{t}h_{\theta
}\Lambda_{\theta})f(g_{t}h_{\theta}\Lambda_{\theta})dt
\end{align*}
and with (\ref{nbvisitingSK1}) and (\ref{L1}), this gives
\[
\sum_{s\in I(T,\Lambda_{\theta},S\setminus K)}f(g_{s}\Lambda_{\theta}
)\leq 3AT\varepsilon+\frac{A}{\Delta}3T\varepsilon\leq 6\frac{A}{\Delta
}T\varepsilon
\]
for all $T\geq T(\Lambda_{\theta},\varepsilon)$.

We want to bound above
\[
\Sigma_{1}(T)=\frac{1}{\operatorname*{card}I(T,\Lambda_{\theta},S)}\sum_{t\in
I(T,\Lambda_{\theta},S)}\varphi(g_{t}\Lambda_{\theta})
\]
with
\[
\Sigma_{2}(T)=\frac{1}{\operatorname*{card}I(T,h_{\theta}\Lambda_{\theta},S)}
\sum_{t\in I(T,h_{\theta}\Lambda_{\theta},S)}\varphi(g_{t}h_{\theta}
\Lambda_{\theta})
\]
because this latter sum tends to $\frac{1}{\mu_{S}(S)}\int_{S}\varphi d\mu_{S}$
when $T$ goes to infinity. We split $\sum_{t\in I(T,\Lambda_{\theta},S)}$ into
two sums $\sum_{t\in I(T,\Lambda_{\theta},K)}$ and $\sum_{t\in I(T,\Lambda
_{\theta},S\setminus K)}$. As in the proof of Proposition \ref{prop:reduction}, for $T$ large enough, we
have
\[
\left\vert I(T,h_{\theta}\Lambda_{\theta},S)-I(T,\Lambda_{\theta
},S)\right\vert \leq 9AT\varepsilon+2,
\]
\[
I(T,h_{\theta}\Lambda_{\theta},S)\geq aT,
\]
and
\[
\left\vert \sum_{t\in I(T,\Lambda_{\theta},K)}\varphi(g_{t}\Lambda_{\theta
})-\sum_{t\in I(T,\Lambda_{\theta},K)}\varphi(g_{t-t^{\prime}}h_{\theta
}\Lambda_{\theta})\right\vert \leq\varepsilon\operatorname{card}
I(T,\Lambda_{\theta},K)
\]
where $t^{\prime}=\sigma(g_{t}h_{\theta}g_{-t},g_{t}\Lambda_{\theta})$ is
defined in Lemma \ref{lem:K-S}. Taking into account  the first element $t_{min}$ and  the last element in $I(T,\Lambda_{\delta},K)$, the latter inequality implies that
\begin{align*}
\sum_{t\in I(T,\Lambda_{\theta},K)}\varphi(g_{t}\Lambda_{\theta})  &  \leq \varphi(g_{t_{min}}\Lambda_{\theta})+
\sum_{t\in I(T,\Lambda_{\theta},K)\setminus\{t_{min}\}}\varphi(g_{t-t^{\prime}}h_{\theta}
\Lambda_{\theta})+\varepsilon\operatorname{card}I(T,\Lambda_{\theta},K)\\
&  \leq\varphi(R\Lambda_{\theta})+\sum_{t\in I(T+1,h_{\theta}\Lambda_{\theta},S)}\varphi(g_{t}h_{\theta
}\Lambda_{\theta})+\varepsilon\operatorname{card}I(T,\Lambda_{\theta},K).
\end{align*}
All together, we obtain (recall that $\varphi\geq 0$)
\begin{align*}
\Sigma_{1}(T)  &  \leq
\frac{1}{\operatorname{card}I(T,\Lambda_{\theta},S)}\varphi(R\Lambda_{\theta})+
\frac{\operatorname{card}I(T+1,h_{\theta}\Lambda_{\theta}
,S)}{\operatorname{card}I(T,\Lambda_{\theta},S)}\Sigma_{2}(T+1)+\varepsilon
\\
&\hspace{0.5cm}+\frac{1}{\operatorname*{card}I(T,\Lambda_{\theta},S)}\sum_{s\in
I(T,\Lambda_{\theta},S\setminus K)}f(g_{s}\Lambda_{\theta})\\
&  \leq\frac{1}{\operatorname{card}I(T,\Lambda_{\theta},S)}\varphi(R\Lambda_{\theta})+
\left(  1+\left\vert \frac{\operatorname*{card}I(T+1,h_{\theta}
\Lambda_{\theta},S)-\operatorname*{card}I(T,\Lambda_{\theta},S)}
{\operatorname*{card}I(T,\Lambda_{\theta},S)}\right\vert \right)  \Sigma
_{2}(T+1)+\varepsilon\\
&  \hspace{0.5cm}+\frac{1}{\operatorname*{card}I(T,\Lambda_{\theta},S)}6\frac{A}{\Delta
}T\varepsilon\\
&  \leq
\Sigma_{2}(T+1)+\frac{1}{aT-9AT\varepsilon-2}\varphi(R\Lambda_{\theta})+\frac{9AT\varepsilon+2+A}{aT-9AT\varepsilon-2}\Sigma_{2}(T+1)+\left(1+\frac{6A}
{(aT-9AT\varepsilon-2)\Delta}\right)\varepsilon
\end{align*}
and we are done. $\square\bigskip$

\subsection{Proofs of Lemma \ref{lem:boundary}}

We need an auxiliary lemma.

\begin{lemma}
Let $E(\lambda,\eta)$ be the set of lattices $\Lambda$ in $\mathcal{L}_{d,c}$ such
that there exist two nonzero vectors $X\neq\pm X^{\prime}$ of $\Lambda$ in the
open ball $\overset{o}B_{d,c}(0,\lambda)\ $with $\frac{1}{1+\eta}
<\frac{\left\vert X\right\vert _{\mathbb{\pm}}}{\left\vert X^{\prime
}\right\vert _{\mathbb{\pm}}}<1+\eta$ or a nonzero vector $X$ in the open ball
$\overset{o}B_{d,c}(0,\lambda)$ with $\left\vert X\right\vert _{\pm}<\eta$.
For all $\lambda>0$, we have $\lim_{\eta\rightarrow 0}\mu(E(\lambda,\eta))=0$.
\end{lemma}

\begin{proof}
Since $\lim_{\rho\rightarrow 0}\mu(\{\lambda_{1}(\Lambda)\leq\rho\})=0$, it is
enough to show that for all $\rho>0$, $\mu(E(\lambda,\eta)\cap\{\lambda
_{1}(\Lambda)\geq\rho\})\rightarrow 0$ when $\eta$ goes to $0$.
Choose a Siegel reduction domain $\mathcal{S}\subset \SL(d+c,\mathbb{R)}$ (see \cite{BeMa}). There is a
constant $c=c(\mathcal{S})>0$ such that for all matrices $M$ in $\mathcal{S}$
and all vectors $Y$ in $\mathbb{R}^{d+c}$, we have
\[
\left\Vert MY\right\Vert _{d,c}\geq c\lambda_{1}(\Lambda
)\left\Vert Y\right\Vert _{d,c}.
\]
where $\Lambda=M\mathbb{Z}^{d+c}$ (this inequality holds for all norms with a
constant $c$ depending only on the norm, just use the norm equivalence). It
follows that we can find a finite subset $F_{\rho}$ of $\mathbb{Z}^{d+c}$ such
that for all matrices $M$ in $\mathcal{S}$ with $\lambda_{1}(M\mathbb{Z}
^{d+c})\geq\rho$, the only $Y$ in $\mathbb{Z}^{d+c}$ such that $\left\Vert
MY\right\Vert _{d,c}\leq\lambda$, are in $F_{\rho}$. Therefore,
if a matrix $M$ in $\mathcal{S}$ is such that $\Lambda=M\mathbb{Z}^{d+c}$
belongs to $E(\lambda,\eta)\cap\{\lambda_{1}(\Lambda)\geq\rho\}$, then there
exist a nonzero $Y$ in $F_{\rho}$ or two nonzero vectors $Y\neq\pm Y^{\prime}$ in
$F_{\rho}$ with
\[
\left\vert MY\right\vert _{\pm}\leq\eta
\]
or
\[
\frac{1}{1+\eta}<\frac{\left\vert MY\right\vert _{\mathbb{\pm}}}{\left\vert
MY^{\prime}\right\vert _{\mathbb{\pm}}}<1+\eta.
\]
For a fixed $Y$ or a fixed pair $Y\neq\pm Y^{\prime}$ of nonzero vectors in
$F_{\rho}$, the measure of the set of $M$ in $\mathcal{S}$ for which the above
inequalities hold, goes to $0$ as $\eta$ goes to $0$. Since $F_{\rho}$ is
finite and since a Siegel domain contains a fundamental domain, we are done.
\end{proof}

\begin{proof}
[Proof of Lemma \ref{lem:boundary} ]Consider the set $V(\lambda,\eta
,\rho)=E(\lambda,\eta)\cup\{\lambda_{1}(\Lambda)<\rho\}$.

\textbf{Step 1: }\textit{ The complementary $V^{C}$ of $V(\lambda,\eta,\rho)$ is a
closed subset of $\mathcal{L}_{d,c}$.}
\newline Let $(\Lambda_{n})_{n\in\mathbb{N}}$
be a sequence of points of $V^{C}$ which converges to $\Gamma$ in $\mathcal{L}_{d,c}
$. Firstly, since $\lambda_{1}$ is continuous, $\lambda_{1}(\Gamma)=\lim
\lambda_{1}(\Lambda_{n})\geq\rho$. Secondly,
there is a sequence of matrices $(M_{n})_{n\in\mathbb{N}}$ such that $\Lambda
_{n}=M_{n}\mathbb{Z}^{d+c}$ for all $n\in\mathbb{N}$ and such that
$(M_{n})_{n\in\mathbb{N}}$ converges to $M$ with $\Gamma=M\mathbb{Z}^{d+c}$.
We have to show
that $\Gamma$ is not in $E(\lambda,\eta)$.
Let $X=MY$ and $X^{\prime}=MY^{\prime}$ be two nonzero vectors in $\Gamma$
with $X\neq\pm X^{\prime}$ and $\left\Vert X\right\Vert _{d,c}
$,\ $\left\Vert X^{\prime}\right\Vert _{d,c}<\lambda$. When $n$
is large enough, $X_{n}=M_{n}Y$ and $X_{n}^{\prime}=M_{n}Y^{\prime}$ are in
the open ball $\overset{o}B_{d,c}(0,\lambda)$ and since $\Lambda_{n}=M_{n}\mathbb{Z}^{d+c}$ is
not in $E(\lambda,\eta)$, we have both
\[
\left\vert M_{n}Y\right\vert _{\pm}\geq\eta \,\text{ and }\,
\frac{\left\vert M_{n}Y\right\vert _{\pm}}{\left\vert M_{n}Y^{\prime
}\right\vert _{\pm}}\notin]\frac{1}{1+\eta},1+\eta[.
\]
Passing through the limit we obtain
\[
\left\vert X\right\vert _{\pm}\geq\eta
\,\text{ and }\,
\frac{\left\vert X\right\vert _{\pm}}{\left\vert X^{\prime}\right\vert _{\pm}
}=\frac{\left\vert MY\right\vert _{\pm}}{\left\vert MY^{\prime}\right\vert
_{\pm}}\notin]\frac{1}{1+\eta},1+\eta[.
\]
Therefore $\Gamma$ is not in $E(\lambda,\eta)$.

\textbf{Step 2: }\textit{$F=S\setminus V(\lambda,\eta,\rho)$ is a compact subset of
$S$ when $\lambda\geq 2\max\{\lambda_{1}(\Lambda):\Lambda\in
\mathcal{L}_{d,c}\}$.}
\newline Thanks to the Malher compactness theorem, it is enough to prove that
$F$ is a closed subset of $\mathcal{L}_{d,c}$. Let $(\Lambda_{n})_{n\in\mathbb{N}}$
be a sequence of points in $F$ which converges to $\Gamma$ in $\mathcal{L}_{d,c}$. We
want to prove that $\Gamma$ is in $F$. By the first step, it is enough to prove
that $\Gamma$ is in $S$. Choose a Siegel domain $\mathcal{S}$. There is a
sequence of matrices $M_{n}\in\mathcal{S}$ such that $\Lambda_{n}
=M_{n}\mathbb{Z}^{d+c}$ for all $n\in\mathbb{N}$ and such that $(M_{n}
)_{n\in\mathbb{N}}$ converges to $M\in\mathcal{S}$. For each $n$, there are
two vectors $Y_{1,n}$ and $Y_{2,n}$ in $\mathbb{Z}^{d+c}$ such that
$X_{1,n}=M_{n}Y_{1,n}$ and $X_{2,n}=M_{n}Y_{2,n}$ are the two vectors
associated with $\Lambda_{n}$ by the definition of $S$. Since the matrices
$M_{n}$ are all in the Siegel domain $\mathcal{S}$ and since $\left\Vert
M_{n}Y_{i,n}\right\Vert _{d,c}=\lambda_{1}(\Lambda_{n})$,
$i=1,2$, the sequences $(Y_{i,n})_{n\in\mathbb{N}}$, $i=1,2$, are bounded
sequences in $\mathbb{Z}^{d+c}$. Therefore, extracting subsequences, we can
suppose that the two sequences $(Y_{i,n})_{n\in\mathbb{N}}$ are constant:
$Y_{i,n}=Y_{i}$ for all $n$, $i=1,2$. It follows that $\left\Vert
MY_{i}\right\Vert _{d,c}=\lim_{n\rightarrow\infty}\left\Vert
M_{n}Y_{i}\right\Vert _{d,c}=\lim_{n\rightarrow\infty}\lambda
_{1}(\Lambda_{n})=\lambda_{1}(\Gamma)$. Moreover
\[
\left\vert MY_{1}\right\vert _{+}=\lim_{n\rightarrow\infty}\left\vert
M_{n}Y_{1}\right\vert _{+}=\lim_{n\rightarrow\infty}\lambda_{1}(\Lambda
_{n})=\lambda_{1}(\Gamma)
\]
and
\[
\left\vert MY_{2}\right\vert _{-}=\lim_{n\rightarrow\infty}\left\vert
M_{n}Y_{2}\right\vert _{-}=\lim_{n\rightarrow\infty}\lambda_{1}(\Lambda
_{n})=\lambda_{1}(\Gamma).
\]
Suppose now that $\lambda\geq 2\max\{\lambda_{1}(\Lambda):\Lambda\in
\mathcal{L}_{d,c}\}$, then making use of the first step, we conclude that $\Gamma$ is
in $S$.

\textbf{Step 3. }For a neighborhood $W$ of $I_{d+c}$ in $\SL(d+c,\mathbb{R)}$,
let
\[
U_{W}=\{g_{t}h\Lambda:t\in[0,1],\ h\in W,\ \Lambda\in V(\lambda
,\eta,\rho)\}.
\]
Let us show that we can choose $W$ in order that $U_{W}\subset V(2e^{d+c}
\lambda,5e^{d+c}\eta,2e^{d+c}\rho)$. It will finish the proof of Lemma \ref{lem:boundary}.
Indeed, we first fix $\lambda\geq 2\max\{\lambda_{1}(\Lambda):\Lambda
\in\mathcal{L}_{d,c}\}$, next we take $\eta$ and $\rho$ such that $\mu
(E(2e^{d+c}\lambda,5e^{d+c}\eta))\leq\frac{\varepsilon}{2}$ and $\mu
(\{\lambda_{1}<2e^{d+c}\rho\})\leq\frac{\varepsilon}{2}$, then we take $W$
such that $U_{W}\subset V(2e^{d+c}\lambda,5e^{d+c}\eta,2e^{d+c}\rho)$ and
$\delta$ such that $B(I_{d+c},\delta)\subset W$. Now by the second step
$K=S\setminus V(\lambda,\eta,\rho)$ is compact and since $U(K,\delta)\subset
U_{W}\subset V(2e^{d+c}\lambda,5e^{d+c}\eta,2e^{d+c}\rho)$, we have
$\mu(U(K,\delta))\leq\varepsilon$.

Let $\Lambda$ be in $V(\lambda,\eta,\rho)$, $h$ in $W$ and $t\in[0,1]$.
We explain how to successively reduce $W$ in order to obtain the above inclusion.

\textbf{Case 1. }Suppose $\lambda_{1}(\Lambda)<\rho$. We can choose $W$ small
enough in order that $\left\Vert hX\right\Vert _{d,c}
\leq 2\left\Vert X\right\Vert _{d,c}$ for all $h$ in $W$ and all
$X$ in $\mathbb{R}^{d+c}$. This implies that $\lambda_{1}(g_{t}h\Lambda
)\leq 2e^{d+c}\lambda_{1}(\Lambda)<2e^{d+c}\rho$, hence $g_{t}h\Lambda\in
V(2e^{d+c}\lambda,5e^{d+c}\eta,2e^{d+c}\rho)$.

\textbf{Case 2. }Suppose there exists a nonzero vector $X$ in $\Lambda\cap
\overset{o}B_{d,c}(0,\lambda)$ with $\left\vert X\right\vert _{-}<\eta$ (the case $\left\vert
X\right\vert _{+}<\eta$ is similar). Call $p_{\pm}$ the projections on the
subspaces $E_{\pm}$ and $\left\Vert u\right\Vert $ the norm of the linear
operator $u$ associated with the norm $\left\Vert .\right\Vert _{d,c}$. The vector $g_{t}hX$ is in the open ball $\overset{o}B_{d,c}(0,2\lambda e^{d+c})$ and
we have
\[
p_{-}g_{t}hX=g_{t}p_{-}hp_{-}X+g_{t}p_{-}hp_{+}X,
\]
hence
\[
\left\vert g_{t}hX\right\vert _{-}<e^{d+c}\left\Vert p_{-}h\right\Vert
\eta+e^{d+c}\left\Vert p_{-}hp_{+}\right\Vert \lambda.
\]
We can choose $W$ in order that $\left\Vert p_{-}hp_{+}\right\Vert <\frac
{\eta}{\lambda}$ and $\left\Vert p_{-}h\right\Vert \leq 2$. Then we have $\left\vert g_{t}hX\right\vert _{-}
<3e^{d+c}\eta$ which implies that $g_{t}h\Lambda\in V(2e^{d+c}\lambda
,5e^{d+c}\eta,2e^{d+c}\rho)$.

\textbf{Case 3. }Suppose there exist two distinct nonzero vectors $X$ and
$X^{\prime}$ in $\Lambda\cap \overset{o}B_{d,c}(0,\lambda)$ such that
\[
\left\vert  X\right\vert  _{-},\left\vert  X^{\prime}\right\vert  _{-}\geq\eta\, \text{ and }\, \frac
{1}{1+\eta}<\frac{\left\vert  X\right\vert  _{-}}{\left\vert  X^{\prime}\right\vert  _{-}
}<1+\eta.
\]
The case with $\left\vert  .\right\vert  _{+}$ is similar. \newline As above,
\begin{align*}
\left\vert  hX\right\vert  _{-}  &  <\left\|  p_{-}h\right\|  \left\vert  X\right\vert
_{-}+\left\|  p_{-}hp_{+}\right\|  \lambda\\
&  \leq(\left\|  p_{-}h\right\|  +\left\| p_{-}hp_{+}\right\|  \frac{\lambda
}{\eta})\left\vert  X\right\vert  _{-}.
\end{align*}
We can choose $W$ in order that $\left\|  p_{-}h\right\|  +\left\|
p_{-}hp_{+}\right\|  \frac{\lambda}{\eta}\leq 1+$ $\eta$. We  also have
\[
\left\vert  hX^{\prime}\right\vert  _{-}\geq\left\Vert p_{-}hp_{-}X^{\prime}\right\Vert_{d,c}
-\left\|  p_{-}hp_{+}\right\|  \left\|  X^{\prime}\right\|_{d,c}  .
\]
We can choose $W$ in order that $\left\|  p_{-}hp_{+}\right\|  \frac{\lambda
}{\eta}\leq\eta$ and 
$\left\Vert p_{-}-p_{-}hp_{-}\right\Vert\leq \eta$. With this choice, we have
\[
\left\vert  hX^{\prime}\right\vert  _{-}\geq\left\vert  X^{\prime}\right\vert  _{-}
-\eta\left\vert  X^{\prime}\right\vert  _{-}-\eta\left\vert  X^{\prime}\right\vert  _{-}.
\]
It follows that
\[
\frac{\left\vert  g_{t}hX\right\vert  _{-}}{\left\vert  g_{t}hX^{\prime}\right\vert  _{-}
}=\frac{\left\vert  hX\right\vert  _{-}}{\left\vert  hX^{\prime}\right\vert  _{-}}
\leq\frac{\left\vert  X\right\vert  _{-}}{\left\vert  X^{\prime}\right\vert  _{-}}
\times\frac{1+\eta}{1-2\eta}\leq\frac{(1+\eta)^{2}}{1-2\eta}\leq 1+5\eta
\]
when $\eta$ is small enough. Inverting the role of $X$ and $X^{\prime}$, we get
the inequality $\frac{\left\vert  hX^{\prime}\right\vert  _{-}}{\left\vert  hX\right\vert
_{-}}\leq 1+5\eta$ and we are done.
\end{proof}

\medskip

\subsection{Proofs of Theorems \ref{thm:levy-d} and \ref{thm:jager}.1.}

We begin by the proof of Theorem \ref{thm:levy-d} which is more difficult.
We want to prove that for almost all $\theta$ in $\M_{d,c}(\mathbb{R})$,
\[
\lim_{n\rightarrow\infty}\frac{1}{n}\ln q_{n}(\theta)=L_{d,c}=\frac{1}{\mu
_{S}(S)}\int_{S}\rho(\Lambda)~d\mu_{S}(\Lambda)
\]
and that
\[
\lim_{n\rightarrow\infty}\frac{-1}{n}\ln r_{n}(\theta)=\frac{c}{d}L_{d,c}.
\]
Let us prove that the convergence almost everywhere of $\frac
{1}{n}\ln q_{n}(\theta)$ to $L_{d,c}=\frac{1}{\mu_{S}(S)}\int_{S}\rho
(\Lambda)~d\mu_{S}(\Lambda)$, implies the convergence almost everywhere of
$\frac{-1}{n}\ln r_{n}(\theta)$ to $\frac{c}{d}L_{d,c}$. By Dirichlet's theorem we know that for all $\theta$ and all integers $n$, $q_n(\theta)^cr_{n+1}(\theta)^d\leq C $ where $C$ is some universal constant (depending on the norms). Next, by the divergence part of the Khintchin-Groshev theorem (see \cite{BeDiVe}), we know that for almost all $\theta$, there exist only finitely many $Q\in\mathbb Z^c$ such that 
\[
d(\theta Q, \mathbb Z^d)\leq \psi(\|\theta\|_{\mathbb R^c})
\]
where $\psi(q)=\tfrac{1}{(q^c\ln q)^{1/d}}$. Therefore, for almost all $\theta$, $\tfrac{1}{\ln q_n(\theta)}\leq q_n(\theta)^cr_n(\theta)^d\leq C$ for all $n$ large enough. It follows that for almost all $\theta$, 
\[
 \frac{c}{n}\ln q_n(\theta)+\frac{\ln \ln q_n(\theta)}{n}\geq -\frac{d}{n}\ln r_n(\theta)^d\geq \frac{c}{n}\ln q_n(\theta)-\frac{\ln C}{n}
\]
for all large enough $n$. Since the left hand side of the first inequality and the right hand side of the second inequality both converge to $cL_{d,c}$ when $\lim_{n\rightarrow\infty}\frac{1}{n}\ln q_{n}(\theta)=L_{d,c}$, the sequence $(-\frac{d}{n}\ln r_n(\theta)^d)_n$ converges to the same limit for almost all $\theta$. 
Therefore, the proof
of Theorem \ref{thm:levy-d} reduces to prove the first almost everywhere limit.

Let us now prove the first almost everywhere limit. We need two lemmas. The
first one is clear.

\begin{lemma}
For all compact sets $K$ in $\mathbb{R}^{d+c}$ and all $\varepsilon>0$, there
exists $\alpha>0$ such that for all $g\in B(I_{d+c},\alpha)$ and all $x\in K$,
$\dd(gx,x)\leq\varepsilon$.
\end{lemma}

\begin{lemma}
Let $\Lambda$ be in $S^{\prime}\setminus\mathcal{N}$. Then the return map $R=R_{S}$ is defined on some neighborhood of $\Lambda$ and is continuous at $\Lambda$.
\end{lemma}

\begin{proof}
 Consider the minimal
vectors $X_{0}=X_{0}(\Lambda)$ and $X_{1}=X_{1}(\Lambda)$. By definition of
$S^{\prime}$ the only nonzero vectors $B_{d,c}(0,\lambda_{1}(\Lambda))=\mathcal{C}
(X_{0})$ are $\pm X_{0}$. Therefore, there exists $\varepsilon>0$ such that all
$X$ in $\Lambda\setminus\{0,\pm X_{0}\}$ are at a distance $\geq\varepsilon$
from $\mathcal{C}(X_{0})$. Since $\Lambda$ is not in $\mathcal{N}$, $\pm
X_{0}$ and $\pm X_{1}$ are the only nonzero vectors of $\Lambda$ in the
cylinder $\mathcal{C}(X_{0},X_{1})$. Therefore reducing $\varepsilon$ if
necessary, all $X$ in $\Lambda\setminus\{0,\pm X_{0},\pm X_{1}\}$ are at a
distance $\geq\varepsilon$ from $\mathcal{C}(X_{0},X_{1})$. By the above lemma
we can choose $\delta>0$ such that
$\forall g\in B(I_{d+c},\delta),\ \forall X\in\mathcal{C}(X_{0},X_{1})+B_{d,c}(0,1), $
\[
 \max(\dd(g^{-1}X,X),\dd(gX,X))\leq\varepsilon/3\text{. }
\]
It follows that for all $g\in B(I_{d+c},\delta)$, $\pm gX_{0}$ are the only
nonzero vectors of $g\Lambda$ in $\mathcal{C}(gX_{0})$ and that $\pm gX_{0}$
and $\pm gX_{1}$ are the only nonzero vectors of $\Lambda$ in $\mathcal{C}
(gX_{0},gX_{1})$. It follows that if the lattice $\Gamma=g\Lambda$ is in the set of lattices $B(I_{d+c}
,\delta)\Lambda\cap S^{\prime}$ then $X_{0}(\Gamma)=gX_{0}$ and $X_{1}
(\Gamma)=gX_{1}$. By definition of $S$, the return times  from $\Lambda$ and $\Gamma$ are well defined and we have
\begin{align*}
\tau(\Lambda)   =\frac{1}{d+1}\ln\frac{\left|  X_{1}\right|  _{-}}{\left|
X_{0}\right|  _{+}} \,\text{ and }
\tau(\Gamma)    =\frac{1}{d+1}\ln\frac{\left|  gX_{1}\right|  _{-}}{\left|
gX_{0}\right|  _{+}},
\end{align*}
hence $R(\Lambda)$ and $R(\Gamma)$ are defined and
\begin{align*}
\left|  \tau(\Lambda)-\tau(\Gamma)\right|   &  =\frac{1}{d+1}\left|  \ln
\frac{\left|  X_{1}\right|  _{-}}{\left|  X_{0}\right|  _{+}}\frac{\left|
gX_{0}\right|  _{+}}{\left|  gX_{1}\right|  _{-}}\right| \\
&  \leq\frac{1}{d+1}(\left|  \ln\frac{\left|  X_{1}\right|  _{-}}{\left|
gX_{1}\right|  _{-}}\right|  +\left|  \ln\frac{\left|  gX_{0}\right|  _{+}
}{\left|  X_{0}\right|  _{+}}\right|  ).
\end{align*}
which goes to zero when $\delta$ goes to zero.
\end{proof}

\begin{proof}[End of proof of Theorem \ref{thm:levy-d}]
We use Theorem \ref{thm:reduction} with $S^{\prime}$ and the function
$\varphi:S^{\prime}\rightarrow\mathbb{R}_{\geq 0}$ defined by $\varphi
(\Lambda)=\rho\circ R(\Lambda)=\ln\frac{q_{1}(\Lambda)}{q_{0}(\Lambda)}$ when $R(\Lambda)$ is defined and by $\varphi(\Lambda)=0$ otherwise.
Since $\rho$ is continuous on $S$, the above lemma implies that $\varphi$ is almost everywhere continuous on
$S^{\prime}$. We want to find a uniformly continuous function $f:\mathcal{L}
\rightarrow\mathbb{R}$ such that $\left\vert \varphi\right\vert \leq f$.
Observe that $\varphi$ is non negative. By Lemma \ref{lem:Minkowski}, for
all lattice $\Lambda\in\mathcal{L}$,
\[
q_{1}(\Lambda)^{c}r_{0}(\Lambda)^{d}\leq C=C_{d,c}
\]
where $C_{d,c}$ depends only on $c$ and $d$. It follows that for all $\Lambda$ in
$S^{\prime}$, we have
\begin{align*}
\varphi(\Lambda) &  =\ln\frac{q_{1}(\Lambda)}{q_{0}(\Lambda)}\\
&  =\ln\frac{q_{1}(\Lambda)r_{0}(\Lambda)^{d/c}}{q_{0}(\Lambda)r_{0}
(\Lambda)^{d/c}}\\
&  \leq\ln C^{d/c}-\ln q_{0}(\Lambda)r_{0}(\Lambda)^{d/c}.
\end{align*}
Now for $\Lambda$  in $S^{\prime}$, we have $q_{0}(\Lambda)=r_{0}(\Lambda)=\lambda_1(\Lambda)$,
thus
\[
\varphi(\Lambda)\leq\ln C^{d/c}-\frac{d+c}{c}\ln\lambda_{1}(\Lambda).
\]
It is well known (see \cite{Chev-Levy}) that the function $\ln\lambda_{1}$ is uniformly continuous
and it is integrable on $\mathcal{L}_{c+d}$ by Lemma \ref{lem:Siegel}, consequently we can use Theorem
\ref{thm:reduction} with $S^{\prime}$ and $\varphi$. It follows that for
almost all $\theta$ in $\M_{d,c}(\mathbb{R})$, we have
\[
\lim_{n\rightarrow\infty}\frac{1}{n}\sum_{k=0}^{n-1}\varphi\circ R_{S^{\prime}
}^{k}(\Lambda_{\theta})=\frac{1}{\mu_{S^{\prime}}(S^{\prime})}\int_{S^{\prime}}
\varphi~d\mu_{S^{\prime}}
\]
where $R_{S^{\prime}}$ is the first return map to $S^{\prime}$. Now by Lemmas \ref{lem:min-best} and \ref{lem:return-best}, for almost
all $\theta$, there is an integer $k_0$ such that
\[
\varphi\circ R_{S^{\prime}}^{k}(\Lambda_{\theta})=\ln\frac{q_{k+k_0+1}(\theta)}{q_{k+k_0}(\theta
)}
\]
for all large enough $k$, where $k_0$ depends only on $\theta$. It follows that for almost all $\theta$ 
\[
\lim_{n\rightarrow\infty}\frac{1}{n}\sum_{k=0}^{n-1}\varphi\circ R_{S^{\prime}
}^{k}(\Lambda_{\theta})=
\lim_{n\rightarrow\infty}\frac{1}{n}\sum_{k=0}^{n-1}\ln\frac{q_{k+k_0+1}(\theta)}{q_{k+k_0}(\theta
)}=
\lim_{n\rightarrow\infty}\frac{1}{n}\ln q_{n}(\theta)
\]
So that, the only thing left is the equality
\[
\int_{S}\rho~d\mu_{S}=\int_{S^{\prime}}\varphi~d\mu_{S^{\prime}}.
\]
Now, the image of $\mu_{S^{\prime}}$ by $R$ is $\mu_{S}$, hence
\begin{align*}
\int_{S^{\prime}}\varphi~d\mu_{S^{\prime}}   =\int_{S^{\prime}}\rho\circ
R\ d\mu_{S^{\prime}}
 =\int_{S}\rho\ d\mu_{S}.
\end{align*}
\end{proof}

\begin{proof}[Proof of Theorem \ref{thm:jager}.1]
Consider the map $F:S\rightarrow \R$ defined by 
\[
F(\Lambda)=q_1^c(\Lambda)r_0^d(\Lambda)=|v_1^S(\Lambda)|_-^c|v_0^S(\Lambda)|^d_+
\]
and call $\nu=\nu_{d,c}$ the image of the measure $\frac{1}{\mu_S(S)}\mu_S$ by $F$.
Let $\varphi:\R\rightarrow \R$ be a continuous and bounded function. We want to prove that 
\[
\lim_{n\rightarrow\infty}\frac{1}{n}\sum_{k=0}^{n-1}\varphi(\beta_k(\theta))
=\int_{\mathbb R} \varphi(x) \ d\nu(x)
\] 
for almost all $\theta \in \M_{d,c}(\mathbb R)$. 

Now by Lemmas  \ref{lem:min-best} and \ref{lem:return-best}, for almost all $\theta$, 
\[
F(R^k(\Lambda_{\theta}))=q_{k+k_0+1}^c(\theta)r_{k+k_0}^d(\theta)=\beta_{k+k_0}(\theta)
\]
for all $k$ large enough. Now the function $\varphi\circ F$ is bounded and continuous, thus by Theorem \ref{thm:reduction} (or Proposition \ref{prop:reduction}) we have for almost all $\theta$,
\begin{align*}
\lim_{n\rightarrow\infty}\frac{1}{n}\sum_{k=0}^{n-1}\varphi\circ F(R^k(\Lambda_\theta))
&=\frac{1}{\mu_S(S)}\int_{S}\varphi\circ F(\Lambda)\, d\mu_S(\Lambda)\\
&=\int_{\mathbb R} \varphi(x) \, d\nu(x)
\end{align*}
which  implies that
\[
\lim_{n\rightarrow\infty}\frac{1}{n}\sum_{k=0}^{n-1}\varphi(\beta_k(\theta))
=\lim_{n\rightarrow\infty}\frac{1}{n}\sum_{k=0}^{n-1}\varphi\circ F(R^k(\Lambda_\theta))
=\int_{\mathbb R} \varphi(x) \ d\nu(x)
\]
and finishes the proof of Theorem \ref{thm:jager}.1.
The proof of Theorem \ref{thm:jager}.2. is postponed  at the end of Section 9. 
\end{proof}

\section{On $\lim\inf q_{n+k}^{c}r_{n}^{d}$ \label{sec:qnrn}}

For each $\theta$ in $\M_{d,c}(\mathbb{R)}$, we consider the sequence of best
approximation denominators $(Q_{n}(\theta))_{n\in\mathbb{N}}$,  their norms
$q_{n}=\left\Vert Q_{n}(\theta)\right\Vert _{\mathbb{R}^{c}}$, and the sequence
$(r_{n})_{n\geq 0}$ defined by
\[
r_{n}=d_{\mathbb{R}^{d}}(\theta Q_{n},\mathbb{Z}^{d}).
\]
For a nonnegative integer $k$, call $Bad_{k}$ the subset of $\M_{d,c}(\mathbb{R})$
defined by
\[
\Bad_{k}(d,c)=\Bad_{k}=\left\{\theta\in\M_{d,c}(\mathbb{R})\setminus\M_{d,c}(\mathbb{Q}):\inf_{n\in\N} q_{n+k}^{c}
r_{n}^{d}>0\right\}
\]
(if $r_n=0$ for some integer $n$, $\theta$ is not in $\Bad_k$). The sequence of sets $(\Bad_{k})_{k\geq 0}$ is clearly nondecreasing. The set
$\Bad_{0}$ is the usual set of badly approximable matrices, i.e., the set of matrices $\theta\in M_{d,c}(\theta)$ such that 
\[
\inf\left\{\frac{\|\theta Q-P\|_{\R^d}}{\|Q\|_{\R^c}}:Q\in\Z^c\setminus\{0\},P\in\Z^{d}\right\}>0.
\]
 When $d=c=1$, the
classical inequality $q_{n+1}r_{n}\geq\frac{1}{2}$ shows that $\Bad_{1}
=\mathbb{R}\setminus\mathbb{Q}$ while in \cite{Chev-Khin} it is shown
that for $c=1$ and $d\geq 2$, $\Bad_{1}$ is negligible. Our first goal is to
show that $\Bad_{1}\setminus \Bad_{0}$ is nonempty for $c=1$ and $d=2$. Next we
will prove that the set
\[
\mathcal{B}(d,c)=\mathcal{B}=\cup_{k\geq 0}\Bad_{k}
\]
is negligible and does not depend on the choice of the norm.

\begin{proposition}\label{prop:notbad}
If $c=1$ and $d=2$ then $\Bad_{1}\setminus \Bad_{0}$ contains uncountably many elements.
\end{proposition}

\begin{remark}The set $\mathbb{Z\theta+Z}^{2}$ is everywhere dense in
$\mathbb{R}^{2}$ for all in $\theta\in \Bad_{1}$. Indeed it is known that the
first minimum of the lattice
\[
\Lambda_{n}=\mathbb{Z}^{2}+\mathbb{Z}\frac{p_{n}}{q_{n}}
\]
is $\asymp r_{n-1}$ where $p_{n}$ is an integer vector such that $r_{n}
=\dd(q_{n}\theta,\mathbb{Z}^{2})$. This implies that the second minimum of this
lattice $\lambda_{2}(\Lambda_{n})$ is $\asymp\frac{1}{q_{n}r_{n-1}}$. Now, if $q_{n}r_{n-1}^{2}\geq\alpha>0$ then $\frac{1}{q_{n}r_{n-1}}
\leq\frac{r_{n-1}}{\alpha}$ which goes to zero when $n\rightarrow\infty$. Therefore, $\lambda_{2}(\Lambda_{n})$ goes to zero which  implies that
$\mathbb{Z\theta+Z}^{2}$ is everywhere dense in $\mathbb{R}^{2}$ (see
\cite{Chev-Mosc} or \cite{Cheu-Chev}).
\end{remark}

\begin{proof}
We assume that $\mathbb{R}^{2}$ is equipped with the standard Euclidean norm $\|.\|=\|.\|_{\mathbb{R}^2}$.
Set $\theta_{0}=(0,0)$ and $\theta_{1}=(\frac{1}{5},\frac{1}{5})$. We
construct inductively a sequence $(\theta_{n})_{n\geq 0}$ of rational vectors
in $\mathbb{R}^{2}$. For each $n$ in $\mathbb{N}$, let $\Lambda_{n}
=\mathbb{Z}^{2}+\theta_{n}\mathbb{Z}$ be the lattice associated with
$\theta_{n}$. Observe that the least common denominator $Q_{n}$ of the coordinates of the 
rational vector $\theta_{n}$ is the inverse of the volume of the lattice
$\Lambda_{n}$,\ $\det\Lambda_{n}=\frac{1}{Q_{n}}$ (even for $n=0$). For $1\leq i\leq n$, set
\begin{align*}
M_{i,n}  &  =\min\{\dd(q\theta_{n},\mathbb{Z}^{2})-\dd(Q_{i-1}\theta
_{n},\mathbb{Z}^{2}):Q_{i-1}<q<Q_{i}\},\\
m_{i,n}  &  =\dd(Q_{i-1}\theta_{n},\mathbb{Z}^{2})-\dd(Q_{i}\theta_{n}
,\mathbb{Z}^{2})
\end{align*}
The sequence $(\theta_{n})_{n\geq 0}$ is constructed such that the following
properties hold for all $n\geq 1$:
\begin{enumerate}
\item $Q_{0}=1<Q_{1}=5<Q_2<\dots<Q_{n}$ are the best approximations (denominators) of $\theta_{n}$,
\item $Q_{n}>2nQ_{n-1}$ and given $\theta_{0},\theta_{1},\,\dots,\theta_{n-1}$,
there are at least two possible choices of $\theta_{n}$ leading to two
different values of $Q_{n}$ (to ensure that we construct an uncountable set),
\item for all $1\leq i\leq n-1$, $M_{i,n}>0$ (we need to avoid the situation
where $\dd(q\theta_{n},\mathbb{Z}^{2})=\dd(Q_{i-1}\theta_{n},\mathbb{Z}^{2})$ for
some $q$ between $Q_{i-1}$ and $Q_{i}$),
\item $\left\Vert \theta_{n}-\theta_{n-1}\right\Vert \leq\frac{1}{8Q_{n-1}
}\min\{M_{i,j}:1\leq i<j\leq n-1\},$
\item $\left\Vert \theta_{n}-\theta_{n-1}\right\Vert \leq\frac{1}{8Q_{n-1}
}\min\{m_{i,j}:1\leq i\leq j\leq n-1\},$
\item $\varepsilon_{n-1}=Q_{n-1}(\theta_{n}-\theta_{n-1})$ is a shortest
vector of $\Lambda_{n}$, i.e. $\lambda_{1}(\Lambda_{n})=\left\Vert
\varepsilon_{n-1}\right\Vert $, and $(-1)^{n-1}\varepsilon_{n-1}$ has positive
coordinates,

\item $ 2\lambda_1(\Lambda_n)\leq\lambda_{2}(\Lambda_{n})\leq 30\lambda_{1}(\Lambda_{n})$.
\end{enumerate}

Observe that, with our choices of $\theta_0$ and $\theta_1$ all these conditions hold for $n=1$ (the conditions 3 and 4 are empty for $n=1$).\vspace{0.5cm}

First, let us show that the above conditions imply that the sequence
$(\theta_{n})_{n\in\mathbb{N}}$ converges to $\theta$ in $\Bad_{1}\setminus
\Bad_{0}$. By 2 and 4, the sequence $\left\Vert \theta_{n-1}-\theta
_{n}\right\Vert $ converges to $0$ at least at a geometric rate, hence the sequence
$(\theta_{n})_{n\geq 1}$ converge to $\theta\in\mathbb{R}^{2}$. Furthermore, by
4, for all $n\geq 2$,
\begin{align*}
\left\Vert \theta-\theta_{n}\right\Vert  &  \leq\sum_{p\geq n+1}\left\Vert
\theta_{p}-\theta_{p-1}\right\Vert \\
&  \leq\sum_{p\geq n+1}\frac{1}{8Q_{p-1}}\min\{M_{i,j}:1\leq i<j\leq p-1\}\\
&  \leq\frac{1}{4Q_{n}}\min\{M_{i,n}:1\leq i<n\}.
\end{align*}
Using 5 instead of 4, we obtain
\[
\left\Vert \theta-\theta_{n}\right\Vert \leq\frac{1}{4Q_{n}}\min
\{m_{i,n}:1\leq i\leq n\}
\]
as well. It follows that for\ all $1\leq i\leq n-1$ and all $Q_{i-1}<q<Q_{i}$,
we have
\begin{align*}
\dd(q\theta,\mathbb{Z}^{2})  &  \geq \dd(q\theta_{n},\mathbb{Z}^{2})-q\left\Vert
\theta-\theta_{n}\right\Vert \\
&  \geq \dd(Q_{i-1}\theta_{n},\mathbb{Z}^{2})+M_{i,n}-q\left\Vert \theta
-\theta_{n}\right\Vert \\
&  \geq \dd(Q_{i-1}\theta,\mathbb{Z}^{2})-Q_{i-1}\left\Vert \theta-\theta
_{n}\right\Vert +M_{i,n}-q\left\Vert \theta-\theta_{n}\right\Vert \\
&  \geq \dd(Q_{i-1}\theta,\mathbb{Z}^{2})+M_{i,n}-2Q_{i}\left\Vert \theta
-\theta_{n}\right\Vert.
\end{align*}
Since $\left\Vert \theta-\theta_{n}\right\Vert \leq\frac{1}{4Q_{n}}M_{i,n}$,
$\dd(q\theta,\mathbb{Z}^{2})>\dd(Q_{i-1}\theta,\mathbb{Z}^{2})$.
For all $1\leq i\leq n-1$, we also have
\begin{align*}
\dd(Q_{i}\theta,\mathbb{Z}^{2})  &  \leq \dd(Q_{i}\theta_{n},\mathbb{Z}^{2}
)+Q_{i}\left\Vert \theta-\theta_{n}\right\Vert \\
&  \leq \dd(Q_{i-1}\theta_{n},\mathbb{Z}^{2})-m_{i,n}+\frac{Q_{i}}{4Q_{n}
}m_{i,n}\\
&  \leq \dd(Q_{i-1}\theta,\mathbb{Z}^{2})+Q_{i-1}\left\Vert \theta-\theta
_{n}\right\Vert -m_{i,n}+\frac{Q_{i}}{4Q_{n}}m_{i,n}\\
&  \leq \dd(Q_{i-1}\theta,\mathbb{Z}^{2})+\frac{Q_{i-1}}{4Q_{n}}m_{i,n}
-m_{i,n}+\frac{Q_{i}}{4Q_{n}}m_{i,n}\\
&  <\dd(Q_{i-1}\theta,\mathbb{Z}^{2}).
\end{align*}
It follows that $Q_{0},Q_{1},...,Q_{n-1}$ are the first $n$ best
approximations of $\theta$. Therefore, $(Q_{n})_{n\geq 0}$ is the sequence of
best approximations of $\theta$ and  $\theta_n$ is the rational approximation of $\theta$ associated with $Q_n$. The standard inequality (see for instance
\cite{Chev-Mosc})

\[
\lambda_{1}(\Lambda_{n})\asymp \dd(Q_{n-1}\theta,\mathbb{Z}^{2})
\]
together with 7 imply that $\theta\in \Bad_{1}\setminus \Bad_{0}$.
\newline

Let $n$ be integer $\geq 1$. Let us  explain the construction $\theta_{n+1}$ given that $\theta
_{0},...,\theta_{n}$ are already constructed. First, choose a primitive point
$\alpha_{n}=k_{n}\theta_{n}+(a_{n},b_{n})$ of $\Lambda_{n}$ with $0\leq
k_{n}<Q_{n}$ and $(a_n,b_n)\in\Z^2$, in either $\mathbb{R}_{>0}^{2}$ when $n$ is even or in
$\mathbb{R}_{<0}^{2}$ when $n$ is odd. Just take $\alpha_n$
a point of $\Lambda_n$ in a square $[x,x+1[\times]0,1]$ with
minimal ordinate when $n$ is even and a point of $\Lambda_n$ in
a square $[x,x+1[\times[-1,0[$ with
maximal ordinate when $n$ is odd. Observe that $\left\Vert\alpha_n\right\Vert$ can be made arbitrarily large by choosing $|x|$ arbitrarily large.

Call $L_{n}=\left\Vert \alpha
_{n}\right\Vert $ the length of the segment $[0,\alpha_{n}]$. The (Euclidean) distance
between two consecutive lines of the set $\mathcal{H}_{n}=\Lambda
_{n}+\mathbb{R\alpha}_{n}$ is
\[
d_{n}=\frac{\det\Lambda_{n}}{L_{n}}=\frac{1}{Q_{n}L_{n}}.
\]
We can choose $\alpha_{n}$ such that
\[
L_{n}^{2}\geq n\det\Lambda_{n},
\]
hence $\frac{L_{n}}{d_{n}}=\frac{L_{n}^{2}}{\det\Lambda_{n}}\geq n$. There
are at least two integers $p_{n}\geq 2$ such that
\[
10\frac{L_{n}}{d_{n}}\leq p_{n}\leq 20\frac{L_{n}}{d_{n}}.
\]
Suppose $p_n$ is one of these and set
\[
\varepsilon_{n}=\frac{1}{p_{n}-\frac{k_{n}}{Q_{n}}}\alpha_{n},
\]
\[
\theta_{n+1}=\theta_{n}+\frac{\varepsilon_{n}}{Q_{n}},
\]
and
\[
Q_{n+1}=Q_{n}p_{n}-k_{n}.
\]
Since by 1, $Q_{0}=1<Q_{1}=5<...<Q_{j}$ are the best approximations of
$\theta_{j}$, $j=1,...,n$, the real number $\min\{m_{i,j}:1\leq i\leq j\leq
n\}$ is strictly positive. Moreover,
\[
\left\Vert \varepsilon_{n}\right\Vert \leq\frac{L_{n}}{p_{n}-1}\leq\frac
{L_{n}}{\frac{p_{n}}{2}}\leq 2\frac{L_{n}}{10\frac{L_{n}}{d_{n}}}\leq
\frac{d_{n}}{5}
\]
and $\alpha_{n}$ can be chosen in order that $d_{n}$ is arbitrarily small,
hence we can choose $\alpha_{n}$ such that $\left\Vert \varepsilon
_{n}\right\Vert <\left\Vert \varepsilon_{n-1}\right\Vert $ and
\[
\left\Vert \theta_{n+1}-\theta_{n}\right\Vert =\frac{1}{Q_{n}}\left\Vert
\varepsilon_{n}\right\Vert \leq\frac{1}{8Q_{n}}\min\{m_{i,j}:1\leq i\leq j\leq
n\}
\]
which is condition 5. Next by 3, $\min\{M_{i,j}:1\leq i<j\leq n\}$ is strictly
positive. As above, it follows that $\alpha_{n}$ can be chosen such that
\[
\left\Vert \theta_{n+1}-\theta_{n}\right\Vert =\frac{1}{Q_{n}}\left\Vert
\varepsilon_{n}\right\Vert \leq\frac{1}{8Q_{n}}\min\{M_{i,j}:1\leq i<j\leq
n\}
\]
which is condition 4.
Clearly $Q_{n+1}\geq Q_{n}(p_{n}-1)\geq 2nQ_{n}$.
Notice that the lattice $\Lambda_{n+1}=\mathbb{Z\theta}_{n+1}+\mathbb{Z}^{2}$
is included in $\mathcal{H}_{n}$. Next observe that
\begin{align*}
Q_{n+1}\theta_{n+1}  &  =(Q_{n}p_{n}-k_{n})(\theta_{n}+\frac{\varepsilon_{n}
}{Q_{n}})\\
&  =(Q_{n}p_{n}-k_{n})(\theta_{n}+\frac{\alpha_{n}}{Q_{n}(p_{n}-\frac{k_{n}
}{Q_{n}})})\\
&  =Q_{n}p_{n}\theta_{n}-k_{n}\theta_{n}+\alpha_{n}\\
&  =p_{n}Q_{n}\theta_{n}+(a_{n},b_{n})\in\mathbb{Z}^{2}.
\end{align*}
It follows that $Q_{n+1}\det\Lambda_{n+1}=l\in\mathbb{N}$. On the other hand,
consider the one dimensional lattice $\Lambda_{n+1}\cap\mathbb{R\alpha}_{n}$.
Because $Q_{n}\theta_{n}\in\mathbb{Z}^{2}$ and $Q_{n}\theta_{n+1}=Q_{n}
\theta_{n}+\varepsilon_{n}$, it contains $\varepsilon_{n}$ and is generated by a
vector $v_{n}=\frac{\varepsilon_{n}}{m}$ where $m$ is an integer. Next observe that $\theta_n\in\Lambda_{n+1}+\R\alpha_n$, hence $\Lambda_{n+1}+\R\alpha_n=\mathcal H_n$. It follows that
\begin{align*}
\frac{l}{Q_{n+1}}  &  =\det\Lambda_{n+1}=\left\Vert v_{n}\right\Vert d_{n}\\
&  =\frac{\left\Vert \varepsilon_{n}\right\Vert }{m}d_{n}=\frac{Q_{n}L_{n}
}{mQ_{n+1}}d_{n}\\
&  =\frac{1}{mQ_{n+1}},
\end{align*}
which implies $m=l=1$. Therefore, $\det\Lambda_{n+1}=\frac{1}{Q_{n+1}}$ and
\[
\Lambda_{n+1}=\{0,...,Q_{n}-1\}\theta_{n+1}+\mathbb{Z\varepsilon}
_{n}+\mathbb{Z}^{2}.
\]
Since $\left\Vert \varepsilon_{n}\right\Vert \leq\frac{d_{n}}{5}$, and
$\Lambda_{n+1}\subset\mathcal{H}_{n}$, $\varepsilon_{n}$ is the shortest
vector of $\Lambda_{n+1}$. The choice of the signs for $\alpha_n$ now implies  that condition 6 holds. Next
\begin{align*}
\lambda_{1}(\Lambda_{n+1})  &  =\left\Vert \varepsilon_{n}\right\Vert ,\\
5\left\Vert \varepsilon_{n}\right\Vert  &  \leq d_{n}\leq\lambda_{2}
(\Lambda_{n+1})\leq d_{n}+\left\Vert \varepsilon_{n}\right\Vert .
\end{align*}
Since
\[
\left\Vert \varepsilon_{n}\right\Vert \geq\frac{L_{n}}{p_{n}}\geq\frac{L_{n}
}{20\frac{L_{n}}{d_{n}}}=\frac{d_{n}}{20},
\]
we obtain
\[
5\lambda_{1}(\Lambda_{n+1})\leq\lambda_{2}(\Lambda_{n+1})\leq 21\left\Vert
\varepsilon_{n}\right\Vert \leq 30\lambda_{1}(\Lambda_{n+1})
\]
which implies condition 7.
Let us show that $Q_{0},...,Q_{n-1}$ are the first best approximations of
$\theta_{n+1}$.\newline For\ all $1\leq i\leq n-1$ and all $Q_{i-1}<q<Q_{i}$,
we have
\begin{align*}
\dd(q\theta_{n+1},\mathbb{Z}^{2})  &  \geq \dd(q\theta_{n},\mathbb{Z}
^{2})-q\left\Vert \theta_{n+1}-\theta_{n}\right\Vert \\
&  \geq \dd(Q_{i-1}\theta_{n},\mathbb{Z}^{2})+M_{i,n}-q\left\Vert \theta
_{n+1}-\theta_{n}\right\Vert \\
&  \geq \dd(Q_{i-1}\theta_{n+1},\mathbb{Z}^{2})-Q_{i-1}\left\Vert \theta
_{n+1}-\theta_{n}\right\Vert +M_{i,n}-q\left\Vert \theta_{n+1}-\theta
_{n}\right\Vert \\
&  \geq \dd(Q_{i-1}\theta_{n+1},\mathbb{Z}^{2})+M_{i,n}-2Q_{i}\left\Vert
\theta_{n+1}-\theta_{n}\right\Vert .
\end{align*}
Since $\left\Vert \theta_{n+1}-\theta_{n}\right\Vert \leq\frac{1}{8Q_{n}
}M_{i,n}$, $\dd(q\theta_{n+1},\mathbb{Z}^{2})>\dd(Q_{i-1}\theta_{n+1}
,\mathbb{Z}^{2})$ and hence $M_{i,n+1}>0$. We also have
\begin{align*}
\dd(Q_{i}\theta_{n+1},\mathbb{Z}^{2})  &  \leq \dd(Q_{i}\theta_{n},\mathbb{Z}
^{2})+Q_{i}\left\Vert \theta_{n+1}-\theta_{n}\right\Vert \\
&  \leq \dd(Q_{i-1}\theta_{n},\mathbb{Z}^{2})-m_{i,n}+\frac{Q_{i}}{8Q_{n}
}m_{i,n}\\
&  \leq \dd(Q_{i-1}\theta_{n+1},\mathbb{Z}^{2})+Q_{i-1}\left\Vert \theta
_{n+1}-\theta_{n}\right\Vert -m_{i,n}+\frac{Q_{i}}{8Q_{n}}m_{i,n}\\
&  \leq \dd(Q_{i-1}\theta_{n+1},\mathbb{Z}^{2})+\frac{Q_{i-1}}{8Q_{n}}
m_{i,n}-m_{i,n}+\frac{Q_{i}}{8Q_{n}}m_{i,n}\\
&  <\dd(Q_{i-1}\theta_{n+1},\mathbb{Z}^{2}).
\end{align*}
It follows that $Q_{0},Q_{1},...,Q_{n-1}$ are the first $n$ best
approximations of $\theta_{n+1}$. The proof will be done once we will have
explained that $Q_{n}$ and $Q_{n+1}$ are the only best approximations that
follow $Q_{n-1}$ and that $M_{n,n+1}>0$. These are the places where the sign
condition 6 plays a role. First observe that $\varepsilon_{n}$ and
$-\varepsilon_{n}$ are the only two shortest vectors of $\Lambda_{n+1}$ and
that
\[
\varepsilon_{n}=Q_{n}\theta_{n+1}-Q_{n}\theta_{n}
\]
and
\begin{align*}
-\varepsilon_{n}&=(Q_{n+1}-Q_{n})\theta_{n+1}-Q_{n+1}\theta_{n+1}+Q_{n}\theta_{n}\\
&=(Q_{n+1}-Q_{n})\theta_{n+1}+\text{ a vector in }\Z^2
.
\end{align*}
Together with the inequality $Q_{n+1}-Q_{n}>Q_{n}$ this implies that $Q_{n}$ is a best approximation
of $\theta_{n+1}$ and that there is no best approximation of $\theta_{n+1}$
between $Q_{n}$ and $Q_{n+1}$.
Next, denoting by $\equiv$ the equivalence $\operatorname{mod}\mathbb{Z}^{2}$,
we have
\begin{align*}
Q_{n-1}\theta_{n+1}  &  =Q_{n-1}(\theta_{n}+\frac{\varepsilon_{n}}{Q_{n}
})=Q_{n-1}(\theta_{n-1}+\frac{\varepsilon_{n-1}}{Q_{n-1}}+\frac{\varepsilon
_{n}}{Q_{n}})\\
&  \equiv\varepsilon_{n-1}+\frac{Q_{n-1}}{Q_{n}}\varepsilon_{n}
\end{align*}
and
\begin{align*}
(Q_{n}-Q_{n-1})\theta_{n+1}  &  =(Q_{n}-Q_{n-1})(\theta_{n}+\frac
{\varepsilon_{n}}{Q_{n}})\\
&  \equiv-Q_{n-1}\theta_{n}+(1-\frac{Q_{n-1}}{Q_{n}})\varepsilon_{n}\\
&  =-Q_{n-1}(\theta_{n-1}+\frac{\varepsilon_{n-1}}{Q_{n-1}})+(1-\frac{Q_{n-1}
}{Q_{n}})\varepsilon_{n}\\
&  \equiv-\varepsilon_{n-1}+(1-\frac{Q_{n-1}}{Q_{n}})\varepsilon_{n},
\end{align*}
by the choice of the signs we obtain that
\[
\dd(Q_{n-1}\theta_{n+1},\mathbb{Z}^{2})
<\dd((Q_{n}-Q_{n-1})\theta_{n+1},\mathbb{Z}^{2}).
\]
If $q$ is an integer $\neq Q_{n}-Q_{n-1}$ lying in $]Q_{n-1},Q_{n}[$, then
$q\theta_{n}$ cannot be $\equiv\pm\varepsilon_{n-1}$ which are the shortest vectors of $\Lambda_n$. Hence
\[
\dd(q\theta_{n},\mathbb{Z}^{2})\geq\min(2\left\Vert \varepsilon_{n-1}\right\Vert
,\lambda_{2}(\Lambda_{n}))=2\left\Vert \varepsilon_{n-1}\right\Vert .
\]
It follows that
\begin{align*}
\dd(q\theta_{n+1},\mathbb{Z}^{2})  &  \geq \dd(q\theta_{n},\mathbb{Z}
^{2})-q\left\Vert \theta_{n+1}-\theta_{n}\right\Vert \\
&  \geq 2\left\Vert \varepsilon_{n-1}\right\Vert -\frac{q}{Q_{n}}\left\Vert
\varepsilon_{n}\right\Vert \\
&  \geq 2\left\Vert \varepsilon_{n-1}\right\Vert -\left\Vert \varepsilon
_{n}\right\Vert >\left\Vert \varepsilon_{n-1}\right\Vert \\
&  \geq\|\varepsilon_{n-1}+\frac{Q_{n-1}}{Q_n}\varepsilon_n\|=\dd(Q_{n-1}\theta_{n+1},\mathbb{Z}^{2})
\end{align*}
which implies both that $Q_{n-1}$ and $Q_{n}$ are consecutive best
approximations of $\theta_{n+1}$ and that $M_{n,n+1}>0$.
\end{proof}

It is not clear whether the set $\Bad_{1}$ depends on the norm. However, using
an easy result about the relations between best approximation vectors
associated with two norms, we can prove:

\begin{proposition}\label{prop:bad}
The set $\mathcal{B}(d,c)$ does not depend on the norms.
\end{proposition}

\begin{proof}
We give the proof only in the case $c=1$. When $c>1$ one has to extend first
the following result about best approximations:

Consider two norms $N$ and
$N^{\prime}$ on $\mathbb{R}^{d}$. For $\theta \in \mathbb R^d$, call $(q_{n})_{n\in\N}$  the sequence of best approximation denominators
associated with the norm $N$ and $(q_{n}^{\prime})_{n\in\N}$  the
sequence associated with the norm $N^{\prime}$. Then (see \cite{Chev-Mosc}
)  there exists an integer $k$ depending only on the norms $N$ and
$N^{\prime}$ such that each interval $]q_{n},q_{n+k}]$, contains a best
approximation denominator $q^{\prime}_m$ associated with the norm $N^{\prime}$.

It is enough to prove that $\mathbb{R}^{d}\setminus \Bad_{kp}\subset
\mathbb{R}^{d}\setminus \Bad_{p}^{\prime}$ for all $p$. Let $\theta$ be in
$\mathbb{R}^{d}$ and $n\geq k$ be an integer. By the above result, there
exist at least one best approximation denominator in each interval $]q_{n+(j-1)k}
,q_{n+jk}]$, $j=0,...,p$. Let $q_{n_{j}}^{\prime}$ be the largest best
approximation denominator in each of these intervals $]q_{n+(j-1)k},q_{n+jk}]$. For each
$j$, we have
\[
r_{n_{j}}^{\prime}\leq Cr_{n+jk}
\]
where $C$ is the constant involved in the norm equivalence. Making use of the
above inequality with $j=0$, we obtain $q_{n_{0}+p}^{\prime}r_{n_{0}}^{\prime
d}\leq C^{d}q_{n_{0}+p}^{\prime}r_{n}^{d}$. Next $q_{n_{0}+p}^{\prime}\leq
q_{n+kp}$, hence,
\[
q_{n_{0}+p}^{\prime}r_{n_{0}}^{\prime d}\leq C^{d}q_{n+kp}r_{n}^{d}.
\]
It follows that $\lim\inf_{n\rightarrow\infty}q_{n+kp}r_{n}^{d}=0$ implies
$\lim\inf_{n\rightarrow\infty}q_{n+p}^{\prime}r_{n}^{\prime d}=0$.
\end{proof}

\begin{theorem}
\label{thm:badk} The set $\mathcal{B}(d,c)=\cup_{k\geq 0}\Bad_{k}$ has zero Lebesgue measure.
\end{theorem}

By the above proposition, $\mathcal{B}(d,c)$ doesn't depend on the norms and we can suppose that $\mathbb{R}^{d}$
and $\mathbb{R}^{c}$ are equipped with the standard Euclidean norms. Let us
show that for each $k$, $\Bad_{k}$ has zero measure. We need two lemmas.

\begin{lemma}
\label{lem:perturbation} 1.
Let $a<b$  be two integers, let
$\Lambda=M\mathbb{Z}^{d+c}$ be a lattice in $\mathcal L_{d+c}$ and let $(Y_{n})_{n=a,...,b}$
be a sequence of vectors in $\mathbb{Z}^{d+c}$. Suppose that for $n=a,...,b$,
\begin{itemize}
\item $X_{n}(\Lambda)=MY_{n}$,
\item the only nonzero points of
$\Lambda$ in the cylinder $C(X_{n}(\Lambda),X_{n+1}(\Lambda))$ are $\pm
X_{n}(\Lambda)$ and $\pm X_{n+1}(\Lambda).$
\end{itemize}
Then there exists a open
neighborhood $W$ of $M$ in $\SL(d+c,\R)$ such that for all lattices $\Lambda^{\prime}
=M^{\prime}\mathbb{Z}^{d+c}\in \mathcal L_{d+c}$ with
$M^{\prime}$ in $W$, the vectors $Z_n=M'Y_n$ are consecutive
minimal vectors of $\Lambda'$ and
\begin{align*}
\left\vert Z_n\right\vert_+  &  \in[\frac{1}{2}r_{n}(\Lambda),2r_{n}(\Lambda)],\\
\left\vert Z_n\right\vert_- &  \in[\frac{1}{2}q_{n}(\Lambda),2q_{n}(\Lambda)]
\end{align*}
for $n=a,...,b$.\newline
2. Suppose furthermore that $a<0$, $b>1$ and $\Lambda\in S$.
Then for all lattices $\Lambda^{\prime}
=M^{\prime}\mathbb{Z}^{d+c}\in S$ with $M^{\prime}$ in $W$, we have $X_{n}(\Lambda^{\prime})  =M^{\prime}Y_{n}$ for $n=a,...,b$..
\end{lemma}

\begin{proof}
1. Consider a ball $B_{d,c}(0,R)$ that contains all the points $MY_{n}$,
$n=a,...,b$. There is a neighborhood $\omega$ of the identity matrix $I_{d+c}$ such
that for all $A$ in $\omega$ \ and all $X$ in $\mathbb{R}^{d+c}$,
\[
\frac{1}{2}\left\Vert X\right\Vert _{d,c}\leq\left\Vert
AX\right\Vert _{d,c}\leq 2\left\Vert X\right\Vert _{d,c},
\]
so that
\begin{align*}
X  &  \notin B_{d,c}(0,8R)\Rightarrow AX\notin B_{d,c}(0,4R)\\
X  &  \in B_{d,c}(0,R)\Rightarrow AX\in B_{d,c
}(0,2R).
\end{align*}
Two vectors $Z_{n}=AMY_{n}$ and $Z_{n+1}=AMY_{n+1}$ are consecutive
minimal vectors of $A\Lambda$ as soon as
\[
\left\vert Z_{n+1}\right\vert _{-}>\left\vert Z_{n}\right\vert _{-}
,\ \left\vert Z_{n+1}\right\vert _{+}<\left\vert Z_{n}\right\vert _{+}
\]
and the cylinder $C(Z_{n},Z_{n+1})$ contains no other nonzero vector of
$A\Lambda$ than $\pm Z_{n}$ and $\pm Z_{n+1}$. Since, $\left\vert
X_{n+1}\right\vert _{-}>\left\vert X_{n}\right\vert _{-},\ \left\vert
X_{n+1}\right\vert _{+}<\left\vert X_{n}\right\vert _{+}$, by reducing
$\omega$, we can assume $\left\vert Z_{n+1}\right\vert _{-}>\left\vert
Z_{n}\right\vert _{-}\ $\ and $\left\vert Z_{n+1}\right\vert _{+}<\left\vert
Z_{n}\right\vert _{+}$, $n=a,\dots,b$. Since $C(Z_{n},Z_{n+1})=C(AX_{n},AX_{n+1})\subset
B_{d,c}(0,2R)$, the image by $A$ of a vector of $\Lambda$ that is
not in the ball $B_{d,c}(0,8R)$, cannot enter in the cylinder
$C(AX_{n},AX_{n+1})$. Therefore, there are only finitely many $X$ in $\Lambda$
such that $AX$ is in $C(AX_{n},AX_{n+1})$. Since by assumption all these
vectors $X$, except $\pm X_n$ and $\pm X_{n+1}$, are at a positive distance from $C(X_{n},X_{n+1})$, we obtain that
$Z_{n}$ and $Z_{n+1}$ are consecutive minimal vectors by reducing once again
$\omega$. It follows that $Z_a,...,Z_b$ are consecutive minimal vectors of the lattice $A\Lambda$. A new reduction of $\omega$ ensures that the two  inequalities of the
lemma hold.

2. We want to see that there is no shift on the indices when $\Lambda\in S$, $a<0$ and $b>1$. By the numbering 
convention (see Section \ref{sec:visit}),
\[
\left\vert X_{0}(\Lambda)\right\vert _{+}=\left\vert X_{1}(\Lambda)\right\vert
_{-},\ \left\vert X_{-1}(\Lambda)\right\vert _{+}>\left\vert X_{0}
(\Lambda)\right\vert _{-},\ \left\vert X_{1}(\Lambda)\right\vert
_{+}<\left\vert X_{2}(\Lambda)\right\vert _{-}.
\]
By a further reduction of $\omega$, we can assume that the two inequalities
hold for the vectors $Z_{-1}=AMY_{-1}$, $Z_{0}=AMY_{0}$, $Z_{1}=AMY_{1}$ and $Z_{2}=AMY_{2}$.
Therefore, if $AM\mathbb{Z}^{d+c}$ is in $S$, we must have
\[
X_{0}(A\Lambda)=AMY_{0}\text{ and }X_{1}(A\Lambda)=AMY_{1}
\]
which implies that $X_{n}(A\Lambda)=Z_{n}$ for $n=a,....,b$.
\end{proof}

\begin{lemma}
\label{lem:smallvectors}Assume that $d+c\geq3$. Let $\Gamma$ be a two dimensional lattice in
$\mathcal{L}_{2}\setminus\mathcal{N}_{1,1}$ which is in the transversal $S_{1,1}\subset \mathcal L_2$ and let $k$ be a
non negative integer. Then for all positive real numbers $\delta$, there
exists a lattice $\Gamma_{\delta}$ in
$S\setminus\mathcal{N}$ such that
\begin{align*}
r_{n}(\Gamma_{\delta})  &  \leq2\delta r_{n}(\Gamma)\\
q_{n}(\Gamma_{\delta})  &  \leq 2\delta q_{n}(\Gamma)
\end{align*}
for $n=0,...,k$.
\end{lemma}

\begin{proof}
Let $\Gamma=A\mathbb{Z}^{2}$ be a lattice in $S_{1,1}\setminus\mathcal{N}_{1,1}$
where
\[
A=\left(
\begin{array}
[c]{cc}
a_{11} & a_{12}\\
a_{21} & a_{22}
\end{array}
\right)  .
\]
Consider the matrix $M_{\delta}\in\SL(d+c,\R)$ defined by
\[
M_{\delta}=\left(
\begin{array}
[c]{cccccc}
\delta a_{11} & 0 & 0 & \ldots & 0 & \delta a_{12}\\
0 & \delta^{-\frac{2}{d+c-2}} & 0 & \ldots & \ldots & 0\\
0 & 0 & \delta^{-\frac{2}{d+c-2}} & 0 & \ldots & 0\\
\vdots & \vdots & \ddots~~ & \ddots~~~~ & \ddots & \vdots\\
0 & 0 & \ldots & 0 & \delta^{-\frac{2}{d+c-2}} & 0\\
\delta a_{21} & 0 & \ldots& \ldots & 0 & \delta a_{22}
\end{array}
\right)  .
\]
\qquad
Let $(U_{n}=(u_{n,1},u_{n,2}))_{n\in\mathbb{Z}}$ be the sequence of vectors in $\mathbb{Z}^{2}$
such that $(X_{n}(\Gamma)=AU_{n})_{n\in\Z}$ is the sequence of minimal vectors of
$\Gamma$. For each $n\in\mathbb{Z}$, let $Y_{n}$ be the element of
$\mathbb{Z}^{d+c}$ defined by $y_{1}=u_{n,1},\ y_{2}=...=y_{d+c-1}=0$ and
$y_{d+c}=u_{n,2}$. If $\mathbb{\delta}>0$ is small enough, then for all
$Z\in\mathbb{Z}^{d+c}$ not in the $\R e_{1}+\R e_{d+c}$-plane, we have
\[
\left\Vert M_{\delta}Z\right\Vert _{d,c}=\max(\left\vert
M_{\delta}Z\right\vert _{+},\left\vert M_{\delta}Z\right\vert _{-})\geq \delta^{-\frac{2}{d+c-2}}\geq
\max(2\mathbb{\delta}r_{-1}(\Gamma),2\mathbb{\delta}q_{k+1}(\Gamma)).
\]
It follows that none of these vectors $M_{\mathbb{\delta}}Z$ are in one of the
cylinders $C(M_{\mathbb{\delta}}Y_{n},M_{\delta}Y_{n+1})$, $n=-1,...,k$.
Therefore the vectors $X_{n}=M_{\mathbb{\delta}}Y_{n}$, $n=-1,...,k+1$ are all
consecutive minimal vectors of $\Lambda_{\mathbb{\delta}}=M_{\mathbb{\delta}
}\mathbb{Z}^{d+1}$ and $\Lambda_{\delta}$ is in $S$. With the numbering
convention, we have $X_{n}(\Lambda_{\mathbb{\delta}})=X_{n}$ for all
$n=-1,...,k+1$.
Now we fix $\delta$ small enough. By the previous lemma applied to
$\Lambda_{\delta}$, there is a sequence of matrices $(M_{p})_{p}$ in
$S\setminus\mathcal{N}$ which converges to $M_{\mathbb{\delta}}$ and such that for all
$p$,
\[
X_{n}(M_{p}\mathbb{Z}^{d+1})=M_{p}Y_{n}
\]
$n=-1,...,k$. When $p$ goes to infinity,
\begin{align*}
r_{n}(M_{p}\mathbb{Z}^{d+1})  &  =\left\vert M_{p}Z_{n}\right\vert
_{+}\rightarrow\left\vert M_{\delta}Z_{n}\right\vert _{+}=\delta r_{n}
(\Gamma),\\
q_{n}(M_{p}\mathbb{Z}^{d+1})  &  =\left\vert M_{p}Z_{n}\right\vert
_{-}\rightarrow\left\vert M_{\delta}Z_{n}\right\vert _{-}=\delta q_{n}(\Gamma)
\end{align*}
for $n=-1,...,k$. So we can take $\Gamma_{\delta}=M_{p}\mathbb{Z}^{d+1}$
for some $p$ large enough.
\end{proof}

\begin{proof}[End of proof of  Theorem \ref{thm:badk}]
Let $k$ and $\eta>0$ be fixed. We want to prove that the set of $\theta$ in
$\M_{d,c}(\mathbb{R})$ such that
\[
\lim\inf_{n\rightarrow\infty}q_{n+k}^{c}(\theta)r_{n}^{d}(\theta)\leq\eta
\]
has full measure.
By Lemma \ref{lem:min-best}, it is enough to show that
\[
\lim\inf_{n\rightarrow\infty}q_{n+k}^{c}(\Lambda_{\theta})r_{n}^{d}(\Lambda
_{\theta})\leq\eta
\]
for almost all $\theta$.
Fix a two-dimensional lattice $\Gamma$ in $S_{1,1}\setminus \mathcal{N}_{1,1}$
and let $\delta$ be a positive real number with $\delta\leq\frac{1}{4}
(\frac{\eta}{q_{k}^{c}(\Gamma)r_{0}^{d}(\Gamma)})^{\frac{1}{d+c}}$. By 
Lemma \ref{lem:smallvectors}, there exists a
lattice $\Lambda_{\delta}$  in $S\setminus\mathcal{N}$ such that
\begin{align*}
r_{n}(\Lambda_{\delta}) &  \leq2\delta r_{n}(\Gamma)\\
q_{n}(\Lambda_{\delta}) &  \leq 2\delta q_{n}(\Gamma)
\end{align*}
for $n=0,...,k$. Hence,
\[
q_{k}^{c}(\Lambda_{\delta})r_{0}^{d}(\Lambda_{\delta})\leq\frac
{\eta}{2^{d+c}}.
\]
By Lemma \ref{lem:perturbation}, there exists an open neighborhood $W$ of
$\Lambda_{\delta}$ such that for all $\Lambda$ in $W$ and some integer $m(\Lambda)$, we have both
\[
r_{m(\Lambda)}(\Lambda)\leq 4\delta r_0(\Gamma)
\]
and 
\[
q_{m(\Lambda)+k}^{c}(\Lambda)r_{m(\Lambda)}^{d}(\Lambda)\leq\eta.
\]

Let us show that, if for a given lattice $\Lambda$, there exists a sequence $(t_{n}
)_{n\in\mathbb{N}}$ going to infinity such that $g_{t_{n}}\Lambda\in$ $W$ for
all $n\in\mathbb{N}$, then
\[
\lim\inf_{n\rightarrow\infty}q_{n+k}^{c}(\Lambda)r_{n}^{d}(\Lambda)\leq\eta.
\]
Indeed, if $g_{t_n}\Lambda \in W$, then for some integer $m(\Lambda,t_n)$, we have 
\[
(e^{-dt_n}q_{m(\Lambda,t_n)+k}(\Lambda))^{c}(e^{ct_n}r_{m(\Lambda,t_n)}(\Lambda))^{d}\leq\eta.
\] 
So, the only thing to see is that $m(\Lambda,t_n)\rightarrow\infty$ when $n\rightarrow\infty$. 
Now $e^{dt_n}r_{m(\Lambda,t_n)}(\Lambda)\leq 4\delta r_0(\Gamma) $,
hence $r_{m(\Lambda,t_n)}(\Lambda)$ goes to zero when $m$ goes
to infinity which implies that $m(\Lambda,t_n)$ goes to infinity.

Making use of the Birkhoff ergodic theorem with the flow $g_t$, the proof
would be already finished if our goal were
$\lim\inf_{n\rightarrow\infty}q_{n+k}^{c}(\Lambda)r_{n}^{d}(\Lambda)\leq\eta$
for almost all lattices. However we want an ``almost all" with respect of the Lebesgue measure of $\M_{d,c}(\R)$.

Let $U$  be a relatively compact nonempty open set in $\mathcal L_{d+c}$ such that
$\overline{U}\subset W$. 
One can find
a neighborhood $V$ of $I_{d+c}$ in $\mathcal{H}_{\leq}$ such
that for all
$\theta\in \M_{d,c}(\mathbb{R)}$, all $t\geq 0$ and all $h\in V$, we have
\[
g_{t}h\Lambda_{\theta}=(g_{t}hg_{-t})g_{t}\Lambda_{\theta}\in U\Longrightarrow
g_{t}\Lambda_{\theta}\in W.
\]
Call $\mathcal{V}$ the set of $\theta$ such that $g_{t}\Lambda_{\theta}\notin
W$ for all $t$ large enough. By the choices of $U$ and $V$, for all $h\in V$ and all
$\theta\in\mathcal{V}$, $g_{t}h\Lambda_{\theta}\notin U$ for all $t$ large
enough. If the Lebesgue measure of $\mathcal{V}$ were nonzero then the set of
lattices of the form $g_{s}h\Lambda_{\theta}$ with $s\in[0,1]$, $h\in
V$ and $\theta\in\mathcal{V}$, would have a nonzero measure. Now, by the Birkhoff ergodic theorem,
for almost all lattices $\Lambda$, there exists a sequence $t_{n}
\rightarrow\infty$ such that $g_{t_{n}}\Lambda\in U$ for all $n$, therefore
$\mathcal{V}$ has zero measure.
\end{proof}

\begin{proof}[Proof of Theorem \ref{thm:jager}. 2] By the proof of the first part of Theorem \ref{thm:jager}, we know that the measure $\nu_{d,c}$ is the image of the measure $\frac{1}{\mu_S(S)}\mu_S$ by the map $F:S\rightarrow \R$ defined by 
	\[
	F(\Lambda)=q_1^c(\Lambda)r_0^d(\Lambda)=|v_1^S(\Lambda)|_-^c|v_0^S(\Lambda)|^d_+.
	\]
We want to prove that the support of the measure $\nu_{d,c}$ contains zero, {\it i.e.}, that $\nu_{d,c}([0,\eta])>0$ for all $\eta>0$. By the Birkhoff ergodic theorem and by the definition of $\nu_{d,c}$, it is enough to prove that 
\[
\lim_{n\rightarrow\infty}\frac{1}{n}\sum_{i=0}^{n-1}1_{[0,\eta]}(F(R^i(\Lambda)))>0
\]
for almost all $\Lambda\in S$. 
By Lemmas \ref{lem:perturbation} and \ref{lem:smallvectors},
there exists a non empty  open set $W$ in $S$ such that 
\[
\left\vert v_1^S(\Lambda)\right\vert^c_-\left\vert v_0^S(\Lambda)\right\vert^d_+\leq\eta
\]
for all $\Lambda\in W$. Hence $1_W\leq 1_{[0,\eta]}\circ F$.
 By the Birkhoff ergodic theorem,
 for almost
all $\Lambda$ in $S$
\begin{align*}
\lim_{n\rightarrow\infty}\frac{1}{n}\sum_{i=1}^{n}1_W\circ R^{i}
(\Lambda)=\frac{1}{\mu_{S}(S)}\int_{S}1_W~d\mu_{S}=a>0.
\end{align*}
therefore,
\[
\lim_{n\rightarrow\infty}\frac{1}{n}\sum_{i=0}^{n-1}1_{[0,\eta]}(F(R^i(\Lambda)))\geq a>0.
\]
\end{proof}

\section{Miscellaneous Questions}\label{sec:question}
\begin{enumerate}
\item 

In Theorems \ref{thm:levy-d} and \ref{thm:jager}, we assume that $\R^d$ and $\R^c$ are equipped with standard Euclidean norms. Clearly, Theorems 1 and 2 are valid for any pair of Euclidean norms and by  Theorem \ref{thm:Birkhoff}, the Levy constant does not depend on the particular choices of the Euclidean norms. However, we do not know what happens for a pair of general norms. Are these theorems valid when $\R^d$ and $\R^c$ are equipped with any norm?
If Theorem \ref{thm:levy-d} is valid for any norm, does the Levy's constant depend on the norm?

\item Is the measure $\nu_{d,c}$ in Theorem \ref{thm:jager}, absolutely continuous with respect to Lebesgue measure?\\
Is the support of $\nu_{d,c}$ an interval? Compute the measure $\nu_{2,1}$.

\item Suppose $c=1$. 
Consider a flow $(g_t)_{t\in\R}$ defined by the matrices
\[
g_{t}=\operatorname{diag}(e^{a_1t},\dots,e^{a_dt},e^{-dt}) \in \SL (d+1,\mathbb R)
\]
where the $a_is$ are positive real numbers with sum $d$.
Best approximation vectors of $\theta\in\R^d$ with respect to the flow $(g_t)_{t\in\R}$  can be defined as follows. A nonzero vector $X$ in $\Z^{d+1}$ is a best approximation vector of $\theta$ if there exists $t\geq0$ such that the interior of the ball $B(0,\|g_tM_{\theta}X\|_{\R^{d+1}})\subset \R^d\times\R$ contains no nonzero vector  of the lattice $g_tM_\theta\Z^{d+1}$ (equivalently $\|g_tM_{\theta}X\|_{\R^{d+1}}=\lambda_1(g_tM_{\theta}\Z^{d+1}$)). Arranging the set of best approximation vector according to their heights, we obtain a sequence $(X_n(\theta))_{n\in\N}$ of best approximation vectors associated with $\theta$. Does Theorem \ref{thm:levy-d} hold for these new sequences of best approximation vectors?
\item 
For a fixed $k\geq 1$, does the set
\[
\Bad_{k}(d,c)=\Bad_{k}=\{\theta\in\M_{d,c}(\mathbb{R})\setminus\M_{d,c}(\mathbb{Q}):\inf_{n\in\N} q_{n+k}^{c}(\theta)
r_{n}^{d}(\theta)>0\}
\]
depend on the norm used to define best approximations vectors?\\
Observe that by Proposition \ref{prop:bad}, the union $\cup_{k\geq 1}\Bad_k$ does not depend on the choice  of the norm.
\item Since  $\Bad_1(d,c)$ contains the set of badly approximable matrices $\Bad_0(d,c)$, it has full Hausdorff dimension. But what are the Hausdorff dimensions of  $\Bad_1(d,c)\setminus \Bad_0(d,c)$ and of $\cup_{k\geq 0}\Bad_k(d,c)\setminus \Bad_0(d,c)$?
Has $\Bad_1(d,c)\setminus \Bad_0(d,c)$  full Hausdorff dimension? 
\end{enumerate}

\end{document}